\documentclass[a4paper]{siamart190516}

\usepackage{damacros}
\usepackage{enumitem}
\usepackage{tikz-cd} 
\usepackage{mathtools} 

\tikzset{
  symbol/.style={
    draw=none,
    every to/.append style={
      edge node={node [sloped, allow upside down, auto=false]{$#1$}}}
  }
}

\newcommand{\sgn}{\text{sign}}

\newcommand{\Lip}{\text{Lip}}

\headers{Analytic theory for fast-slow systems on Banach space}{D. Doorakkers}

\title{A functional analytic theory for differential equations on Banach spaces with slowly evolving parameters}

\author{%
  Dirk Doorakkers
  \thanks{%
    Vrije Universiteit Amsterdam,
    Department of Mathematics,
    Faculteit der Exacte Wetenschappen,
    De Boelelaan 1081a,
    1081 HV Amsterdam, The Netherlands.
  \protect\\
  (\email{h.a.doorakkers@vu.nl}).
  }
  \and
  Daniele Avitabile%
  \thanks{%
    Vrije Universiteit Amsterdam,
    Department of Mathematics,
    Faculteit der Exacte Wetenschappen,
    De Boelelaan 1081a,
    1081 HV Amsterdam, The Netherlands.
  \protect\\
    Inria Sophia Antipolis M\'editerran\'ee Research Centre,
    MathNeuro Team,
    2004 route des Lucioles-Boîte Postale 93 06902,
    Sophia Antipolis, Cedex, France.
  \protect\\
    (\email{d.avitabile@vu.nl}, \url{www.danieleavitabile.com}).
  }
  \and
  Jan Bouwe van den Berg%
  \thanks{%
    Vrije Universiteit Amsterdam,
    Department of Mathematics,
    Faculteit der Exacte Wetenschappen,
    De Boelelaan 1081a,
    1081 HV Amsterdam, The Netherlands.
  \protect\\
  (\email{janbouwe.vanden.berg@vu.nl}).
  }}

\begin{document}

\maketitle
\begin{abstract}
This paper provides a functional analytic approach to differential equations on Banach spaces with slowly evolving parameters. We develop a Fenichel-like 
theory for attracting subsets of critical manifolds via a Lyapunov-Perron method. This functional analytic approach to invariant manifold theory for fast-slow systems 
of differential equations has not been fully developed before, especially for the case that the fast subsystem lives on an infinite-dimensional Banach space. 
We provide rigorous functional analytic proofs for both the persistence of attracting critical manifolds as smooth slow manifolds, as well as the validity of slow manifold 
reduction near slow manifolds. Several aspects of our proofs are new in the literature even for the finite-dimensional case. For instance, we require less smoothness
assumptions on systems than comparable texts based on Lyapunov-Perron approaches. The theory as developed here provides 
a rigorous framework that allows one (for example) to derive formal statements on the dynamics of biologically meaningful spatially extended models with slowly varying parameters.
\end{abstract}


 \section{Introduction}

The aim of this article is to prove via functional analytic methods the existence and
properties of attracting slow manifolds for smooth systems of differential equations
on a Banach space $\XSet$ with $n \geq 1$ slowly evolving parameters. We consider
systems of the type
 \begin{align}
 \begin{split} \label{eq:fs-system}
 \dot{x}(t) &= F(x(t),y(t),\eps), \\
 \dot{y}(t) &= \eps \cdot g(x(t),y(t), \eps),
 \end{split}
 \end{align}
 where $x\in \XSet$, $y \in \RSet^n$ and $\eps\in \RSet_{\geq 0}$. We assume $F \in
 C^{k}(\XSet \times \RSet^n \times \RSet_{\geq 0},\XSet)$ and $g \in C^k(\XSet \times
 \RSet^n \times \RSet_{\geq 0},\RSet^n)$ for some $k \geq 1$, in the sense of
 Fr\'echet.  When $0 < \eps\ll 1$, we can interpret $y$ as a finite collection of
 slowly evolving parameters for the (fast) dynamical system on $\XSet$ governed by
 the differential equation $\dot{x} = F(x,y, \eps)$.
 
 In the case $\XSet = \RSet^m$ for some $m \geq 1$, singularly perturbed systems like
 \cref{eq:fs-system} 
 have been extensively studied, see for example
 \cite{Ku15} for a comprehensive overview. For the more general case that $\XSet$ is
 a -- possibly infinite-dimensional -- Banach space, much less literature is
 available so far. However, there has been a recent increase in interest in multiple timescale dynamics
 in spatially extended systems of differential equations \cite{Ro15, AvDe17, KaVo18,
 AvDe20, HuKu22, GoKaSc23, TsLi24, VoDoKa24, HuJeKu25, GoKaVo25, JeDoKa25}. Here we develop a functional analytic theory that is
 immediately applicable to (for example) integro-differential equations with slowly evolving
 paramaters and that may be extended, with suitable modifications, to a wider range of evolution equations including
 PDEs with slowly evolving parameters.
 
 A crucial concept for the study of fast-slow systems like system \cref{eq:fs-system} is the so-called critical manifold $C_0$. This set $C_0$ is
 typically defined as the set of all equilibria of system \cref{eq:fs-system} at
 $\eps = 0$, or more formally
 \[
 C_0 := \{ (x,y) \in \XSet \times \RSet^n: F(x,y,0)=0\}.
 \]
 Understanding the behavior of system \cref{eq:fs-system} in neighborhoods of $C_0$
 is a main focus of Geometric Singular Perturbation Theory (GSPT), which has played a
 central role in the study of fast-slow systems now for decades. In this paper we
 focus on submanifolds of $C_0$ that are normally hyperbolic. This means here that
 the spectrum of $D_x F(x,y,0)$ is uniformly bounded away from the imaginary axis for
 all points $(x,y)$ in such a submanifold. For the case $\XSet = \RSet^m$, the
 behavior of system \cref{eq:fs-system} near a compact normally hyperbolic
 submanifold $K_0$ of $C_0$ is completely described by Fenichel Theory \cite{Fe71,
 Fe74, Fe77, Fe79, Wi94}, which is widely regarded as the foundation of GSPT. 
A core result of Fenichel Theory is that $K_0$ is regularly perturbed as a locally invariant manifold
 for sufficiently small $\eps > 0$, and such regular perturbations of $K_0$ are usually
referred to as \textit{slow manifolds}. Moreover, if $K_0$ is stable, orbits whose initial 
value lies nearby $K_0$, are asymptotically attracted to an orbit that  lies on the 
slow manifold (for fixed $\eps > 0$), and we call this concept \textit{Slow Manifold Reduction}
 in this paper.
 
 Closely related to Fenichel's work is that by Hirsch, Pugh and Shub \cite{HiPuSh77}
 which was developed in parallel. Fenichel Theory was generalized to the setting
 of semiflows in abstract Banach space by Bates and co-authors in \cite{BaLuZe98,
 BaLuZe99, BaLuZe00}. Further work related to the application of the theories of
 Fenichel and Bates to systems of differential equations can be found in \cite{Jo94,
 Ka99, PlSe01, MeHa01, JoSh01, Ve05, BaLuZe08, He10, NiSt13, Ku15}.
 
In contrast to the work of Fenichel and Bates, which relies on applying the so-called
Graph Transform to (semi)flows, here we prove a Fenichel-like theory for systems of
the form \cref{eq:fs-system} on general Banach spaces $\XSet$ via a Lyapunov-Perron
approach. Such an approach relies on contraction mappings that are set up via
variation of constants formulas. The Lyapunov-Perron method can be considered to be
of a functional analytic nature compared to the more geometric proofs of Fenichel and Bates. 

Earlier attempts to study singularly perturbed systems of differential equations via the
Lyapunov-Perron method can be found in \cite{Za65, Ba68, He81, Sa90, El12}. Of these, 
only Henry's and Eldering's works in \cite{He81, El12} apply to abstract Banach space settings.
A Lyapunov-Perron approach to infinite-dimensional singularly perturbed maps was studied in \cite{Se13}.
The proof constructions used in this work are closest to those developed in Chapter 9 of Henry
\cite{He81}, although not identical (see also \cref{rem:bdd-sol-fast-subsystem}), and we prove
more results on smoothness of slow manifolds and validity of Slow Manifold Reduction than Henry
or Eldering did. We construct contraction mappings differently than was done in \cite{Sa90, El12},
which leads to less strict assumptions for our Theorems, see also \cref{rem:construction-contraction}.
In particular, our proof for the persistence of compact attracting critical manifolds requires significantly leaner
smoothness assumptions on the system at hand than Henry \cite{He81} and other aforementioned authors
made when working with Lyapunov-Perron approaches. This brings the smoothness assumptions more
in line with those known from geometric approaches as by Fenichel. Our proof approach to Slow Manifold Reduction
(\crefrange{sec:slow-mfds-straightening}{sec:red-map-smoothness})
 appears to be a new contribution to the literature, even in the case of finite-dimensional systems. 
 We work by proving the existence of a certain nonlinear projector onto the slow manifolds, rather than giving
 primacy to construction of invariant fibers as is often seen in GSPT and invariant manifold literature. 
 As discussed further on, this nonlinear projector implies existence of invariant foliations near slow manifolds as known from Fenichel theory. 

The main conditions and results of Bates are phrased and proven for the time-$t$ map 
(for sufficiently large $t  > 0$) associated to a semiflow (see for example Section 3 
of \cite{BaLuZe98}). In contrast, the main results here, namely
\cref{thm:slow-mfd-intro,thm:red-map-intro},
rely on natural assumptions on the vector field \cref{eq:fs-system} and properties
of the two-parameter
semigroup generated by the linearization of the fast component. As mentioned also for example by \cite{El12},
one advantage of the Lyapunov-Perron method is its relatively straightforward extension to generalizations,
such as the inclusion of non-autonomous terms in the vector field. Although we only provide detailed proofs here for the case of
attracting submanifolds of $C_0$, our proofs can be extended with minor modifications to the normally hyperbolic case 
when unstable eigenspaces along $C_0$ are finite-dimensional (see \cref{sec:cu-manifolds}), as is often the case in applications. We opt to present 
the proofs for the attracting case here in order to focus on the novelties of our approach.

Although in general our results for system \cref{eq:fs-system} will only hold locally in time and space near attracting
 compact submanifolds of $C_0$, the main proofs of our paper are given for a setting in which we can
 obtain results globally. The local results can be derived from these by moving the relevant compact submanifold of 
 $C_0$ to nearby $\{ x = 0 \}$  via a coordinate transformation, as well as applying suitable cut-off functions, as is further 
 discussed in \cref{sec:application}.  This can be compared to the derivation of local center manifold theory 
 via the Lyapunov-Perron method as for example in \cite{HaIo10}. The results correspond to those by Fenichel for 
 finite-dimensional systems (see \cite{Fe79, Ku15}), although our formulations are different, and arguably more 
 on the analytic side. 

Suppose that $K_0$ is a compact $\mu$-attracting submanifold of $C_0$. 
This means that  for some compact $K^y \subset \RSet^n$ and some  $h_0 \in C^k(K^y, \XSet)$ we have
\[
K_0 = \{ (h_0(y),y) \in \XSet \times K^y  \} \subseteq C_0,
\]
such that for some $\mu > 0$ it holds that
\[
 \sup \{ \real \lambda : y \in K^y, \lambda  \in \sigma(D_x F(h_0(y),y,0)) \} < -\mu. 
\]
This last condition is also sometimes referred to in literature as the spectral gap assumption.
Our first main local result is then the following; intuitively, it states that for sufficiently small $\eps \geq 0$ 
there exists a smooth invariant manifold near compact attracting subsets of $C_0$:

 \begin{theorem}[Local existence of smooth compact slow manifolds]\label{thm:slow-mfd-intro}
Suppose that $K_0$ is a compact $\mu$-attracting submanifold of the critical manifold $C_0$ of system \cref{eq:fs-system}. 
Then there exist $\eps_0 > 0$ and a function $h \in C^k(K^y \times [0,\eps_0], \XSet)$  such that:
\begin{enumerate}
\item For all $\eps \in (0,\eps_0]$, the submanifold with boundary
\[
K_{\eps} := \{ (h(y,\eps),y) \in \XSet \times K^y \}
\]
is locally invariant under the flow of \cref{eq:fs-system};
\item It holds that $h(\cdot, 0) = h_0$ and $\sup_{y \in K^y} \| h(y,\eps) - h(y,0) \|_{\XSet} \rightarrow 0$ as $\eps \downarrow 0$.
\end{enumerate}
\end{theorem}

We henceforth set $h(\cdot,\eps) =: h_{\eps}$. The second main result
describes Slow Manifold Reduction. This entails that for system \cref{eq:fs-system}
under the appropriate assumptions and for sufficiently small $\eps > 0$, all orbits near $K_{\eps}$
are exponentially attracted to an orbit on the \textit{slow manifold} $K_{\eps}$ described by
the function $h_{\eps}$ from \cref{thm:slow-mfd-intro}. 

\begin{theorem}[Reduction map]\label{thm:red-map-intro} 
Let $k \geq 2$ and suppose that $K_0$ is a compact $\mu$-attracting submanifold of the critical manifold
$C_0$ of system \cref{eq:fs-system}. 
Let 
\[
	\varphi: (t; \hat{x}_0,\eps) \mapsto \varphi(t;\hat{x}_0,\eps) \in \XSet \times \overline{V}
\]
denote the flow of system \cref{eq:fs-system}, and let $h_{\eps}$ be as in \cref{thm:slow-mfd-intro}. Then
there exist $0 < \eps_1 \leq \eps_0$, tubular neighborhoods $\hat{U}_{\eps}$ of the family of slow manifolds
\[
	\left\{ (\hat{x},\eps) \in (\XSet \times \overline{V}) \times [0,\eps_1] : \eps \in [0,\eps_1], \hat{x} \in K_{\eps} \right\},
\]
 and maps $P_{\eps} \in C^{k}(\hat{U}_{\eps}, K_{\eps})$ such that
\begin{enumerate}
\item(Commutativity of reduction map with flow) It holds for every initial condition 
$(\hat{x}_0, \eps)$ with $\eps \in [0,\eps_1]$ and $\hat{x}_0 \in \hat{U}_{\eps}$ that
\[
	P_{\eps}(\varphi(t; \hat{x}_0,\eps)) = \varphi(t; P_{\eps}(\hat{x}_0), \eps)
\]
for all $t \geq 0$ such that $\varphi(t; P_{\eps}(\hat{x}_0),\eps) \in K_{\eps}$;
\item(Attractivity of slow manifolds) There exists $C \geq 1$ such that  for every $(\hat{x}_0, \eps)$ with 
$\eps \in [0,\eps_1]$ and $\hat{x}_0 \in \hat{U}_{\eps}$  we have
\[
	\| \varphi(t; \hat{x}_0,\eps) - \varphi(t; P_{\eps}(\hat{x}_0), \eps)\|_{\XSet \times \RSet^n} \leq C e^{-\mu t/2} \|\hat{x}_0 - P_{\eps}(\hat{x}_0)\|_{\XSet \times \RSet^n}
\]
for all $t \geq 0$ such that $\varphi(t; P_{\eps}(\hat{x}_0),\eps) \in K_{\eps}$.
\end{enumerate}
\end{theorem}

The reduction map $P_{\eps}$ gives rise to a smooth stable invariant foliation near $K_{\eps}$ for each 
$\eps \in[0,\eps_1]$ by defining as leaves (or fast fibers) the Banach manifolds
$P_{\eps}^{-1}(h_{\eps}(y),y)$ for each $y \in \overline{V}$. The slow manifolds $K_{\eps}$
then act as stems of these foliations. This way, \cref{thm:red-map-intro}
corresponds to Fenichel's results on invariant foliations near slow manifolds.  

The reduction map $P_{\eps}$ is a topological semi-conjugation between the semiflow of the system given by
\[
\RSet_{\geq 0} \times \hat{U} \ni (t; \hat{x}_0) \mapsto \varphi(t; \hat{x}_0, \eps)
\]
 and the slow flow on $K_{\eps}$, which for $t \geq 0$ is the semiflow restricted to
 $K_{\eps}$. This property can be displayed via a commutative diagram as in
 \cref{fig:com-dia-red-map-intro}.

 \begin{figure}\label{fig:com-dia-red-map-intro}\centering
\begin{tikzcd}[sep = 25mm, every label/.append style = {font = \normalfont}]
\mathllap{{}  \hat{x}_0 \in \,\,}  \hat{U}_{\eps} \arrow{d}[swap]{\varphi(t; \hat{x}_0, \eps)} \arrow{r}{P_{\eps}(\hat{x}_0)} & K_{\eps} \arrow{d}  \arrow{d}{\varphi(t; \hat{x}^p_0, \eps)} \mathrlap{{} \ni \hat{x}^p_0}  \\
\mathllap{{}  \hat{x}_t \in \,\,} \hat{U}_{\eps} \arrow[swap]{r}{P_{\eps}(\hat{x}_t)}& K_{\eps}  \\   [-60pt]
\end{tikzcd}
\caption{The relationship between the (semi)flow $\varphi$ of the system and the reduction map $P_{\eps}$, for arbitrary $t \geq 0$ such
			that $\varphi(t; (h_{\eps}(y_0),y_0),\eps) \in K_{\eps}$.}
\end{figure}

\cref{thm:slow-mfd-intro,thm:red-map-intro}  
are a consequence
of \cref{thm:slow-mfd-existence,thm:slow-mfd-k-smoothness,thm:red-principle-slow-mfd}
that are proven in
\crefrange{sec:slow-mfds-existence}{sec:red-map-smoothness}. Intuitively, in these
proofs the full system \cref{eq:fs-system} is understood as a perturbation of the system
  \begin{align}
 \begin{split}\label{ODE-system-unperturbed}
 \dot{x}(t) &= D_x F(0,y(t),0) \, x(t) , \\
 \dot{y}(t) &= 0.
  \end{split}
 \end{align}
Therefore, the matter of existence of smooth slow manifolds for small $\eps \geq 0$
boils down to showing that the invariant set $\{ x = 0 \}$ of the system
\cref{ODE-system-unperturbed} persists under small perturbations as a smooth
invariant manifold of the perturbed system. The proof of this is treated in 
\crefrange{sec:slow-mfds-existence}{sec:k-smooth-slow-mfds}, 
and the resulting invariant manifold is referred to as
the slow manifold. In \cref{sec:slow-mfds-straightening} we discuss a
coordinate transformation which straightens this slow manifold. This prepares us for
the proof of existence of a locally smooth reduction map in \crefrange{sec:red-map-existence}{sec:red-map-smoothness}.
Observe that in \crefrange{sec:slow-mfds-existence}{sec:red-map-smoothness}, see also \cref{ODE-system}, we only work with $x$-
an $y$-variables and there is no explicit parameter $\eps$. One can translate \cref{eq:fs-system}
to this setting by taking $y$ and $\eps$ together as one vector of parameters $\tilde{y} \in \RSet^{n+1}$, and setting $\partial_t \, \tilde{y}_{n+1} = 0$.

Our theorems and proofs can be considered a generalization of center manifold theory
for equilibria of ODEs. Specifically, \cref{thm:slow-mfd-intro} can be seen as
an extension of the (local) Center Manifold Theorem, whilst \cref{thm:red-map-intro} can
then be seen as an extension of the concept of Center Manifold Reduction, which is
classically understood as the idea that in the center-stable manifold of a
non-hyperbolic equilibrium (or periodic orbit) all orbits are asymptotically
attracted in forward time to an orbit lying on the center manifold, see for example
\cite{Ke67b}. In the following we provide some more perspective on how our proof methods and 
results relate to existing literature on Lyapunov-Perron approaches to the Center Manifold Theorem and
singularly perturbed systems of differential equations.

 A classical text on proving the Center Manifold Theorem via the Lyapunov-Perron
 method is that by Carr \cite{Ca81}, and many texts on the subject seem to closely
 follow his approach, see for example Kelley \cite{Ke67}, Chapter 6 of Henry
 \cite{He81}, Sijbrand \cite{Sij85}, Chow \& Lu \cite{ChLu88a, ChLu88b}, Chapter 10
 of Hale \& Verduyn Lunel \cite{HaVL93} and Chicone \& Latushkin \cite{ChLa97}. The
 proof relies on constructing a contraction mapping on a space of
 parameterizations of manifolds via a variation of constants formula (also referred to in literature as Lyapunov-Perron formula). 
 Other authors construct this contraction mapping on exponentially weighted spaces of bounded solutions
  instead (see for example Vanderbauwhede \& Van Gils \cite{VaVG87} or Haragus \& Iooss \cite{HaIo10}) but
 their construction remains conceptually similar to that of Carr otherwise. A direct
 extension of Vanderbauwhede's theory to perturbations of systems of the form
 \cref{ODE-system-unperturbed} in the context of fast-slow systems of ODEs can be
 found in \cite{Sa90}.
 
Lyapunov-Perron approaches for proving the Center Manifold Theorem, as worked out by Kelley
 \cite{Ke67}, Carr \cite{Ca81} and others, are closely related to methods for
 integral manifolds developed by Bogolyubov and Mitropolsky and their students in the
 Soviet literature. The application of these methods to analyzing periodic solutions of non-linear
 differential equations is also known as the Krylov-Bogolyubov-Mitropolsky (KBM)
 method for periodic solutions. Expositions on the KBM-method can be found in texts
 by Hale \cite{Ha61,Ha80}. Attempts to adapt methods for integral manifolds to the
 setting of fast-slow systems of ODEs appear to have been done first by Zadiraka
 \cite{Za65} and Baris \cite{Ba68}, and this work appears to be very similar to the
 methods from Carr \cite{Ca81} and others mentioned above. Yi \cite{Yi93a, Yi93b}
 proved existence of integral manifolds and their properties in a setting that can be
 seen as a generalization of that of system \cref{ODE-system-unperturbed}. More
 related historical notes and references can be found in \cite{KnAu84, Ly92,
 ScSoMo14}. 
 
We provide a functional analytic theory for perturbations of systems of the type
\cref{ODE-system-unperturbed} by constructing contraction mappings differently than
most of the texts mentioned above, which we believe allows for significantly more
compact proofs for the case at hand than when more closely following the
constructions of Carr \cite{Ca81} and others. Moreover, we do not require a Lipschitz
condition on $D_x F(0,y,\eps)$ for existence of slow manifolds in contrast to other
texts such as \cite{Sa90, Yi93a}, see also \cref{rem:construction-contraction}.
As mentioned before, our constructions appear to be more similar to those used in Chapter 9 of Henry
\cite{He81}, but we still deviate from his approach in certain proof steps, see also 
\cref{rem:bdd-sol-fast-subsystem}. All these differences are explained in more detail
in \cref{sec:slow-mfds-existence} and onwards. 

In \crefrange{sec:slow-mfds-straightening}{sec:red-map-smoothness} we
provide proofs for \cref{thm:red-principle-slow-mfd} (which implies
\cref{thm:red-map-intro}), and this Theorem can be seen as a generalization of the
classical reduction principle from Pliss (see \cite{Ke67b}). Many texts on center
manifold theory, including those relying on the Lyapunov-Perron method, treat this
reduction principle by first proving the existence of certain invariant
foliations. The existence of a locally smooth reduction map (like $P_{\eps}$ of
\cref{thm:red-map-intro}) is implied by these invariant foliations. See for
example \cite{ChLiLu91} for a Lyapunov-Perron approach to invariant foliations near
equilibria of evolution equations. We opt for a reverse approach where we prove the
existence of a locally smooth reduction map, and then (as mentioned before) invariant
foliations follow from the reduction map. In the context of singularly perturbed
system of differential equations, we are not aware of any literature that has proven
existence of invariant foliations near slow manifolds in this manner yet. So our
proof approach appears to be a new contribution, even in the case of
finite-dimensional systems. The approach is in many ways close to the approach we
follow for proving the existence of smooth slow manifolds in
\crefrange{sec:slow-mfds-existence}{sec:k-smooth-slow-mfds}. 

For neighborhoods of isolated non-hyperbolic points on a critical manifold,
reduction of the dynamics to a finite-dimensional center manifold was partially
worked out in \cite{AvDe20} for some integro-differential equations and PDEs with 
slowly varying parameters. Our work here can be considered complementary to \cite{AvDe20}
 as part of a broader effort towards building a complete theory for constructing relaxation oscillations,
canards and delayed bifurcations in spatially extended systems with slowly varying
parameters. While the setting of system \cref{eq:fs-system} does not directly apply
to PDEs, the proofs worked out here should provide a substantial step towards an analogous
theory for PDEs. 
Our work partly makes use of theory on uniformly continuous (semi)groups. A PDE
theory would firstly have to extend these theorems to the case of strongly continuous
semigroups (or a subclass thereof). We believe this to be readily possible. One could utilize for
example the sectorial estimates in the text of Henry \cite{He81}. Next, theorems
and proofs would need to be amended to cover the possibility of densely defined operators and interpolation
spaces for non-linearities. A guide for this step  would be the text on infinite-dimensional center manifold
theory by Haragus \& Iooss \cite{HaIo10}, and we expect to require similar technicalities for the extension.

Our proof methods provide an alternative functional analytic approach to fast-slow systems 
on Banach space compared to the more geometric material in texts like \cite{BaLuZe98, BaLuZe08, MeHa01}.
We believe such a functional analytic approach to have utility for the analysis of biologically meaningful
spatially extended models, such as integro-differential equations with neuroscientific interpretation. We
also foresee that the methods developed in this paper can be of interest to other studies
into the extension of classical theory for fast-slow systems of differential equations, for example in the 
case of adding noise to a fast-slow system, or when considering spatially extended systems with slowly 
evolving spatial domain.
 
The article is built up as follows; in \cref{sec:preliminaries} we go through notational conventions 
and basic concepts that are used throughout this text. 
In \crefrange{sec:slow-mfds-existence}{sec:k-smooth-slow-mfds} we prove
existence of attracting smooth slow manifolds for \cref{eq:fs-system} under a set of
conditions, which implies \cref{thm:slow-mfd-intro}. In
\crefrange{sec:slow-mfds-straightening}{sec:red-map-smoothness} we treat rigorously
Slow Manifold Reduction for  \cref{eq:fs-system}, which implies
\cref{thm:red-map-intro}. We provide local theory for \cref{eq:fs-system} in
\cref{sec:application}. In \cref{sec:cu-manifolds}
we briefly discuss how to extend our theory to normally hyperbolic critical
manifolds, when there are finitely many unstable directions. In this case, with minor
modifications, our theory is expected to hold by replacing the slow manifolds from the
attracting case by their center-unstable manifolds.


 \section{Preliminaries}\label{sec:preliminaries}
 
 In this Section we introduce some notational conventions that will be used throughout the rest of this article. Suppose that $A,B$ are two Banach spaces; we indicate their norms by $\| \, . \, \|$, where it follows from the context which norm is meant. We denote by $L^i(A,B)$ (with $i \in \mathbb{N}_{\geq 1}$) the space of bounded $i$-linear operators from $A^i$ to $B$. If $T \in L^i(A,B)$, then by $\| T \|$ we indicate the operator norm of $T$. Note that $L^i(A,B)$ equipped with the operator norm is a Banach space itself. In the case $i=1$ we shall also simply denote $L(A,B)$ for the space $L^1(A,B)$.
 
Let $U \subseteq A$ be an open subset. We denote by $C^k(U,B)$ the space of all continuous functions $f: U \rightarrow B$ that are $k$-times continuously differentiable in the sense of Fr\'echet. We denote the $k$-th derivative of $f$ as $D^k f$, and note that then for all $1 \leq i \leq k$ we have $D^i f \in C^0(U, L^{i}(A,B))$.

Next, we define the space of bounded continuous functions from $A$ to $B$ as
\[
BC^0(U,B) := \{ f \in C^0(U,B): \sup_{a \in U}\|f(a)\| < \infty \}.
\]
The space $BC^0(U,B)$ is a Banach space when equipped with the supremum norm
\[
\| f \|_{BC^0(U,B)} :=  \| f \|_{\infty} = \sup_{a \in U}\|f(a)\|.
\]
For $k \geq 1$ we inductively define
\[
BC^k(U,B) := \left\{ f \in BC^{k-1}(U,B) \cap C^k(U,B): \sup_{a \in U}\|D^k f(a)\| < \infty \right\},
\]
and $BC^k(U,B)$ is a Banach space when equipped with the norm
\[
\| f \|_{BC^k(U,B)} :=  \| f \|_{BC^{k-1}(U,B)} +  \|D^k f \|_{\infty},
\]
where 
\[
 \|D^k f \|_{\infty} = \sup_{a \in U}\|D^k f(a)\|.
\]

Moreover, we define for $k \geq 1$ the spaces  
 \[
 BC^{k-1,1}(U,B) :=  \left\{ f \in BC^{k-1}(U, B): \sup_{\substack{a_1,a_2 \in U \,: \\ a_1 \neq a_2}} \frac{\|D^{k-1}f(a_2)-D^{k-1}f(a_1)\|}{\|a_2-a_1\|} < \infty \right\}.
 \]
The spaces $BC^{k-1,1}(U,B)$ are Banach spaces when equipped with the norm
\[
\| f \|_{BC^{k-1,1}(U,B)} := \| f \|_{BC^{k-1}(U,B)} + \sup_{\substack{a_1,a_2 \in U \,: \\ a_1 \neq a_2}} \frac{\|D^{k-1}f(a_2)-D^{k-1}f(a_1)\|}{\|a_2-a_1\|}.
\]
We note that  $BC^{k-1,1}(U,B)$ is the subspace of functions $f$ in $BC^{k-1}(U,B)$ for which $D^{k-1}f$ is also uniformly Lipschitz.

Let $V \subseteq B$ be an open subset. Sometimes we shall also consider the space of $\gamma$-bounded functions over $\RSet$ and into $V$ which is defined as
\[
BC_{\gamma}^0(\RSet,V) := \{ f \in C^0(\RSet,V): \sup_{t \in \RSet} e^{\gamma |t|} \|f(t)\| < \infty \},
\]
which is a Banach space when equipped with the norm
\[
\| f \|_{\gamma,\infty} :=   \sup_{t \in \RSet} e^{\gamma |t|} \|f(t)\|.
\]

If in particular $\overline{U}, \overline{V}$ are closed subsets of Euclidean space, then we assume the norm on them to be the standard Euclidean norm. In this case we also use single bars $| \, . \, |$ to indicate the norm. We shall do this as well when working with the norm on $\XSet$ or with product norms on $\XSet$ times some subset of Euclidean space.

For functions of multiple variables, we denote the uniform Lipschitz constant with respect to one variable as $\Lip$ subscript the relevant variable. So for example for a function $R: \XSet \times \RSet^n \rightarrow \XSet$, we have
\[
\Lip_x  R := \sup_{\substack{a_1,a_2 \in \XSet, \, y \in \RSet^n: \\ a_1 \neq a_2}} \frac{| R(a_2,y)-R(a_1,y)|}{|a_2-a_1|}.
\]

The following Lemma from Henry \cite{He81} shall turn out useful for proofs later on;

 \begin{lemma}\label{closedsubset}
Suppose $A,B$ are Banach spaces, let $U \subseteq A$ be an open subset and $k \geq 1$. For arbitrary $\alpha, \beta > 0$, the set
\[
\{ f \in BC^{k-1,k}(U,B) : \| f \|_{BC^k(U,B)} \leq \alpha, \, \Lip \, D^{k-1} f \leq \beta \}
\]
 is a closed and bounded subset of $BC^0(U,B)$. 
 \end{lemma}

  \begin{proof}
The proof is a variation on material in Henry \cite{He81} pages 151 -- 152.
  \end{proof}

  
  \section{Existence of global invariant manifolds}\label{sec:slow-mfds-existence}
  
This Section and \cref{sec:slow-mfds-smoothness} treat theory which implies 
\cref{thm:slow-mfd-intro} from the Introduction for the case $k=1$. We consider the
system of differential equations
 \begin{equation}\label{ODE-system}
 \begin{aligned}
 \dot{x}(t) &= F(x(t),y(t)), \\
 \dot{y}(t) &= g(x(t),y(t)),
 \end{aligned}
 \end{equation}
 where $x \in \XSet$, $y \in \overline{V}$ with $V$ an open subset of $\RSet^n  (n \geq 1)$, and $F \in C^0(\XSet \times \overline{V},\XSet)$, $g \in C^0(\XSet\times \overline{V}, \RSet^n)$ are locally Lipschitz. Also $g(x,y)=0$ whenever $y \in \partial V$, and we can write 
 \[
 F(x,y)=A_0(y)x+R_0(x,y),
 \]
 with $A_0 \in C^0(\overline{V},L(\XSet,\XSet))$.

We are looking for an invariant manifold of system \cref{ODE-system} that can be parameterized over the $y$-variable via a function $h \in BC^{0,1}(\overline{V},\XSet)$. Such a manifold corresponds to a slow manifold in the context of fast-slow systems, see for example \cite{He81, Sa90, El12}.

To find such an invariant manifold, consider a \emph{candidate} $\sigma:\overline{V}\to\XSet$ for the map $h$ that we are seeking. We can then restrict the second component of the system~\cref{ODE-system} to $x=\sigma(y)$ and consider a solution $y=\tilde{\psi}(t)$ of
$
   \dot{y}=g(\sigma(y),y).
$
Plugging this into the first component of the system~\cref{ODE-system} naturally leads us to study the non-autonomous equation
\[
   \dot{x}=A_0(\tilde{\psi}(t)) x + R_0(x,\tilde{\psi}(t)).
\]
To start with the linear part, 
for arbitrary $\tilde{\psi} \in C^0(\RSet,\overline{V})$, we denote by $T_0(t,s;\tilde{\psi})$ the process associated to the non-autonomous linear ODE
\[
\dot{x}(t) = A_0(\tilde{\psi}(t)) x(t).
\]
That is, for each  $s \in \RSet$ and $\xi \in \XSet$, the set $\{T_0(t,s; \tilde{\psi}) \xi : t \geq s \}$ is the (forward) orbit of this ODE for $x(s) = \xi$. Observe that whilst in general $x(t) \in \XSet$ for all $t \geq s$, in the case $\XSet = \RSet$ we have the explicit expression
\[
T_0(t,s; \tilde{\psi}) = e^{\int^t_s A_0(\tilde{\psi}(\tilde{s})) \, d\tilde{s}}.
\]
 The process on $\XSet$ has the following properties;
 \begin{lemma}\label{lem:process-props}
For arbitrary $\tilde{\psi} \in C^0(\RSet,\overline{V})$, it holds for $t \geq s$, with $s \in \RSet$, that $T_0(t,s; \tilde{\psi})$ is a bounded linear operator on $\XSet$ with the properties
 \begin{enumerate}
 \item $T_0(s,s; \tilde{\psi}) = I_{\XSet}, \, T_0(t,r; \tilde{\psi}) = T_0(t,s; \tilde{\psi}) T_0(s,r; \tilde{\psi})$ for all $t \geq s \geq r$;
 \item $(t,s) \mapsto T_0(t,s; \tilde{\psi})$ is continuous in the operator norm for all $t \geq s$;
 \item $\partial_t T_0(t,s; \tilde{\psi}) = A_0(\tilde{\psi}(t)) T_0(t,s; \tilde{\psi})$ for $t \geq s$;
 \item $\partial_s T_0(t,s; \tilde{\psi}) = -T_0(t,s; \tilde{\psi})A_0(\tilde{\psi}(s)) $ for $t \geq s$.
 \end{enumerate}
 \end{lemma}
 
\begin{proof}
These statements can be proven as in Chapter 5 of Pazy \cite{Pa83}. 
\end{proof}

 \begin{remark}
The family $T_0(t,s; \tilde{\psi})$ of linear operators with $(t,s) \in \RSet^2$ such that $t \geq s$ that we call `process' here is also known under several other terms in the literature, such as two-parameter semigroup and Green's function.
 \end{remark}
 
 We shall also make the following assumption on the process $\{ T_0(t,s, \tilde{\psi}): t \geq s \}$, when $\tilde{\psi}$ is a solution to the $y$-subsystem of \cref{ODE-system}. The assumption entails exponential stability of the process uniformly for all such $\tilde{\psi}$:
 \vspace{0.2cm}
 
 \begin{enumerate}
 \item[$\mathrm{(H1)}$] There exist $\mu>0$ and  $K \geq 1$ such that 
\begin{equation}\label{eq:process-exp-bdd}
| T_0(t,s; \tilde{\psi}) \xi | \leq K e^{-\mu(t-s)} |\xi| \quad \text{for } \xi \in \XSet, t \geq s,
\end{equation}
for all $\tilde{\psi} \in C^1(\RSet,\overline{V})$ that solve the ODE $\dot{y} = g(x(t),y)$ for some function $x\in BC^{0}(\RSet, \XSet)$. 
 \end{enumerate}
 
 \begin{remark}\label{rem:spectral-gap-to-H1}
By applying \cref{cor:slow-A-stable} from the Appendix (or alternatively combine
\cref{lem:un-bounds-stable-process} with a lemma from Daletskii and Krein
\cite{DaKr74}, similarly to what is done in the proof of \cref{cor:slow-A-Henry}), one can show that if $\overline{V}$ is bounded and there
exists $\mu > 0$ such that
\[
\sup \{ \real \lambda: y \in \overline{V}, \lambda \in \sigma(A(y)) \} < -\mu,
\]
then there exists $N_0 > 0$ such that if
\[
\sup \{ | g(x,y) | : x \in \XSet, y \in \overline{V} \} \leq N_0,
\]
then the assumption (H1) is satisfied with the same choice of $\mu$. This means
assumption (H1) is satisfied when applying the theory in this Section to the
perturbation of compact critical manifolds of fast-slow systems as we do in \cref{sec:application}.
 \end{remark}
 
 We now make two further assumptions (H2)-(H3) on system \cref{ODE-system}:
\vspace{0.2cm}
 \begin{enumerate}
 \item[$\mathrm{(H2)}$] There exist constants $M_0, M_1^x, M_1^y > 0$  with $M^x_1 < \mu / K$  such that
 \[ 
 \sup\left\{| R_0(x,y) | : (x,y) \in \XSet \times \overline{V} \right\} \leq M_0,
 \]
 as well as
  \[ 
| R_0(x_2,y) -  R_0(x_1,y) |  \leq M_1^x |x_2 - x_1|
 \]
 for all $(x_1,x_2,y) \in \XSet \times \XSet \times \overline{V}$, and
 \[
| F(x,y_2) -  F(x,y_1) | \leq M_1^y |y_2 - y_1|
 \]
 for all $(x,y_1,y_2) \in \XSet \times \overline{V} \times \overline{V}$;
 \item[$\mathrm{(H3)}$] There exists a constant $N_1 > 0$ with $N_1 < \mu - K M_1^x$ such that
 \[
 | g(x_2,y_2) - g(x_1,y_1) | \leq N_1 |(x_2,y_2) - (x_1,y_1)|
 \]
 for all $(x_1,y_1), (x_2,y_2) \in \XSet \times \overline{V}$.
 \end{enumerate}
 \vspace{0.25cm}
 Note that (H2)-(H3) ensure that solutions to system \cref{ODE-system} exist for all time.
 
 \begin{remark}
 The condition $N_1 < \mu-K M_1^x$ in assumption (H3) on the global Lipschitz constant $N_1$ for $g$ can be interpreted as a timescale separation, as it ensures that the growth rate for $y$ in System \cref{ODE-system} is strictly smaller than the rate of decay for $x$.
 \end{remark}
 
 By the above remark, under assumptions (H1)-(H3) the $y$-variable in system \cref{ODE-system} may be interpreted as slow when compared to the (fast) $x$-variable.
 
 \begin{definition}\label{def:slow-mfd}
If (H1)-(H3) hold, we say that a function $h \in BC^{0,1}(\overline{V},\XSet)$ parameterizes a (global) invariant manifold of system \cref{ODE-system} 
if for each $\eta \in \overline{V}$ the unique solution $(x(t), y(t))$, with $t \in [s,\infty)$, of system \cref{ODE-system} with initial condition $(x(s), y(s))=(h(\eta),\eta)$ 
has the property that $x(t)=h(y(t))$ for all $t \geq s$. The invariant manifold itself is then defined as
\[
S_h := \{ (h(\eta),\eta) : \eta \in \overline{V} \}.
\]
 \end{definition}

In this Section we aim to prove the following Theorem:
 \begin{theorem}[Existence of invariant manifolds]\label{thm:slow-mfd-existence}
Assume that (H1)-(H3) are satisfied for system \cref{ODE-system}. Suppose that there exists $\delta > 0$ with $K M_1^x+N_1(\delta +1) < \mu$ such that
\begin{align*}
\frac{K M_1^y}{\mu-K M_1^x-N_1(\delta +1)}  < \delta.
\end{align*}
Then the following statements hold:
 \begin{enumerate}
 \item There exists a function $h \in BC^{0,1}(\overline{V},\XSet)$ that parameterizes a global invariant manifold of system \cref{ODE-system} with 
 \[
 \|h\|_{BC^{0}(\overline{V},\XSet)} \leq K M_0 / \mu, \quad \Lip \, h \leq \delta;
 \]
 \item Moreover, any solution $(x,y) \in C^1(\RSet, \XSet \times \overline{V})$ of system \cref{ODE-system} for which
 \[
 \sup_{t \in \RSet} | x(t) | < \infty,
 \]
is contained in the manifold $S_h = \{ (h(\eta),\eta) : \eta \in \overline{V} \}$.
 \end{enumerate}
 \end{theorem}
 
 \begin{remark}\label{rmk:delta-existence}
 The existence of a $\delta$ as mentioned in \cref{thm:slow-mfd-existence} is
 guaranteed for sufficiently small $N_1 > 0$; namely set 
 \[
 \delta := \frac{2KM^y_1}{\mu - K M^x_1},
 \]
 then for $N_1 < (\mu-KM^x_1)/(2(\delta+1))$ we have
   \[
\frac{KM^y_1}{\mu - K M^x_1 - N_1(\delta+1)} < \frac{2KM^y_1}{\mu - K M^x_1} = \delta.
 \]
  \end{remark}
 
 \begin{remark}
 By applying to \cref{eq:fs-system} appropriate coordinate transformations and/or
 parameter rescalings, the above \cref{thm:slow-mfd-existence}, combined with
 \cref{rmk:delta-existence} and \cref{thm:slow-mfd-k-smoothness} from
 \cref{sec:k-smooth-slow-mfds}, translates into \cref{thm:slow-mfd-intro} from the Introduction.
 Technical details on this can be found in \cite{Do26}.
 \end{remark}
 
We prove \cref{thm:slow-mfd-existence} via a contraction mapping argument. That means
we aim to find $h$ from part 1 of \cref{thm:slow-mfd-existence} as the unique
fixed point of a contraction mapping $\Lambda$ by applying the Banach Fixed-Point
Theorem. We construct $\Lambda$ in our argument as follows;

\vspace{0.2cm}

 \begin{itemize}
 \item[\underline{\bf{Step 1}}]  For arbitrary $\eta \in \overline{V}$ and $\sigma \in BC^{0,1}(\overline{V},\RSet)$ we let $\psi(\blank; \eta, \sigma) \in C^1(\RSet, \overline{V})$ be the unique solution to the IVP $\dot{y} = g(\sigma(y),y), \, y(0) = \eta$ and consider the properties of this solution;
 \item[\underline{\bf{Step 2}}]  We show that for each $\tilde{\psi} = \psi(\blank ; \eta, \sigma)$ there exists a unique bounded function $\phi(\blank; \tilde{\psi}) \in BC^0(\RSet, \XSet)$ that solves the differential equation
 \[
\dot{x} = F(x,\psi(t;\eta,\sigma)).
\]
For brevity, we shall also denote $\phi(\blank; \eta, \sigma) = \phi(\blank ; \tilde{\psi})$ for the composition of $\phi$ with $\psi$;
 \item[\underline{\bf{Step 3}}]  We now define $\Lambda$ by setting $\Lambda(\sigma)(\eta) := \phi(0;\eta,\sigma)$. We show that $\Lambda$ is a contraction mapping on some closed ball in $BC^{0,1}(\overline{V},\XSet)$ by making use of the properties of the bounded solution $\phi(\cdot; \eta, \sigma)$ from the previous step. 
 \end{itemize} 
 
 \vspace{0.2cm}
 
If $h \in BC^0(\overline{V}, \XSet)$ parameterizes an invariant manifold, then the mappings
$\psi: (\eta, h) \mapsto \psi(\blank; \eta, h)$ and $\phi$ together with the flow of
system \cref{ODE-system} should satisfy a commutative diagram as displayed in
\cref{fig:com-dia-slow-mfd}. Therefore $h$ has to be a fixed point of $\Lambda$. A
formal justification that any fixed point of this map $\Lambda$ provides an invariant
manifold parameterization is given in \cref{thm:slow-mfd-eqv} further on. 

\begin{figure}\label{fig:com-dia-slow-mfd}\centering
\begin{tikzcd}[sep = 20mm, every label/.append style = {font = \normalfont}]
\eta \arrow[symbol = \in]{d} & \tilde{\psi} \arrow[symbol = \in]{d} \\ [-50pt]
 \overline{V} \arrow{r}{\psi(\blank; \eta, h)} \arrow[swap]{d}{h(\eta) \times \eta} & C^0(\RSet, \overline{V}) \arrow{d}{\phi(\blank; \tilde{\psi})}  \\
\XSet \times \overline{V}  \arrow[swap]{r}{x(\blank; (\xi,\eta))} \arrow[symbol = \ni]{d} & BC^0(\RSet,\XSet) \\ [-50pt]
(\xi,\eta) & 
\end{tikzcd}
\caption{Suppose that $h \in BC^0(\overline{V}, \XSet)$ parameterizes an invariant manifold, and the mappings $\psi : \overline{V} \times \{ h \}\rightarrow C^0(\RSet, \overline{V})$ 
as well as $\phi: C^0(\RSet, \overline{V}) \rightarrow BC^0(\RSet, \XSet)$ are as outlined in steps 1-3. Then we should have for every $\eta \in \overline{V}$ and $t \in \RSet$ 
that $x(t;(h(\eta),\eta)) = \phi(t;\psi(\cdot;\eta,h))$. The Figure displays this relationship as a commutative diagram. In particular, setting $t=0$ gives 
$h(\eta) = \phi(0;\psi(\cdot;\eta,h)) = \Lambda(h)(\eta)$, with $\Lambda$ defined as in step 3.}
\end{figure}

 \begin{remark}\label{rem:construction-contraction}
We stress that the second step in this construction differs from the method described in the classical text by Carr \cite{Ca81} -- also used by many others like for example \cite{Ke67, Sij85}  -- for proving the existence of the center manifold of an equilibrium of an ODE. If one were to try to stick to this method as closely as possible here, an obvious second step would be to look for the unique bounded solution to the linear inhomogeneous differential equation
 \[
\dot{x}(t) = A_0(\psi(t;\eta,\sigma)) x(t) + R_0(\sigma(\psi(t;\eta,\sigma)),\psi(t;\eta,\sigma)),
\]
and then for the third step pick the value of this solution at $t=0$. Although this
seems to lead to a more explicit expression for the map $\Lambda$, we believe our
construction to require less assumptions and to allow for simplification of certain
proof steps later on. Our construction differs in a similar way from the KBM-method
as presented in Chapter 7 of Hale \cite{Ha80} or in Lykova \cite{Ly92}. The
construction in Chapter 9 of Henry \cite{He81} parallels ours more closely but still
certain proof steps below appear to deviate from his approach, see also
\cref{rem:bdd-sol-fast-subsystem}.
 \end{remark}

The Lipschitz constant of $\sigma \in BC^{0,1}(\overline{V}, \XSet)$ shall henceforth be denoted by 
\[
  L_{\sigma} := \sup_{a,b \in \overline{V} \,: \, a \neq b} \frac{|\sigma(b)-\sigma(a)|}{|b-a|}.
\]
We now first prove the following properties of the unique solution $\psi(t; \eta, \sigma)$ to the IVP $\dot{y} = g(\sigma(y),y), \, y(0) = \eta$ (step 1 mentioned above):
 
 \begin{lemma}\label{lem:sol-slow-subsystem}
 Under assumption (H3) we have
  \begin{enumerate}
\item For every $\eta_1, \eta_2 \in \overline{V}$ and $\sigma \in BC^{0,1}(\overline{V}, \XSet)$, it holds for all $t \in \RSet$ that
\[
|\psi(t; \eta_1, \sigma) - \psi(t; \eta_2, \sigma) | \leq |\eta_2 - \eta_1| e^{N_1 (L_{\sigma}+1)|t|};
\]
\item For every $\eta \in \overline{V}$ and $\sigma_1, \sigma_2 \in BC^{0,1}(\overline{V}, \XSet)$, it holds for all $t \in \RSet$ and any $\delta \geq L_{\sigma_1}$   that
\[
|\psi(t; \eta, \sigma_2) - \psi(t; \eta, \sigma_1) | \leq  \frac{\| \sigma_2 - \sigma_1 \|_{\infty}}{\delta + 1} \left( e^{N_1 (\delta+1)|t|} - 1 \right).
\]
 \end{enumerate}
 \end{lemma}
 
 \begin{proof}
  \begin{enumerate}
\item Set $y_i(t) = \psi(t; \eta_i, \sigma)$ for $i=1,2$. Then integration of the IVP gives
\[
y_2(t)-y_1(t) = \eta_2 - \eta_1 + \int_0^t g(\sigma(y_2(s)),y_2(s)) - g(\sigma(y_1(s)),y_1(s)) \, ds.
\]
This implies
\[
|y_2(t)-y_1(t)| \leq |\eta_2 - \eta_1| + N_1 (L_{\sigma}+1) \int_0^t |y_2(s) -y_1(s)| \, ds.
\]
An application of Gr\"onwall's inequality gives
\begin{align*}
|y_2(t)-y_1(t)| \leq |\eta_2 - \eta_1| e^{N_1 (L_{\sigma}+1)|t|}.
\end{align*}

\item Set $y_i(t) = \psi(t; \eta, \sigma_i)$ for $i=1,2$.
Then integration of the IVP gives
\[
y_2(t)-y_1(t) = \int_0^t g(\sigma_2(y_2(s)),y_2(s)) - g(\sigma_1(y_1(s)),y_1(s)) \, ds.
\]
This implies, for $\delta \geq L_{\sigma_1}$:
\begin{align*}
|y_2(t)-y_1(t)| &\leq \sgn(t)  N_1  \int_0^t |\sigma_2(y_2(s)) - \sigma_1(y_1(s))| + |y_2(s) - y_1(s)| \, ds \\
&\leq \sgn(t)  N_1  \int_0^t \| \sigma_2 - \sigma_1 \|_{\infty} +  (\delta+ 1) |y_2(s) - y_1(s)| \, ds \\
&=  N_1  \| \sigma_2 - \sigma_1 \|_{\infty} |t| + \sgn(t) N_1  (\delta + 1)  \int_0^t |y_2(s) - y_1(s)| \, ds.
\end{align*}
An application of Gr\"onwall's inequality gives
\begin{align*}
|y_2(t)-y_1(t)| &\leq  N_1 \| \sigma_2 - \sigma_1 \|_{\infty} \left( |t| +  N_1  (\delta + 1) \int_0^t s e^{ N_1 (\delta+1)|t-s|}\, ds \right)  \\
&=  \sgn(t) N_1  \| \sigma_2 - \sigma_1 \|_{\infty} \int_0^t e^{ N_1 (\delta+1)|t-s|}\, ds \\
&= \frac{\| \sigma_2 - \sigma_1 \|_{\infty}}{\delta + 1} \left( e^{ N_1 (\delta+1)|t|} - 1 \right).
\end{align*}
   \end{enumerate}
 \end{proof}
 
Next we discuss some Lemmas that are useful for carrying out step 2 of our contraction mapping construction. The following one provides an integral condition for the existence of bounded solutions to the ODE $\dot{x}(t) = F(x(t),\tilde{\psi}(t))$:
 
 \begin{lemma}\label{lem:bdd-sol-eqv}
Let $\tilde{\psi} \in C^0(\RSet,\overline{V})$ and suppose that for some $\mu > 0$ and  $K \geq 1$ the process $\{ T_0(t,s, \tilde{\psi}): t \geq s \}$ associated to $A(\tilde{\psi}(t))$  satisfies the property 
\[
| T_0(t,s; \tilde{\psi}) \xi | \leq K e^{-\mu(t-s)} |\xi| \quad \text{for } \xi \in \XSet, t \geq s.
\]

 Then a function $x \in BC^0(\RSet, \XSet)$ satisfies the ODE
\begin{equation}\label{fastODE}
\dot{x}(t) = F(x(t),\tilde{\psi}(t)), \quad t \in \RSet,
\end{equation}
if and only if for all $t \in \RSet$ it holds that
\begin{align}\label{bddsolution}
x(t) &=  \int_{-\infty}^t T_0(t,s; \tilde{\psi}) R_0(x(s),\tilde{\psi}(s)) \, ds.
\end{align}
\end{lemma} 

 \begin{proof}
Suppose first that equation \cref{bddsolution} holds for all $t \in \RSet$ for some
function $x \in BC^0(\RSet,\XSet)$. Then for arbitrary $r \in \RSet$, it follows from
\cref{lem:process-props}, part~1, that, for all $t \geq r$,
\[
x(t)  = T_0(t,r; \tilde{\psi}) x(r) + \int_r^t T_0(t,s;\tilde{\psi}) R_0(x(s),\tilde{\psi}(s)) \, ds.
\]
We derive now by Leibniz' integral rule for $t > r$ that $\dot{x}(t)$ exists and is given by
\begin{align*}
\dot{x}(t) &= A_0(\tilde{\psi}(t)) T_0(t,r;\tilde{\psi})  x(r)+ \partial_t \int_r^t T_0(t,s;\tilde{\psi})  R_0(x(s),\tilde{\psi}(s)) \, ds  \\
&= A_0(\tilde{\psi}(t))  \left( T_0(t,r; \tilde{\psi})  x(r) + \int_r^t T_0(t,s;\tilde{\psi})  R_0(x(s),\tilde{\psi}(s)) \, ds \right) + R_0(x(t),\tilde{\psi}(t))  \\
&= A_0(\tilde{\psi}(t)) x(t) + R_0(x(t),\tilde{\psi}(t)) = F(x(t),\tilde{\psi}(t)).
\end{align*}
As $r \in \RSet$ was picked arbitrarily, we conclude thereby that $x  \in BC^0(\RSet,\XSet)$ satisfies the ODE \cref{fastODE} for all time.

Now suppose that $x  \in BC^0(\RSet,\XSet)$ satisfies the ODE \cref{fastODE}. Then we must have for all $s \in \RSet$ that
\begin{align*}
\dot{x}(s) & = A_0(\tilde{\psi}(s)) x(s) + R_0(x(s),\tilde{\psi}(s)).
\end{align*}
Pick $r \in \RSet$ arbitrarily. By reordering terms and applying an integrating factor we derive using
\cref{lem:process-props}, part~4, that, for all $t \geq r$, 
\[
x(t)  = T_0(t,r; \tilde{\psi}) x(r) +  \int_r^t T_0(t,s; \tilde{\psi}) R_0(x(s),\tilde{\psi}(s)) \, ds.
\]
As $x$ is a bounded function and $|T_0(t,r; \tilde{\psi}) x(t) | \leq K e^{-\mu(t-r)} \| x \|_{\infty} $, it must hold as $r \rightarrow -\infty$ that
\[
x(t) = \int_{-\infty}^t T_0(t,s;\tilde{\psi}) R_0(x(s),\tilde{\psi}(s)) \, ds.
\]
 \end{proof}
 
 The integral condition from \cref{lem:bdd-sol-eqv} can now be used to prove the
 existence and uniqueness of a bounded solution to the ODE \cref{fastODE}.

\begin{lemma}\label{lem:bdd-sol-uniqueness}
Assume (H2).  Let $\tilde{\psi} \in C^0(\RSet,\overline{V})$ and suppose that for some $\mu > 0$ and  $K \geq 1$ the process $\{ T_0(t,s,\tilde{\psi}): t \geq s \}$ associated to $A(\tilde{\psi}(t))$ satisfies the property 
\[
| T_0(t,s; \tilde{\psi}) \xi | \leq K e^{-\mu(t-s)} |\xi| \quad \text{for } \xi \in \XSet, t \geq s.
\]
Then there exists a unique function $\phi(\blank; \tilde{\psi}) \in BC^0(\RSet, \XSet)$ that solves the ODE
\[
\dot{x}(t) = F(x(t),\tilde{\psi}(t)).
\]
\end{lemma} 

\begin{proof}
Pick $\tilde{\psi} \in C^0(\RSet,\RSet^n)$ arbitrarily. By \cref{lem:bdd-sol-eqv}, a
function $\phi \in BC^0(\RSet, \XSet)$ solves the ODE if and only if it is a fixed
point of the map $\Phi$ defined by
\begin{align*}
\Phi(\phi; \tilde{\psi})(t) := \int_{-\infty}^t T_0(t,s; \tilde{\psi}) R_0(\phi(s), \tilde{\psi}(s)) \, ds.
\end{align*}
for all $t \in \RSet$. Therefore it is sufficient to show that there exists a unique function $\phi \in BC^0(\RSet, \XSet)$ such that $\Phi(\phi; \tilde{\psi}) = \phi$.  

Observe that by (H1) and (H2)
\[
|\Phi(\phi; \tilde{\psi}) (t) | \leq \frac{K M_0}{\mu} < \infty,
\]
so this map is well-defined on $BC^0(\RSet, \XSet)$, and for $\phi_2, \phi_1 \in BC^0(\RSet, \XSet)$ we have
\begin{align*}
| \Phi(\phi_2; \tilde{\psi}) (t) -  \Phi(\phi_1; \tilde{\psi}) (t)|  &\leq \frac{K M_1^x}{\mu} \|\phi_2 - \phi_1\|_{\infty}.
\end{align*}
So the map $\Phi(\blank, \tilde{\psi})$ is a contraction on $BC^0(\RSet, \XSet)$ for $0 < M_1^x < \mu / K$, and therefore has a unique fixed point in $BC^0(\RSet, \XSet)$. 
\end{proof}

For fixed $\eta \in \overline{V}$ and $\sigma \in BC^{0,1}(\overline{V}, \XSet)$, we will from here on also denote the process $T_0(t,s; \psi(\cdot; \eta, \sigma))$ as $T_0(t,s; \eta, \sigma)$ for convenience. The following Lemma treats the existence and uniqueness of a bounded solution to the ODE $\dot{x}(t) = F(x(t),\psi(t;\eta,\sigma))$, and proves some properties of this solution that we use later on.
 
\begin{lemma}\label{lem:bdd-sol-fast-subsystem}
Assume (H1)-(H3). For each $\eta \in \overline{V}$ and $\sigma \in BC^{0,1}(\overline{V},\XSet)$, there exists a unique function $\phi(\blank; \eta, \sigma) \in BC^0(\RSet, \XSet)$ that solves the ODE
\[
\dot{x}(t) = F(x(t),\psi(t;\eta,\sigma)).
\]
This solution has the following properties;
\begin{enumerate}
\item For each $\eta \in \overline{V}$ and $\sigma \in BC^{0,1}(\overline{V},\XSet)$, it holds for all $t \in \RSet$ that
\begin{align*}
\phi(t;\eta,\sigma) &=  \int_{-\infty}^t T_0(t,s; \eta, \sigma) R_0(\phi(s;\eta,\sigma),\psi(s;\eta,\sigma)) \, ds;
\end{align*}
\item If $\eta_i \in \overline{V}$ for $i \in \{1,2\}$ and  $\sigma \in BC^{0,1}(\overline{V}, \XSet)$, then for any $\delta \geq L_{\sigma}$ 
\begin{align*}
\sup_{t \in \RSet} & \left\{ e^{-N_1 (\delta + 1) |t|} |\phi(t;\eta_2,\sigma)  - \phi(t;\eta_1,\sigma) |   \right\} \\
& \leq \frac{K M_1^y}{\mu-K M_1^x-N_1(\delta +1)}  |\eta_2 - \eta_1| ;
\end{align*}
\item If $\eta \in \overline{V}$ and $\sigma_i \in BC^{0,1}(\overline{V},\XSet)$ for $i \in \{1,2\}$, then for any $\delta \geq L_{\sigma_1}$ 
\begin{align*}
\sup_{t \in \RSet} & \left\{ e^{-N_1 (\delta + 1) |t|}  |\phi(t;\eta,\sigma_2)  - \phi(t;\eta,\sigma_1)|  \right\} \\
&\qquad \leq \frac{K M_1^y}{(\delta+1)(\mu-K M_1^x-N_1(\delta +1))} \| \sigma_2 - \sigma_1 \|_{\infty} .
\end{align*}
\end{enumerate}
\end{lemma} 

\begin{proof}
\begin{enumerate}
\item To prove the first statement, apply
  \cref{lem:bdd-sol-eqv,lem:bdd-sol-uniqueness} with $\tilde{\psi}(\cdot) =
  \psi(\blank ; \eta, \sigma)$. 
\item For the second statement, let $\eta_i \in \overline{V}$, $i \in \{1,2\}$ and $\sigma \in BC^{0,1}(\overline{V},\XSet)$. Set $x_i(t) := \phi(t;\eta_i,\sigma)$ and $y_i(t) := \psi(t;\eta_i,\sigma)$. Let now $u(t):=x_2(t)-x_1(t)$ and consider bounded solutions to the ODE
\begin{align*}
\dot{u}(t) &= F(x_2(t),y_2(t)) - F(x_1(t),y_1(t)) \\
&= A_0(y_1(t)) u(t) + ( F(x_2(t),y_2(t)) - F(x_1(t),y_1(t)) ) \\
&\qquad \qquad \qquad \qquad \qquad \qquad \qquad - A_0(y_1(t)) (x_2(t)-x_1(t)) \\
&= A_0(y_1(t)) u(t) + (R_0(x_2(t),y_1(t)) - R_0(x_1(t),y_1(t)) ) \\
&\qquad \qquad \qquad \qquad \qquad \qquad \qquad +  (F(x_2(t),y_2(t)) - F(x_2(t),y_1(t)) ). 
\end{align*}

Then by arguments similar to the proof of \cref{lem:bdd-sol-eqv}, we have
\begin{equation}\label{eq:bdd-sol-fast-estimate}
\begin{aligned}
u(t) &=  \int_{-\infty}^t T_0(t,s;\eta_1,\sigma)  \bigg( R_0(x_2(s),y_1(s)) - R_0(x_1(s),y_1(s))  \\
&\qquad \qquad \qquad \qquad \qquad +  F(x_2(s),y_2(s)) - F(x_2(s),y_1(s)) \bigg)  \, ds.
\end{aligned}
\end{equation}
We can now derive by using (H1) and (H2) the estimate
\begin{align*}
|u(t)| &\leq K \int_{-\infty}^t e^{-\mu(t-s)}  \big( | R_0(x_2(s),y_1(s)) - R_0(x_1(s),y_1(s))  | \\
&\qquad \qquad \qquad \qquad + | F(x_2(s),y_2(s)) - F(x_2(s),y_1(s)) | \big)  \, ds \\
&\leq K \int_{-\infty}^t e^{-\mu(t-s)} \left( M_1^x |u(s)| + M_1^y |y_2(s) - y_1(s)| \right)  \, ds.
\end{align*}
An application of \cref{lem:sol-slow-subsystem} gives for $t \in \RSet$:
\begin{align*}
e^{-N_1 (\delta+ 1) |t|} |u(t)| &\leq \frac{K M_1^x}{\mu-N_1(\delta +1)} \sup_{t \in \RSet} \left\{ e^{-N_1 (\delta+ 1) |t|} |u(t)|   \right\} \\
&\qquad \qquad +\frac{K M_1^y |\eta_2-\eta_1|}{\mu-N_1(\delta+1)}.
\end{align*}
As above inequality holds for all $t \in \RSet$, by reordering terms we obtain
\[
\sup_{t \in \RSet} \left\{ e^{-N_1 (\delta + 1) |t|} |x_2(t) - x_1(t)|   \right\} \leq \frac{K M_1^y |\eta_2 - \eta_1|}{\mu-K M_1^x-N_1(\delta +1)}.
\]
\item The proof of the third statement is similar to that of the second one, and therefore we will be more brief here. Let $\eta \in \overline{V}$ and $\sigma_i \in BC^{0,1}(\overline{V},\XSet)$ for $i \in \{1,2\}$. Set $x_i(t) := \phi(t;\eta,\sigma_i)$ and $y_i(t):= \psi(t;\eta,\sigma_i)$, let $u(t) := x_2(t) - x_1(t)$ and consider bounded solutions to the ODE
\begin{align*}
\dot{u}(t) &= F(x_2(t),y_2(t)) - F(x_1(t),y_1(t)) \\
&= A_0(y_1(t)) u(t) + (R_0(x_2(t),y_1(t)) - R_0(x_1(t),y_1(t)) ) \\
&\qquad \qquad \qquad \qquad \qquad \qquad \qquad +  (F(x_2(t),y_2(t)) - F(x_2(t),y_1(t)) ). 
\end{align*}
We derive by (H2) the estimate 
\begin{align*}
|u(t)| &\leq  K \int_{-\infty}^t e^{-\mu(t-s)} \left( M_1^x |u(s)| + M_1^y |y_2(s) - y_1(s)| \right)  \, ds.
\end{align*}
An application of \cref{lem:sol-slow-subsystem} gives for $t \in \RSet$ and any $\delta \geq L_{\sigma_1}$ :
\begin{align*}
e^{-N_1 (\delta + 1) |t|} |u(t)| &\leq \frac{K M_1^x}{\mu-N_1(\delta+1)} \sup_{t \in \RSet} \left\{ e^{-N_1 (\delta + 1) |t|} |u(t)|   \right\} \\
&\qquad  \qquad +\frac{K M_1^y \|\sigma_2-\sigma_1\|_{\infty}}{(\delta +1)(\mu-N_1(\delta +1))}.
\end{align*}
As above inequality holds for all $t \leq 0$, by reordering terms we obtain
\[
\sup_{t \in \RSet} \left\{ e^{-N_1 (\delta+ 1) |t|} |x_2(t) - x_1(t)|   \right\} \leq \frac{K M_1^y  \|\sigma_2 - \sigma_1\|_{\infty}}{(\delta+1)(\mu-K M_1^x-N_1(\delta+1))}.
\]
\end{enumerate}
\end{proof}

\begin{remark}\label{rem:bdd-sol-fast-subsystem}
The proof approach for items 2 and 3 here deviates from that described in Lemma 9.1.7 of Henry \cite{He81}. Translated to our setting, Henry begins there by applying the Mean Value Theorem in integral form to the right hand side of the differential equation
\[
\partial_t (x_2(t) - x_1(t)) = F(x_2(t),y_2(t)) - F(x_1(t),y_1(t)).
\]
This requires the differentiability of $F$ such that $DF$ is uniformly bounded in the operator norm, which is not necessary to assume in our approach.
\end{remark}

\Cref{lem:bdd-sol-fast-subsystem} completes step 2 of the contraction mapping construction, and we are thus ready to move on to step 3. For this, we define for arbitrary $\delta > 0$ the set
 \begin{equation}
 \mathcal{B}_{\delta} := \left\{ \sigma \in BC^{0,1}(\overline{V},\XSet): \| \sigma \|_{\infty} \leq \frac{K M_0}{\mu}, \, L_{\sigma} \leq  \delta \right\}.
 \end{equation}
 It follows from \cref{closedsubset} that $\mathcal{B}_{\delta}$ is complete when equipped with the supremum norm. We now define an operator $ \Lambda: \mathcal{B}_{\delta} \rightarrow \mathcal{B}_{\delta}$ via 
 \begin{equation}\label{fast-contraction-map}
\Lambda(\sigma)(\eta) := \phi(0;\eta,\sigma),
 \end{equation}
 and proceed to show that $\Lambda$ is a contraction on $(\mathcal{B}_{\delta}, \| \blank \|_{\infty})$ for a suitable choice of $\delta$.

 \begin{lemma}\label{lem:Lambda-well-defined}
Assume (H1)-(H3). Suppose that $K M_1^x+N_1(\delta +1) < \mu$ as well as $\frac{K M_1^y}{\mu-K M_1^x-N_1(\delta +1)}  \leq \delta$. 
Then the operator $\Lambda$ is well-defined.
 \end{lemma}
 
 \begin{proof}
Pick $\sigma \in \mathcal{B}_{\delta}$. We have for every $\eta \in \overline{V}$ that
\begin{align*}
|\Lambda(\sigma)(\eta)| &\leq \int_{-\infty}^0 | T_0(0,s;\eta,\sigma) R_0(\phi(s;\eta,\sigma),\psi(s;\eta,\sigma))| \, ds \\
&\leq K M_0 \int_{-\infty}^0 e^{\mu s} \, ds = \frac{K M_0}{\mu}.
\end{align*}
Now suppose that $\eta_1,\eta_2 \in \overline{V}$. We see from
\cref{lem:bdd-sol-fast-subsystem} that
\begin{align*}
|\Lambda(\sigma)(\eta_2) - \Lambda(\sigma)(\eta_1)| &\leq \frac{K M_1^y}{\mu-K M_1^x-N_1(\delta +1)} |\eta_2 - \eta_1| \leq \delta |\eta_2 - \eta_1|.
\end{align*} 
Combining these estimates we obtain that $\Lambda( \mathcal{B}_{\delta}) \subseteq \mathcal{B}_{\delta}$.
 \end{proof}

 \begin{lemma}\label{lem:Lambda-contraction}
Suppose that the assumptions of \cref{lem:Lambda-well-defined} are satisfied. Then
the operator $\Lambda$ is a contraction on $(\mathcal{B}_{\delta}, \| \blank
\|_{\infty})$.
 \end{lemma}
 
 \begin{proof}
Let $\sigma_1, \sigma_2 \in \mathcal{B}_{\delta}$ and $\eta \in \overline{V}$.
Then we see from \cref{lem:bdd-sol-fast-subsystem} that
\begin{align*}
| \Lambda(\sigma_2)(\eta) - \Lambda(\sigma_1)(\eta) | \leq \frac{K M_1^y}{(\delta+1)(\mu-K M_1^x-N_1(\delta +1))} \|\sigma_2 - \sigma_1\|_{\infty}.
\end{align*}
As $\frac{K M_1^y}{\mu-K M_1^x-N_1(\delta +1)} \leq \delta$ by assumption, we have
\[
\frac{K M_1^y}{(\delta+1)(\mu-K M_1^x-N_1(\delta +1))} \leq \frac{\delta}{\delta+1} < 1,
\]  
and thus the assertion follows.
 \end{proof}
 
 We can now deduce the existence of a unique fixed point $h \in \mathcal{B}_{\delta}$ for $\Lambda$, as stated in the following Corollary:
 
 \begin{corollary}\label{cor:Lambda-fp}
Assume (H1)-(H3). Suppose that $K M_1^x+N_1(\delta +1) < \mu$ as well as  $\frac{K M_1^y}{\mu-K M_1^x-N_1(\delta +1)}  \leq \delta$ are satisfied. Then the operator $\Lambda$ has a unique fixed point $h \in \mathcal{B}_{\delta}$.
 \end{corollary}
 
 \begin{proof}
 This now follows from 
 \cref{lem:Lambda-contraction} by applying the Contraction Mapping Theorem.
 \end{proof}
 
 It remains to prove that the fixed point from \cref{cor:Lambda-fp} indeed provides an
 invariant manifold. This is guaranteed by the following Theorem, that also shows
 equivalence of this fact with the condition \cref{InvMfdIntCond} which is
 comparable to statements from other texts such as \cite{Sa90} and Chapter 6 of
 \cite{He81}. 
 
 \begin{theorem}\label{thm:slow-mfd-eqv}
Assume $\sigma \in BC^{0,1}(\overline{V},\XSet)$. Under assumptions (H1)-(H3), the following statements are equivalent:
 \begin{enumerate}
\item For all $\eta \in \overline{V}$ it holds that
 \begin{equation}\label{InvMfdIntCond}
 \sigma(\eta) = \int_{-\infty}^0 T_0(0,s; \eta,\sigma) R_0(\sigma(\psi(s;\eta,\sigma)),\psi(s;\eta,\sigma)) \, ds;
 \end{equation}
 \item For all $\eta \in \overline{V}$ it holds that
 \[
\sigma(\eta) = \Lambda(\sigma)(\eta) = \phi(0; \eta, \sigma);
 \] 
 \item The function $\sigma$ parameterizes an invariant manifold of system \cref{ODE-system}.
 \end{enumerate}
 \end{theorem}
 
 \begin{proof}
\begin{enumerate}
\item[$\mathrm{(1 \Rightarrow 2)}$] Suppose first that equation \cref{InvMfdIntCond}
  holds for arbitrary $\eta \in \overline{V}$. For all $t \in \RSet$ we obtain
\begin{align*}
\sigma(\psi(t;\eta,\sigma)) &= \int_{-\infty}^0 T_0(t,s+t;\eta,\sigma) R_0(\sigma(\psi(s+t;\eta,\sigma)),\psi(s+t;\eta,\sigma)) \, ds \\
&=  \int_{-\infty}^t T_0(t,s;\eta,\sigma)
R_0(\sigma(\psi(s;\eta,\sigma)),\psi(s;\eta,\sigma)) \, ds,
\end{align*}
where in the second passage we have used a change of variables under the integral.
Recalling that the function $\phi(\cdot;\eta,\sigma)$ was previously defined as the
unique bounded solution to the ODE $\dot{x}(t) = F(x(t),\psi(t;\eta,\sigma))$, and
using \cref{lem:bdd-sol-eqv,lem:bdd-sol-fast-subsystem} we obtain
 \[
\Lambda(\sigma)(\eta) = \phi(0; \eta, \sigma) = \sigma(\eta). \\
 \] 
\item[$\mathrm{(2 \Rightarrow 3)}$]    Secondly, assume that $\sigma(\eta) = \phi(0; \eta, \sigma)$ for arbitrary $\eta \in \overline{V}$. 
Then we readily check for all $t \in \RSet$ that
 \[
\sigma(\psi(t;\eta,\sigma)) = \phi(0; \psi(t;\eta,\sigma),\sigma) = \phi(t;\eta,\sigma).
 \]
Therefore, for any $r \in \RSet$, if we set $x(t) = \sigma(\psi(t-r;\eta,\sigma))$
and $y(t) = \psi(t-r;\eta,\sigma)$, we have that $(x(t), y(t))$ is the solution to
the system \cref{ODE-system} for initial condition $(x(r),y(r)) =
(\eta,\sigma(\eta))$, and this solution clearly satisfies the property $x(t) =
\sigma(y(t))$ for all $t \geq r$. As $\eta \in \overline{V}$ was picked arbitrarily,
this implies that $\sigma$ parameterizes an invariant manifold as was defined in \cref{def:slow-mfd}. \\
\item[$\mathrm{(3 \Rightarrow 1)}$]  Thirdly, suppose that for arbitrary $r \in
  \RSet$ and $\eta \in \overline{V}$, the solution to the fast-slow ODE
  \cref{ODE-system}  with initial condition $(x(r),y(r))=(\sigma(\eta),\eta)$
  satisfies $x(t) = \sigma(y(t))$ for all $t \geq r$. 
  Owing to the boundedness of $\sigma$, if $(\tilde{x},\tilde{y})(\cdot)$ solves
  system \cref{ODE-system} with $(\tilde{x},\tilde{y})(0) = (\sigma(\eta),\eta)$,
  then $(\sigma \circ \tilde{y})(\cdot)$ is a bounded solution to the ODE
\[
\dot{x} = F(x(t), \tilde{y}(t)),
\]
with $\tilde{y}(\cdot) = \psi(\cdot; \eta,\sigma)$. An application of 
\cref{lem:bdd-sol-eqv} with $t=0$ now gives that
\[
 \sigma(\eta) = \int_{-\infty}^0 T_0(0,s; \eta,\sigma) R_0(\sigma(\psi(s;\eta,\sigma)),\psi(s;\eta,\sigma)) \, ds.
\]
\end{enumerate}
 \end{proof}
 
 We are now able to prove the existence of an invariant manifold for system \cref{ODE-system} parameterized
 globally over the $y$-variable:
 
 \begin{proof}[Proof of \cref{thm:slow-mfd-existence}]
By \cref{cor:Lambda-fp} the operator $\Lambda$ has a unique fixed point 
\[
h \in \mathcal{B}_{\delta} = \left\{ \sigma \in BC^{0,1}(\overline{V},\XSet) : \|\sigma\|_{\infty} \leq \frac{K M_0}{\mu}, \, L_{\sigma} \leq  \delta \right\},
\] and then by \cref{thm:slow-mfd-eqv} this function $h$ parameterizes an invariant manifold. 

For the second property of \cref{thm:slow-mfd-existence}, suppose that
$(x_i,y_i) \in C^1(\RSet, \XSet \times \overline{V})$, with $i \in \{1,2 \}$, are two
solutions such that $y_i(0)=\eta$ (for arbitrary pick of $\eta \in \overline{V}$) and 
\[
\sup_{t \in \RSet} |x_i(t)| < \infty.
\]
As we have
\[
\dot{y}_i(t) = g(x_i(t), y_i(t)),
\]
similar to the proof of \cref{lem:sol-slow-subsystem}, item 2 we can obtain, after integration, for all $t \in \RSet$ the inequality 
\begin{align*}
|y_2(t)-y_1(t)| &\leq   \sgn(t) N_1  \int_0^t |x_2(s)- x_1(s)| + |y_2(s) - y_1(s)| \, ds \\
&\leq \frac{1}{\delta + 1} \|x_2 - x_1 \|_{-(\delta + 1)N_1, \infty} \left(e^{(\delta+1)N_1 |t|} - 1 \right) \\
&\qquad \qquad \qquad \qquad \qquad + \sgn(t) N_1 \int_0^t |y_2(s) - y_1(s)| \, ds,
\end{align*}
where
\[
\| x_2 - x_1 \|_{-(\delta + 1) N_1,\infty} = \sup_{t \in \RSet} \left\{ e^{-(\delta + 1) N_1|t|} |x_2(t) - x_1(t)|  \right\}.
\]
Now an application of Gr\"onwall's inequality gives for $t \in \RSet$ the estimate 
\begin{align*}
	|y_2(t) - y_1(t) | \leq \frac{1}{\delta} \| x_2 - x_1 \|_{-(\delta + 1) N_1,\infty} \, e^{(\delta+1)N_1|t|}.					
\end{align*}

Analogously to the proof of \cref{lem:bdd-sol-fast-subsystem}, item 3 we may then furthermore obtain
\[
\| x_2 - x_1 \|_{-(\delta + 1) N_1,\infty}  \leq \frac{1}{\delta} \cdot \frac{K M_1^y}{\mu-KM_1^x-N_1(\delta + 1)} \|x_2 - x_1 \|_{-(\delta + 1) N_1,\infty}.
\]
As by assumption we have 
\[
\frac{K M_1^y}{\mu - KM_1^x - N_1(\delta + 1)} < \delta,
\] 
this implies we must have $x_2(0) = x_1(0)$. So for any $\eta \in \overline{V}$, there is a unique solution $(x,y) \in C^1(\RSet, \XSet \times \overline{V})$ that satisfies both $y(0) = \eta$ and 
\[
\sup_{t \in \RSet} |x(t)| < \infty.
\]

Thereby we conclude that if $(x,y)  \in C^1(\RSet, \XSet \times \overline{V})$ satisfies the condition
\[
\sup_{t \in \RSet} |x(t)| < \infty,
\]
it must hold that $x(t) = h(\psi(t; y(0), h))$, and therefore $(x,y)$ must lie on $S_h$. 
 \end{proof}


\section{Smoothness of invariant manifolds}\label{sec:slow-mfds-smoothness}
The purpose of this Section is to show that, under the additional assumptions
(H2)$^\prime$-(H3)$^\prime$ below on system \cref{ODE-system}, the invariant manifold
parameterization $h$ from \cref{thm:slow-mfd-existence} is continuously
differentiable with a uniformly bounded derivative. We let $U$ be some open set in
$\XSet$ that contains $h(\overline{V})$. Then the new assumptions that are made in
addition to (H2)-(H3) from \cref{sec:slow-mfds-existence}, can be stated as
follows:

\vspace{0.25cm}

 \begin{enumerate}
 \item[$\mathrm{(H2)'}$] $F \in C^1(U \times \overline{V}, \XSet)$, $D F \in BC^0(U \times \overline{V}, L(\XSet \times \RSet^n, \XSet))$ such that
  \[ 
 \sup\left\{\| D_x R_0(x,y) \| : (x,y) \in U \times \overline{V} \right\} \leq M^x_1,
 \]
 and
 \[
 \sup\left\{\| D_y F(x,y) \| : (x,y) \in U \times \overline{V} \right\} \leq M^y_1;
 \]
 \item[$\mathrm{(H3)'}$] $g \in C^1(U \times \overline{V}, \RSet^n)$ such that
 \[
 \sup\left\{\| D g(x,y) \|: (x,y) \in U \times \overline{V} \right \}  \leq N_1.
 \]
 \end{enumerate}
 
 \vspace{0.25cm}
 
 The assumptions (H2)$^\prime$-(H3)$^\prime$ lead to the following result:
 
 \begin{theorem}[Smoothness of invariant manifolds]\label{thm:slow-mfd-smoothness}
Suppose that the hypotheses of \cref{thm:slow-mfd-existence} hold and additionally
the assumptions (H2)$^\prime$-(H3)$^\prime$ are satisfied. Then for the function $h
\in BC^{0,1}(\overline{V},\XSet)$ as described in \cref{thm:slow-mfd-existence} it
holds moreover that $h \in BC^{1}(\overline{V},\XSet)$.
 \end{theorem}
 
 We will henceforth assume in this Section that the conditions for
 \cref{thm:slow-mfd-existence} are satisfied, and denote $y(t; \eta) := \psi(t; \eta,
 h)$ and $x(t; \eta) := h(y(t;\eta))$. Now in order to set up a proof for 
 \cref{thm:slow-mfd-smoothness}, first suppose that $h$ is indeed contained in
 $BC^1(\overline{V},\XSet)$.  We then have for $w(t;\eta) := \partial_\eta x(t;
 \eta)$ that
 \begin{align}
 \begin{split}\label{eq:w-ODE}
 \partial_t w(t; \eta) &= \partial_{\eta} F(x(t;\eta),y(t;\eta)) \\
 &= D_x F(x(t;\eta),y(t;\eta))  w(t;\eta) + D_y F(x(t;\eta),y(t;\eta))  \partial_{\eta} y(t;\eta); \\
 w(0; \eta) &= D h(\eta).
 \end{split}
 \end{align}
  
Now let $T_h(t,s; \eta)$ be the process associated to the non-autonomous linear ODE
\begin{align}
\begin{split}\label{eq:Th-ODE}
\partial_t \tilde{w}(t; \eta) &= D_x F(x(t;\eta),y(t;\eta))  \, \tilde{w}(t;\eta) \\
&=D_x F(h(y(t;\eta)),y(t;\eta)) \, \tilde{w}(t;\eta),
\end{split}
\end{align}
where $\tilde{w}$ takes on values in either $\XSet$ or $L(\RSet^n,\XSet)$. We have the following exponential bound on this process:
\begin{lemma}\label{lem:Th-props}
  Let $T_h(t,s;\eta)$ be the process associated to \cref{eq:Th-ODE}, interpreted as an
  ODE on either $\XSet$ or $L(\RSet^n,\XSet)$. For every $\eta \in \overline{V}$ and
  $t \geq s$ it holds that
\[
\|T_h(t,s;\eta) \| \leq K e^{-(\mu-K M^x_1)(t-s)},
\]
where $\mu > 0$ and $K \geq 1$ are the constants from hypothesis (H1).
\end{lemma}

\begin{proof}
  The proof runs indifferently for the case $\XSet$ and $L(\RSet^n,\XSet)$, and the
  estimates below are intended to be in the corresponding norms.
Pick $\eta \in \overline{V}$ arbitrarily. Observe that
\[
D_x F(x(t;\eta),y(t;\eta)) = A_0(y(t;\eta)) + D_x R_0(x(t;\eta),y(t;\eta)).
\]
Therefore if $\tilde{w}(t; \eta)$ is a solution of \cref{eq:Th-ODE}, with $t \geq s$ for any fixed $s \in \RSet$, we have by the variation of constants formula that:
\[
\tilde{w}(t; \eta) = T_0(t,s; \eta,h) \, \tilde{w}(s; \eta) + \int_s^t T_0(t,r; \eta, h) D_x R_0(x(r;\eta),y(r;\eta)) \, \tilde{w}(r; \eta) \, dr,
\]
which implies
\[
e^{\mu t} \| \tilde{w}(t; \eta) \| \leq K e^{\mu s} \|\tilde{w}(s; \eta) \| + K M^x_1 \int_s^t  e^{\mu r} \|\tilde{w}(r; \eta) \| \, dr.
\]
An application of Gr\"onwall's inequality now gives
\[
e^{\mu t} \| \tilde{w}(t; \eta) \| \leq K e^{\mu s} \| \tilde{w}(s; \eta) \| e^{K M^x_1 (t-s)},
\]
which implies the Lemma.
\end{proof}

Now returning to the initial value problem \cref{eq:w-ODE}, by the variation of constants formula we obtain for $t \geq s$ that
\[
w(t;\eta) = T_h(t,s;\eta) \, w(s;\eta) + \int_s^t T_h(t,r;\eta) D_y F(h(y(r;\eta)),y(r;\eta)) \, \partial_{\eta} y(r;\eta) \, dr.
\]
Observe that $w(t;\eta) = D h(y(t;\eta)) \partial_{\eta} y(t;\eta)$, and
 \begin{align*}
 \partial_t \partial_{\eta} y(t;\eta) &=  D g(h(y(t;\eta)), y(t;\eta)) \begin{pmatrix} Dh(y(t;\eta)) \\ I_{\RSet^n} \end{pmatrix} \partial_{\eta} y(t;\eta); \\
 \partial_{\eta} y(0;\eta) &=  I_{\RSet^n},
 \end{align*}
 thus we see for $t \geq s$
 \[
\| \partial_{\eta} y(s;\eta) \| \leq e^{N_1(\| Dh \|_{\infty} +1) (t-s)} \| \partial_{\eta}y(t;\eta) \|,
 \]
 and thereby, by applying the result from \cref{lem:Th-props}:
 \[
 \| T_h(t,s;\eta) w(s;\eta) \| \leq K e^{-(\mu-KM^x_1-N_1(\| Dh \|_{\infty} +1))(t-s)} \| Dh \|_{\infty} \| \partial_{\eta}y(t;\eta) \|.
 \]
 Under the assumption $\mu-KM^x_1-N_1(\| D h \|_{\infty} +1)> 0$, boundedness of $D h$ implies then
\begin{equation}\label{eq:LP-formula-w}
w(t;\eta) = \int_{-\infty}^t T_h(t,s;\eta) D_y F(h(y(s;\eta)),y(s;\eta)) \, \partial_{\eta} y(s; \eta) \, ds,
\end{equation}
and moreover $\sup_{t \in \RSet} e^{-N_1(\|Dh\|_{\infty}+1)|t|} \| w(t ;\eta) \| < \infty$. As we have $w(0;\eta) = Dh(\eta)$ in  \cref{eq:w-ODE}, setting $t=0$ in this last integral equation gives
\[
Dh(\eta) = \int_{-\infty}^0 T_h(0,s;\eta) D_y F(h(y(s;\eta)),y(s;\eta)) \, \partial_{\eta} y(s;\eta) \, ds.
\]

\begin{remark}\label{rem:w-unique-bdd-sol}
Formula \cref{eq:LP-formula-w} tells us that, with $\gamma := -N_1(\| Dh \|_{\infty} + 1)$, $w(\blank; \eta)$ is the unique $\gamma$-bounded solution to the ODE
 \begin{align*}
 \partial_t \tilde{w}(t; \eta) &= \partial_{\eta} F(x(t;\eta),y(t;\eta)) \\
 &= D_x F(x(t;\eta),y(t;\eta)) \,  \tilde{w}(t;\eta) + D_y F(x(t;\eta),y(t;\eta)) \,  \tilde{z}(t), 
 \end{align*}
where $\tilde{w}$ takes on values in $L(\RSet^n,\XSet)$, and where $\tilde{z} = \partial_{\eta} y(\blank;\eta)$. We can see this by defining a mapping
 \begin{align*}
\tilde{\Gamma}_{\eta}: BC^0_{\gamma}(\RSet, L(\RSet^n,\RSet^n)) &\rightarrow  BC_{\gamma}^0(\RSet, L(\RSet^n,\XSet)), \\
 \tilde{z}&\mapsto \left( t \mapsto  \int_{-\infty}^t T_h(t,s;\eta) D_y F(h(y(s;\eta)),y(s;\eta)) \, \tilde{z}(s) \, ds \right),
 \end{align*}
 where  
 \[
 BC^0_{\gamma}(\RSet, L(\RSet^n,\RSet^n)) = \{ \tilde{z} \in C^0(\RSet,
 L(\RSet^n,\RSet^n)) \colon \| \tilde z \|_{\gamma,\infty }:= \sup_{t \in \RSet}
e^{\gamma |t|} \| \tilde{z}(t) \|< \infty \},
 \]
 which is a Banach space when equipped with the norm $\|  \blank \|_{\gamma,\infty
 }$. Then $\tilde{\Gamma}_{\eta}(\tilde{z})$ is the unique $\gamma$-bounded solution to the ODE 
  \begin{align*}
 \partial_t \tilde{w}(t; \eta) &= \partial_{\eta} F(x(t;\eta),y(t;\eta)) \\
 &= D_x F(x(t;\eta),y(t;\eta)) \,  \tilde{w}(t;\eta) + D_y F(x(t;\eta),y(t;\eta)) \,  \tilde{z}(t)
 \end{align*}
 for any $\tilde{z} \in  BC^0_{\gamma}(\RSet, L(\RSet^n,\RSet^n))$. The proof is
 similar to the proof of item 2 of \cref{thm:slow-mfd-existence} from \cref{sec:slow-mfds-existence}. 
 
This allows us to interpret the construction of the map $h^1$ in step 2 below in a
similar fashion to how the construction of the mapping $\Lambda$ from
\cref{sec:slow-mfds-existence} was interpreted through the commutative diagram from
\cref{fig:com-dia-slow-mfd}. For $h^1$ the corresponding diagram is illustrated in \cref{fig:com-dia-slow-mfd-var}.
\end{remark}

This motivates us to take the following approach to proving differentiability of $h$; 
 \begin{itemize}
 \vspace{0.2cm}
 \item[\underline{\bf{Step 1}}]  Firstly, for $\tilde{w} \in BC^0_{\gamma}(\RSet, L(\RSet^n, \XSet))$ with $\gamma := -N_1(\rho+1)$ for some $\rho > 0$, let $z(t;\eta,\tilde{w})$ denote the unique solution to the IVP
\begin{equation}\label{eq:var-eqn-z}
 \begin{aligned}
 \partial_t z(t;\eta,\tilde{w}) &=  D g(h(y(t;\eta)), y(t;\eta)) \begin{pmatrix} \tilde{w}(t)  \\ z(t;\eta, \tilde{w}) \end{pmatrix} ; \\
z(0;\eta, \tilde{w}) &=  I_{\RSet^n}.
 \end{aligned}
 \end{equation}
 We analyse some properties of $z(t;\eta,\tilde{w})$ for this step; 
  \item[\underline{\bf{Step 2}}]  Next we define a map $\Gamma_{\eta}: BC^0_{\gamma}(\RSet, L(\RSet^n, \XSet)) \rightarrow BC^0_{\gamma}(\RSet, L(\RSet^n, \XSet))$, for arbitrary $\eta \in \overline{V}$, by setting
\begin{equation}\label{eq:map-Gamma}
 \Gamma_{\eta}(\tilde{w})(t) := \int_{-\infty}^t T_h(t,s;\eta) \, D_y F(h(y(s;\eta)),y(s;\eta)) \, z(s;\eta,\tilde{w}) \, ds.
 \end{equation}
By making use of the results from step 1, we provide conditions under which $\Gamma_{\eta}$ is a well-defined mapping on $BC^0_{\gamma}(\RSet, L(\RSet^n, \XSet))$ with a unique fixed point $w(\blank;\eta)$. We then define $h^1(\eta) := w(0; \eta) = \Gamma_{\eta}(w(\blank; \eta))(0)$ and show moreover that $h^1 \in BC^{0}(\overline{V}, L(\RSet^n,\XSet))$ (i.e. $h^1$ is continuous and uniformly bounded);
  \item[\underline{\bf{Step 3}}]  Lastly, we show that for arbitrary $\eta \in \overline{V}$ it holds that $h^1(\eta) = Dh(\eta)$ and therefore indeed $h \in BC^1(\overline{V},\XSet)$.
 \end{itemize}
  \vspace{0.2cm}
  
 \begin{figure}\label{fig:com-dia-slow-mfd-var}\centering
\begin{tikzcd}[sep = 30mm, every label/.append style = {font = \normalfont}]
\eta \arrow[symbol = \in]{d} & (\eta,\tilde{z}) \arrow[symbol = \in]{d} \\ [-75pt]
 \overline{V} \arrow{r}{\eta \times z(\blank; \eta, Dh)} \arrow[swap]{rd}{\partial_{\eta}x (\blank; \eta)} & \overline{V} \times BC^0_{\gamma}(\RSet, L(\RSet^n,\RSet^n)) \arrow{d}{\tilde{\Gamma}_{\eta}(\tilde{z})}  \\
&  BC_{\gamma}^0(\RSet,L(\RSet^n,\XSet)) 
\end{tikzcd}
\caption{Suppose that $h \in BC^1(\overline{V}, \XSet)$ parameterizes a slow
manifold, and the mappings $z  : \overline{V} \times \{ Dh \} \rightarrow
BC_{\gamma}^0 (\RSet, L(\RSet^n, \RSet^n))$ with $\gamma = -N_1(\| Dh \|_{\infty} +
1)$, as well as $\tilde{\Gamma}_{\eta}: BC_{\gamma}^0 (\RSet, L(\RSet^n, \RSet^n))
\rightarrow BC_{\gamma}^0(\RSet,L(\RSet^n,\XSet)) $ (for arbitrary $\eta \in
\overline{V}$) are as outlined in step 1 and \cref{rem:w-unique-bdd-sol}. Then we should have for every $\eta \in \overline{V}$ and $t \in \RSet$ that $\partial_{\eta} x(t;\eta) = \tilde{\Gamma}_{\eta}(z(\blank; \eta, Dh))(t)$. The Figure displays this relationship as a commutative diagram. In particular, if $w=\partial_{\eta}x(\blank;\eta)$, then  $w = \Gamma_{\eta}(w) = \tilde{\Gamma}_{\eta}(z(\blank;\eta,w))$ and setting $t=0$ gives $Dh(\eta) = w(0) = h^1(\eta) $, with $\Gamma_{\eta}$ and $h^1$ as defined in step 2.}
\end{figure}
 
 \begin{remark}\label{rem:proof-smoothness-variational}
 There exist several approaches to proving \cref{thm:slow-mfd-smoothness}, see for example Section 1.11 of Chicone \cite{Ch06} for a short discussion of this, and Chow \& Lu \cite{ChLu88a} for another approach that utilises a Lemma provided by Henry. We opt to use the definition of the derivative here, but observe that we do not construct a candidate derivative by differentiating the contraction mapping $\Lambda$ from the previous Section. Instead we first consider the variational equation for trajectories on the invariant manifold parameterized by the fixed point $h$ of $\Lambda$. The proofs in this section were partially inspired by the proof for Chapter 13, Theorem 4.2 from the classical textbook by Coddington and Levinson \cite{CoLe55}.
  \end{remark}

To carry out step 1 of our proof approach for \cref{thm:slow-mfd-smoothness}, we now
establish the following properties for solutions $z(\blank; \eta, \tilde{w})$ to the
IVP \cref{eq:var-eqn-z}:
 \begin{lemma}\label{lem:z-properties}
Under (H3)$^\prime$, for arbitrary $\rho > 0$ we have :
 \begin{enumerate}
 \item For any $\eta \in \overline{V}$ and every $\tilde{w} \in  BC^0_{\gamma}(\RSet, L(\RSet^n, \XSet))$ it holds for all $t \in \RSet$ that
\[
 \|z(t; \eta,\tilde{w})\| \leq  \left( 1 + \frac{1}{\rho}  \|\tilde{w}\|_{-N_1 (\rho+1), \infty}  \right)  e^{N_1(\rho+1)|t|}.
\]
If moreover the inequality $\|\tilde{w}\|_{-N_1 (\rho+1), \infty} \leq \rho$ is satisfied, then we have for all $t \in \RSet$ that
\[
 \|z(t; \eta,\tilde{w})\| \leq e^{N_1(\rho+1)|t|}.
\]
 \item For all $\eta_1, \eta_2 \in  \overline{V}$ and every $\tilde{w} \in  BC^0_{\gamma}(\RSet, L(\RSet^n, \XSet))$  it holds for any $\tilde{\rho} > \rho$ that
  \[
 \sup_{t \in \RSet} \left\{ e^{-N_1(\tilde{\rho}+1)|t|}  \|z(t;\eta_2,\tilde{w}) - z(t;\eta_1,\tilde{w})\| \right\} \rightarrow 0 \quad \text{as } |\eta_2 - \eta_1| \rightarrow 0;
 \]
 \item For any $\eta \in \overline{V}$ and all $\tilde{w}_1,\tilde{w}_2 \in  BC^0_{\gamma}(\RSet, L(\RSet^n, \XSet))$ it holds for all $t \in \RSet$ that
 \[
 e^{-N_1(\rho+1)|t|}  \|z(t;\eta,\tilde{w}_2) - z(t;\eta,\tilde{w}_1)\| \leq \frac{1}{\rho} \| \tilde{w}_2 - \tilde{w}_1\|_{-N_1(\rho+1),\infty}.
 \]
 \end{enumerate}
 \end{lemma}
 
 \begin{proof}
 \begin{enumerate}
 \item The variation of constants formula gives
\begin{align*}
 \|z(t; \eta,\tilde{w})\| &\leq e^{N_1|t|} +  \sgn(t) N_1  \int_0^t e^{N_1|t-s|} \|\tilde{w}(s)\| \, ds \\
 &= \left(1-\frac{1}{\rho} \|\tilde{w}\|_{-N_1 (\rho+1), \infty} \right) e^{N_1|t|} + \frac{1}{\rho} \|\tilde{w}\|_{-N_1 (\rho+1), \infty} e^{N_1(\rho+1)|t|} \\
 &\leq \left( 1 + \frac{1}{\rho}  \|\tilde{w}\|_{-N_1 (\rho+1), \infty}  \right)  e^{N_1(\rho+1)|t|}.
\end{align*}
The equality in the second line implies that if $ \|\tilde{w}\|_{-N_1 (\rho+1), \infty} \leq \rho$, then
\[
 \|z(\blank; \eta,\tilde{w})\|_{-N_1 (\rho+1), \infty} \leq 1.
\]

\item We set $z_i(t) := z(t; \eta_i,\tilde{w})$ and $v(t) := z_2(t) - z_1(t)$. Then $v(t)$ has to satisfy the IVP
\begin{equation*}
\begin{aligned}
\dot{v}(t)  &= D_y g(h(y(t;\eta_2)), y(t;\eta_2)) v(t) \\ 
&\qquad  +  \left( Dg(h(y(t;\eta_2)), y(t;\eta_2)) - Dg(h(y(t;\eta_1)), y(t;\eta_1)) \right)  \begin{pmatrix} \tilde{w}(t) \\ z_1(t) \end{pmatrix}; \\
v(0) &=  0.
 \end{aligned}
 \end{equation*}
An application of the variation of constants formula to this IVP gives that for all $t \in \RSet$
\begin{align*}
&e^{-N_1(\rho+1)|t|} \| z_2(t) -z_1(t) \| \\
&\qquad \leq  C  \left( Dg(h(y(t;\eta_2)), y(t;\eta_2)) - Dg(h(y(t;\eta_1)), y(t;\eta_1)) \right),
\end{align*}
where $C:=(1 + \frac{\rho+1}{\rho} \| \tilde{w} \|_{-N_1 (\rho+1), \infty})$. So we obtain for $\tilde{\rho} > \rho$ that
\begin{align*}
&\sup_{t \in \RSet} e^{-N_1(\tilde{\rho}+1)|t|} \| z_2(t) -z_1(t) \| \\
&\quad \leq C \sup_{t \in \RSet} e^{- N_1(\tilde{\rho} - \rho) |t|} \|Dg(h(y(t;\eta_2)), y(t;\eta_2)) - Dg(h(y(t;\eta_1)), y(t;\eta_1)) \|.
\end{align*}
Now by continuity and boundedness of $Dg$ it holds that this last term goes to $0$ as $|\eta_2 - \eta_1| \rightarrow 0$, and thereby the claim follows.

\item We set $z_i(t) := z(t; \eta,\tilde{w}_i)$ and $v(t) := z_2(t) - z_1(t)$. Then $v(t)$ has to satisfy the IVP
\begin{equation*}
 \begin{aligned}
\dot{v}(t)  &=  D_y g(h(y(t;\eta)), y(t;\eta))  v(t) \\
&\qquad \qquad +  D_x g(h(y(t;\eta)), y(t;\eta)) \,  (w_2(t;\eta) - w_1(t; \eta)); \\
v(0) &=  0.
 \end{aligned}
 \end{equation*}
An application of the variation of constants formula to this IVP gives that for all $t \in \RSet$
\begin{align*}
&\| z_2(t) -z_1(t) \| \\
&\quad \leq \sgn(t) N_1 \int_0^t e^{N_1|t-s|} \| \tilde{w}_2(s;\eta) - \tilde{w}_1(s; \eta)  \| \, ds \\
&\quad \leq \frac{1}{\rho} \left( e^{N_1(\rho+1)|t|} - e^{N_1|t|} \right) \sup_{t \in \RSet} \left\{ e^{-N_1(\rho+1)|t|}  \| \tilde{w}_2(t;\eta) - \tilde{w}_1(t;\eta)   \|  \right\}. 
\end{align*}
Reordering terms now gives the desired result.
\end{enumerate}
 \end{proof}
 
 For step 2, for arbitrary $\eta \in \overline{V}$, let $\Gamma_{\eta}$ be the map as defined in equation \cref{eq:map-Gamma}. We first prove that under suitable conditions on $K, M^x_1, M^y_1, N_1$ and $\rho$, the map $\Gamma_{\eta}$ is a well-defined contraction on $BC^0_{\gamma}(\RSet, L(\RSet^n, \XSet))$.
 
 \begin{lemma}\label{lem:gamma-well-defined}
Assume (H2)$^\prime$-(H3)$^\prime$, and $N_1(\rho+1) < \mu - K M^x_1$. Then for fixed $\eta \in \overline{V}$, the mapping
\begin{equation*}
\begin{aligned}
\Gamma_{\eta} : BC^0_{\gamma}(\RSet, L(\RSet^n, \XSet)) &\rightarrow BC^0_{\gamma}(\RSet, L(\RSet^n, \XSet)); \\
			  \tilde{w} &\mapsto \left( t \mapsto \int_{-\infty}^{t} T_h(t,s;\eta) \, D_y F(h(y(s;\eta)),y(s;\eta)) \, z(s;\eta,\tilde{w}) \, ds \right)
\end{aligned}
\end{equation*}
 is well-defined.
 \end{lemma}
 
 \begin{proof}
 Pick $\eta \in \overline{V}$ and $\tilde{w} \in BC^0_{\gamma}(\RSet, L(\RSet^n,
 \XSet))$ arbitrarily. Then we have by \cref{lem:z-properties} for all $t \in \RSet$:
 \begin{align*}
 \| \Gamma_{\eta}(\tilde{w})(t) \| &\leq K M^y_1 \int_{-\infty}^t e^{-(\mu-K M^x_1)(t-s)} e^{N_1(\rho+1)|s|} \, ds \, \left( 1 + \frac{1}{\rho}\|\tilde{w}\|_{-N_1(\rho+1),\infty} \right) \\
 &= e^{N_1(\rho+1)|t|} \frac{K M^y_1}{\mu-K M^x_1-N_1 (\rho+1)} \left( 1 + \frac{1}{\rho} \|\tilde{w}\|_{-N_1(\rho+1),\infty} \right).
 \end{align*}
 This implies that 
 \[
  \| \Gamma_{\eta}(\tilde{w}) \|_{-N_1(\rho+1),\infty} \leq \frac{K M^y_1}{\mu-K M^x_1-N_1 (\rho+1)} \left( 1 + \frac{1}{\rho} \|\tilde{w}\|_{-N_1(\rho+1),\infty} \right) < \infty,
 \]
 and thereby $\Gamma_{\eta}$ is well-defined on $ BC^0_{\gamma}(\RSet, L(\RSet^n, \XSet))$. 
 \end{proof}
 
 Next, we show that under appropriate conditions the map $\Gamma_{\eta}$ is also a contraction on  $BC^0_{\gamma}(\RSet, L(\RSet^n, \XSet))$.

 \begin{lemma}\label{lem:gamma-contraction}
 Assume (H2)$^\prime$-(H3)$^\prime$. Suppose that $N_1(\rho+1) < \mu - K M^x_1$ as well as $\frac{K M^y_1}{\mu-K M^x_1-N_1 (\rho+1)} < \rho$. Then the mapping $\Gamma_{\eta}$ is a contraction on  $BC^0_{\gamma}(\RSet, L(\RSet^n, \XSet))$.
 \end{lemma}

 \begin{proof}
 Fix $\eta \in \overline{V}$ and let $\tilde{w}_1, \tilde{w}_2 \in
 BC_{\gamma}^{0}(\RSet, L(\RSet^n,\XSet))$. We derive by \cref{lem:z-properties}:
 \begin{align*}
 &\left\|  \Gamma_{\eta}(\tilde{w}_2)(t) -\Gamma_{\eta}(\tilde{w}_1)(t)  \right\| \\
 &\qquad \leq K M^y_1 \int_{-\infty}^t e^{(\mu-K M^x_1)(t-s)} \| z(s;\eta,\tilde{w}_2) - z(s; \eta, \tilde{w}_1)\| \, ds \\
&\qquad \leq e^{N_1(\rho+1)|t|} \frac{1}{\rho} \frac{K M^y_1 }{\mu-K M^x_1-N_1(\rho+1)} \, \|\tilde{w}_2 - \tilde{w}_1\|_{-N_1(\rho+1),\infty}.
 \end{align*}
 
So we see that
\[
\left\|  \Gamma_{\eta}(\tilde{w}_2) -\Gamma_{\eta}(\tilde{w}_1)  \right\|_{-N_1(\rho+1),\infty} \leq \frac{1}{\rho} \frac{K M^y_1 }{\mu-K M^x_1-N_1(\rho+1)} \, \|\tilde{w}_2-\tilde{w}_1\|_{-N_1(\rho+1),\infty},
\]
and as  $\frac{K M^y_1}{\rho(\mu-K M^x_1-N_1 (\rho+1))} < 1$ by assumption, it thus follows that $\Gamma_{\eta}$ is a contraction on $BC^0_{\gamma}(\RSet, L(\RSet^n, \XSet))$.
 \end{proof}
 
 Now by the preceding Lemmas we obtain conditions under which $\Gamma_{\eta}$ is a well-defined map on $BC_{\gamma}^{0}(\RSet, L(\RSet^n,\XSet))$ with a unique fixed point:
 
\begin{corollary}\label{cor:Gamma-unique-fp}
 Assume (H2)$^\prime$-(H3)$^\prime$. Suppose that $N_1(\rho+1) < \mu - K M^x_1$ as well as $\frac{K M^y_1}{\mu-K M^x_1-N_1 (\rho+1)} < \rho$. Then for each $\eta \in \overline{V}$ the operator $\Gamma_{\eta}$ has a unique fixed point $w(\blank; \eta) \in BC^0_{\gamma}(\RSet, L(\RSet^n, \XSet))$. Moreover it holds that for every $\eta \in \overline{V}$ that
 \[
 \| w(\blank; \eta) \|_{-N_1(\rho+1),\infty} \leq \frac{K M^y_1}{\mu-K M^x_1-N_1(\rho+1)}.
 \]
\end{corollary}
 
 \begin{proof}
The existence of a unique fixed point $w(\blank; \eta) \in BC^0_{\gamma}(\RSet,
L(\RSet^n, \XSet))$ of $\Gamma_{\eta}$ for every $\eta \in \overline{V}$ follows from
\cref{lem:gamma-contraction} and the Contraction Mapping Theorem. Furthermore if we have for some $\tilde{w} \in BC^0_{\gamma}(\RSet, L(\RSet^n, \XSet))$ that
\[
 \| \tilde{w} \|_{-N_1(\rho+1),\infty} \leq \rho,
 \]
 then by repeating the steps from the proof of \cref{lem:gamma-well-defined}, but now
 using the results from the first item of \cref{lem:z-properties} that $\| z(t; \eta, \tilde{w}) \| \leq e^{N_1(\rho+1)|t|}$ for all $t \in \RSet$, we have 
 \[
 \|\Gamma_{\eta}(\tilde{w}) \|_{-N_1(\rho+1),\infty} \leq \frac{K M^y_1}{\mu-K M^x_1-N_1(\rho+1)} < \rho.
 \]
 So $\Gamma_{\eta}$ is also a contraction on the closed subset of $BC^0_{\gamma}(\RSet, L(\RSet^n, \XSet))$ given by
 \[
 \left\{ \tilde{w} \in BC^0_{\gamma}(\RSet, L(\RSet^n, \XSet)):   \| \tilde{w} \|_{-N_1(\rho+1),\infty} \leq \rho \right\},
 \]
 and then $w(\blank; \eta)$ needs to be contained in this subset for every $\eta \in \overline{V}$.
  \end{proof}
  
 \begin{lemma}\label{lem:h^1-bdd-ct}
Assume (H2)$^\prime$-(H3)$^\prime$. Suppose that $N_1(\rho+1) < \mu - K M^x_1$ as well as $\frac{K M^y_1}{\mu-K M^x_1-N_1 (\rho+1)} < \rho$. Then we have for $h^1(\eta) := w(0; \eta)$ that 
\[
h^1 \in  \left\{ \sigma \in BC^{0}(\overline{V}, L(\RSet^n,\XSet)): \| \sigma \|_{\infty} \leq \frac{ KM^y_1}{\mu-K M^x_1-N_1(\rho+1)} \right\}.
\]
 \end{lemma}
 
 \begin{proof}
 By the definitions of $h^1$ and $\Gamma_{\eta}$, as well as
 \cref{cor:Gamma-unique-fp} and \cref{lem:z-properties}, we firstly have for each $\eta \in \overline{V}$ that
 \begin{align*}
 \| h^1(\eta) \| &= \left\| \int^0_{-\infty} T_h(t,s; \eta) \, D_yF(h(y(s;\eta)),y(s;\eta)) \, z(s; \eta, w(\blank; \eta)) \, ds \right\| \\
  &\leq K M^y_1 \int_{-\infty}^0 e^{(\mu-K M^x_1-N_1(\rho+1))s} \, ds =  \frac{K M^y_1}{\mu-K M^x_1-N_1(\rho+1)}.
 \end{align*}
 
For $\tilde{\rho} > \rho$ set $\tilde{\gamma} := -N_1(\tilde{\rho} + 1)$. As
$BC^0_{\gamma}(\RSet, L(\RSet^n,\XSet))$ can be continuously embedded in
$BC^0_{\tilde{\gamma}}(\RSet, L(\RSet^n,\XSet))$, by the Uniform Contraction Theorem
(see for example \cite{VaVG87} or \cref{thm:U-Contr-Map} in the Appendix) the continuity of $h^1$ follows from the continuity with respect to $\eta$ of the mapping
\[
\Gamma_{\eta}: BC^0_{\gamma}(\RSet, L(\RSet^n,\XSet)) \rightarrow BC^0_{\tilde{\gamma}}(\RSet, L(\RSet^n,\XSet)).
\]
The continuous dependence of $z(\blank; \eta, \tilde{w})$ on $\eta$ in the
$BC^0_{\tilde{\gamma}}(\RSet, L(\RSet^n,\RSet^n))$-norm follows from item 2 of
\cref{lem:z-properties}. Therefore it suffices to show the continuous dependence on
$\eta$ of the map $\tilde{\Gamma}_{\eta}$ from \cref{rem:w-unique-bdd-sol} in the $BC^0_{\tilde{\gamma}}(\RSet, L(\RSet^n,\XSet))$-norm. 

For this latter claim, observe that for $\eta_1, \eta_2 \in \overline{V}$ and $\tilde{z} \in BC_{\gamma}^{0}(\RSet, L(\RSet^n,\RSet^n))$, the function $v := \tilde{\Gamma}_{\eta_2} (\tilde{z}) - \tilde{\Gamma}_{\eta_1} (\tilde{z})$ is the unique bounded solution to
\begin{align*}
\dot{v} &= D_x F(h(y(s;\eta_1)),y(s;\eta_1)) v  \\
&\qquad +   \left( D_x F(h(y(s;\eta_2)),y(s;\eta_2)) - D_x F(h(y(s;\eta_1)),y(s;\eta_1)) \right) \tilde{w}_2(t) \\
&\qquad + \left( D_y F(h(y(s;\eta_2)),y(s;\eta_2)) - D_y F(h(y(s;\eta_1)),y(s;\eta_1)) \right) \tilde{z}(t),
\end{align*}
where $\tilde{w}_2 := \tilde{\Gamma}_{\eta_2}(\tilde{z})$. This implies
\begin{align*}
&\left(\tilde{\Gamma}_{\eta_2} (\tilde{z}) - \tilde{\Gamma}_{\eta_1} (\tilde{z}) \right)(t) \\
&\quad = \int^t_{-\infty} T_h(t,s;\eta_1) \, \Big(  \left( D_x F(h(y(s;\eta_2)),y(s;\eta_2)) - D_x F(h(y(s;\eta_1)),y(s;\eta_1)) \right) \tilde{w}_2(t) \\
&\qquad \qquad \qquad - \left( D_y F(h(y(s;\eta_2)),y(s;\eta_2)) - D_y F(h(y(s;\eta_1)),y(s;\eta_1)) \right) \tilde{z}(t) \Big)  \, ds.
\end{align*}
Now we can use the assumed continuity of $DF$ to prove the continuous dependence on $\eta$ of $\tilde{\Gamma}_{\eta}(\tilde{z})$ in the $BC^0_{\tilde{\gamma}}(\RSet, L(\RSet^n,\XSet))$-norm  via a standard argument. 
 \end{proof}
 
For step 3, we aim to show that if $h^1(\eta) := \Gamma_{\eta}(w(\blank; \eta))(0)$, then $h^1 = D h$. For this, we proceed as follows; Suppose $w(\blank; \eta)$ is a family of fixed points of the mappings $\Gamma_{\eta}$. First set $z(t;\eta):= z(t; \eta, w(\blank; \eta))$. For $t \in \RSet$, $\eta \in \overline{V}$ and $\Delta \in \RSet^n$ then define
 \begin{align*}
  p(t, \eta, \Delta) &:= x(t; \eta+\Delta) - x(t; \eta) - w(t;\eta) \Delta, \\
  q(t,\eta,\Delta) &:= y(t; \eta+\Delta) - y(t; \eta) - z(t; \eta) \Delta.
 \end{align*}
 Observe that
\[
 p(0,\eta,\Delta) = x(0; \eta + \Delta) - x(0; \eta) - w(0; \eta) \Delta = h( \eta + \Delta) - h(\eta) - h^1(\eta) \Delta,
\]
 and therefore showing that $h^1 = D h$ comes down to proving the following Lemma:
  \begin{lemma}\label{lem:p-o-delta}
     Assume (H2)$^\prime$-(H3)$^\prime$, and suppose that $\rho > \delta$ as well as
     \[
     \frac{K M^y_1}{\mu-K M^x_1-N_1 (\rho+1)} < \rho.
     \] 
    Then it holds that 
  \[
  p(0,\eta,\Delta) = o(|\Delta|) \quad \text{as $\Delta \rightarrow 0$}.
  \]
  \end{lemma}
  
  We break up the proof of this Lemma into two other Lemmas;
  
  \begin{lemma}\label{lem:p-bound}
    Assume (H2)$^\prime$-(H3)$^\prime$ and suppose that $\rho > \delta$. Then for every $\gamma_1 > 0$ there is $\mu_1(\gamma_1) > 0$ such that $|\Delta| < \mu_1$ implies
   \begin{align*}
 e^{-N_1(\rho + 1) |t|} |p(t,\eta,\Delta)| &\leq \gamma_1 |\Delta|  + \frac{K M^y_1}{\mu-K M^x_1-N_1(\rho+1)} \sup_{t \in \RSet} \left\{ e^{-N_1(\rho+1) |t|}|q(t,\eta,\Delta)| \right\}
\end{align*}
   for all $t \in \RSet$.
  \end{lemma}
 
 \begin{proof}
Set $R_h(x,y) := F(x,y)-D_xF(h(y),y) x$, so that $D_x R_h(h(y),y) = 0$. As $x(t; \eta) = h(y(t; \eta))$ is bounded (with respect to $t$) and satisfies the differential equation
\[
\partial_t x(t; \eta) = F(x(t;\eta), y(t;\eta)) = D_x F(h(y(t;\eta)),y(t; \eta)) x(t;\eta) + R_h(x(t;\eta), y(t;\eta)),
\]
we have
\begin{align*}
x(t;\eta) =   \int_{-\infty}^t T_h(t,s;\eta) R_h(x(s;\eta),y(s;\eta)) \, ds,
\end{align*}
where $T_h$ is as defined in the lead-up to \cref{lem:Th-props} and satisfies the properties of this Lemma.  

Making use of this integral equation for $x(t;\eta)$, as well as the integral equation $w(\blank; \eta)=\Gamma_{\eta}(w(\blank; \eta))$, we therefore have
 \begin{align*}
 p(t, \eta, \Delta) &= \int_{-\infty}^t T_h(t,s; \eta+\Delta) R_h(x(s; \eta+\Delta), y(s; \eta+\Delta)) \\
 &\qquad \qquad \qquad - T_h(t,s; \eta) R_h(x(s; \eta), y(s; \eta))\\
 &\qquad \qquad \qquad \qquad - T_h(t,s;\eta) D_y F(x(s;\eta),y(s;\eta)) z(s; \eta) \Delta  \, ds.
 \end{align*}
 Now set
 \begin{align*}
  \hat{F}(t, \eta, \Delta) := &F(x(t;\eta+\Delta), y(t;\eta+\Delta)) - F(x(t;\eta), y(t;\eta)) \\
& - D F(x(t;\eta), y(t;\eta)) \begin{pmatrix} x(t;\eta+\Delta) - x(t;\eta) \\  y(t;\eta+\Delta)-y(t;\eta)  \end{pmatrix}.
  \end{align*}
Adding and subtracting the term $T_h(t,s;\eta) \hat{F}(s, \eta, \Delta)$ under the integral sign in the integral equation for $p(t, \eta, \Delta)$ allows us to rewrite is as
 \begin{align*}
 p(t, \eta, \Delta) &=  \int_{-\infty}^t \left( T_h(t,s; \eta+\Delta) - T_h(t,s;\eta) \right)  F(x(s; \eta+\Delta), y(s; \eta+\Delta))  \\
 &\qquad \qquad  + \partial_s \left(T_h(t,s;\eta+\Delta) - T_h(t,s;\eta) \right)  x(s; \eta+\Delta)  \\
 &\qquad \qquad \qquad + T_h(t,s;\eta) \left(  \hat{F}(s, \eta, \Delta) + D_y F(x(s;\eta),y(s;\eta)) q(s,\eta,\Delta)  \right)  \, ds \\
 &=  \int_{-\infty}^t T_h(t,s;\eta) \left(  \hat{F}(s, \eta, \Delta) + D_y F(x(s;\eta),y(s;\eta)) q(s,\eta,\Delta)  \right)  \, ds,
 \end{align*}
 where integration by parts was used to obtain the last equality.

By uniform boundedness of $DF$ (this follows from hypothesis (H3)$^\prime$) we must have for any $\rho > \delta$ and any $C>0$ that
 \[
 \lim_{t \rightarrow \pm \infty} \sup_{|\Delta| < C} e^{-N_1(\rho - \delta)|t|} \| DF(x(t;\eta+\Delta),y(t;\eta+\Delta)) - DF(x(t;\eta),y(t;\eta)) \| = 0.
 \]
 By \cref{cor:uniformdiff} it then holds that for every $\gamma_1 > 0$ there exists $\mu_1(\gamma_1)>0$ such that $|\Delta| < \mu_1$ implies that
    \begin{align*}
e^{ -N_1 (\rho-\delta)|t|} \hat{F}(t,\eta,\Delta) \leq \gamma_1 (\rho + 1) |y(t;\eta+\Delta)-y(t;\eta)|
  \end{align*}
  for all $t \in \RSet$. Now $|\Delta| < \mu_1$ implies as well by \cref{lem:sol-slow-subsystem} that
      \begin{align*}
\sup_{t \in \RSet} e^{ -N_1 (\rho+1)|t|} \hat{F}(t,\eta,\Delta) \leq \gamma_1 (\rho + 1) |\Delta|.
  \end{align*}
  
From estimates we made before we can also see that it holds that
\[
\sup_{t \in \RSet}  e^{ - N_1 (\rho +1)|t|} |q(t,\eta,\Delta)|  < \infty,
\]
and therefore we derive that for $|\Delta|<\mu_1$ we have
\begin{align*}
 e^{- N_1(\rho + 1) |t|} |p(t,\eta,\Delta)| &\leq \frac{K \gamma_1 (\rho+1)}{\mu-K M^x_1-N_1(\rho+1)} |\Delta| \\
 &\qquad \quad + \frac{K M^y_1}{\mu-K M^x_1-N_1(\rho+1)} \sup_{t \in \RSet} \left\{ e^{-N_1(\rho+1)|t|}|q(t,\eta,\Delta)| \right\}
\end{align*}
for all $t \in \RSet$.
\end{proof}

\begin{lemma}\label{lem:q-bound}
       Assume (H2)$^\prime$-(H3)$^\prime$ and suppose that $\rho > \delta$.  Then for every $\gamma_2 > 0$ there is $\mu_2(\gamma_2) > 0$ such that $|\Delta| < \mu_2$ implies
   \begin{align*}
    e^{ -N_1 (\rho +1)|t|} |q(t,\eta,\Delta)| \leq  \gamma_2 |\Delta| + \frac{1}{\rho+1} \sup_{t \in \RSet} e^{-N_1 (\rho +1)|t|}  \left( |p(t,\eta,\Delta)| +  |q(t, \eta, \Delta)| \right)
  \end{align*}
   for all $t \in \RSet$.
 \end{lemma}

\begin{proof}
We first set
  \begin{align*}
  \hat{g}(t,\eta,\Delta) &:= g(x(t;\eta+\Delta), y(t;\eta+\Delta)) - g(x(t;\eta), y(t;\eta)) \\
&\qquad  - D g(x(t;\eta), y(t;\eta)) \begin{pmatrix} x(t;\eta+\Delta) - x(t;\eta) \\ y(t;\eta+\Delta)-y(t;\eta)  \end{pmatrix}.
  \end{align*}
 Integration gives for $t \in \RSet$:
  \begin{align*}
  q(t,\eta,\Delta) &= \sgn(t) \int_0^t g(x(s;\eta+\Delta), y(s;\eta+\Delta)) - g(x(s;\eta), y(s;\eta)) \\
  &\qquad \qquad - D g(x(s;\eta), y(s;\eta)) \begin{pmatrix} h^1(y(s;\eta)) \\ I_{\RSet^n} \end{pmatrix} z(s;\eta) \Delta \, ds.
    \end{align*}
  Then by adding and subtracting under the integral sign the term 
  \[
  D g(x(s;\eta), y(s;\eta)) \begin{pmatrix} x(s;\eta+\Delta) - x(s;\eta) \\ y(s;\eta+\Delta)-y(s;\eta)  \end{pmatrix},
 \]
we obtain the equality
  \[
    q(t,\eta,\Delta) = \sgn(t) \int_0^t \hat{g}(s,\eta,\Delta) +  D g(x(s;\eta), y(s;\eta)) \begin{pmatrix} p(s,\eta,\Delta) \\ q(s,\eta,\Delta) \end{pmatrix}   \, ds.
  \]

 Like in the proof of \cref{lem:p-bound} if $\rho > \delta$ we can use \cref{cor:uniformdiff} to derive that for arbitrary $\gamma_2 > 0$ there exists $\mu_2(\gamma_2) > 0$ such that $|\Delta|<\mu_2$ implies that
  \begin{align*}
e^{-N_1 (\rho+1)|t|} \hat{g}(t,\eta,\Delta) \leq \gamma_2 (\rho + 1) |\Delta|.
  \end{align*}
  for all $t \in \RSet$. Therefore for $|\Delta| < \mu_2$ we have for all $t \in \RSet$
  \begin{align*}
  |q(t,\eta,\Delta)| &\leq \left| \int_0^t \gamma_2 (\rho+1)  |\Delta| e^{N_1 (\rho +1) |s|}  + N_1 |p(s,\eta, \Delta)| + N_1 |q(s, \eta,\Delta)| \, ds \right|.
  \end{align*}
  Thus if $|\Delta| < \mu_2$ we derive that 
    \begin{align*}
    e^{ - N_1 (\rho +1)|t|} |q(t,\eta,\Delta)| \leq  \frac{\gamma_2|\Delta|}{N_1} + \frac{1}{\rho+1} \sup_{t \in \RSet} e^{-N_1 (\rho +1)|t|}  \left( |p(t,\eta,\Delta)| +  |q(t, \eta, \Delta)| \right)
  \end{align*}
  for all $t \in \RSet$. 
  \end{proof}
  
 We are now ready to prove \cref{lem:p-o-delta}.

\begin{proof}[Proof of \cref{lem:p-o-delta}]
As by assumption we have 
\[
\frac{K M^y_1}{\mu-K M^x_1-N_1(\rho+1)} \cdot \frac{1}{\rho}  < 1,
\]
combining the results of \cref{lem:p-bound,lem:q-bound} we have for $\rho > \delta$ that
  \[
  \sup_{t \in \RSet} e^{-N_1 (\rho +1)|t|}  |p(t,\eta,\Delta)| = o(|\Delta|),
  \]
 so in particular also $p(0,\eta,\Delta) = o(|\Delta|)$.
 \end{proof}
 
 With \cref{lem:p-o-delta} proven, we are able to derive the smoothness of the slow
 manifold parameterized by $h$ as stated in \cref{thm:slow-mfd-smoothness};
 
 \begin{proof}[Proof of \cref{thm:slow-mfd-smoothness}]
 By \cref{lem:p-o-delta}, if $\{ w(\blank; \eta) : \eta \in \overline{V} \}$ is a
 family of fixed points of the mappings $\Gamma_{\eta}$, and we set $h^1(\eta) :=
 \Gamma_{\eta}(w(\blank; \eta))(0)$, then we have $h^1 = Dh$. The existence of a
 family of fixed points $\{ w(\blank; \eta) : \eta \in \overline{V} \}$ of the
 mappings $\Gamma_{\eta}$ is guaranteed by \cref{cor:Gamma-unique-fp}, and $h^1 = Dh$
 is continuous and uniformly bounded by \cref{lem:h^1-bdd-ct}. Observe that the
 conditions of \cref{cor:Gamma-unique-fp,lem:h^1-bdd-ct,lem:p-o-delta} can be
 satisfied by picking $\rho > \delta$ sufficiently close to $\delta > 0$ from \cref{thm:slow-mfd-existence}.
 \end{proof}
 
 
 \section{$k$-smooth invariant manifolds}\label{sec:k-smooth-slow-mfds}.
In this Section we show, under some suitable additional conditions on the vector
field, that the invariant manifold parameterization $h$ from \cref{thm:slow-mfd-existence} is actually $k$-times continuously differentiable 
($k \geq 1$) with uniformly bounded derivatives up to order $k$. Our proof method of
choice is an induction argument that is set up by considering the $k$-th variational
equation of the system alongside orbits on the invariant manifold. We remark that, similar
to what was mentioned in \cref{rem:proof-smoothness-variational},  this set-up
deviates from other texts that construct candidate derivatives for $h$ by
differentiating (their version of) the mapping $\Lambda$ from \cref{sec:slow-mfds-existence}.

We enunciate below the central result of this section, albeit we only provide here a
sketch of the proof, for conciseness.
 
 \begin{theorem}[$k$-smooth invariant manifolds]\label{thm:slow-mfd-k-smoothness}
Assume that the conditions for \cref{thm:slow-mfd-existence} are satisfied, and let $h \in BC^{0,1}(\overline{V},\XSet)$ be the invzriant manifold parameterization 
as described in that Theorem. Suppose that $U$ is an open set in $\XSet$ that contains $h(\overline{V})$ such that the following assumptions are satisfied for some $k \geq 1$:
 \begin{enumerate}
 \item $F \in C^k(U \times \overline{V}, \XSet)$, $D F \in BC^{k-1} (U \times \overline{V}, L(\XSet \times \RSet^n, \XSet))$ and the estimates from (H2)$^{\prime}$ hold;
 \item $g \in C^k(U \times \overline{V}, \RSet^n)$, $D g \in BC^{k-1} (U \times \overline{V}, L(\XSet \times \RSet^n, \RSet^n))$  and there exists a constant $N_1 > 0$ with $k N_1 < \mu - K M^x_1$ such that
 \[
 \sup\left\{\| D g(x,y) \|: (x,y) \in U \times \overline{V}\right \}  \leq N_1;
 \]
 \item There exists $\rho > 0$ such that 
\[
\frac{K M^y_1}{\mu-K M^x_1-k N_1 (\rho+1)} < k(\rho+1)-1.
\]
Then for the function $h \in BC^{0,1}(\overline{V},\XSet)$ it holds moreover that $h \in BC^{k}(\overline{V},\XSet)$.
 \end{enumerate}
 \end{theorem}
 
To prove this Theorem, we use an induction argument. The base case $k =1$ was proven
in \cref{sec:slow-mfds-smoothness}. For the induction step, assume the Theorem holds
up until $k-1$ for some $k \in \mathbb{N}_{\geq 2}$. As the induction step can be
carried out in a manner that is mostly analogous to the material from
\cref{sec:slow-mfds-smoothness}, we will be comparatively brief in our description of
the proofs involved.

 To set up a strategy for carrying out the induction step, first suppose that indeed
 also $h \in BC^{k}(\overline{V},\XSet)$. We denote $y(t; \eta) := \psi(t; \eta, h)$
 (recall \cref{lem:sol-slow-subsystem}) and $x(t; \eta) := h(y(t;\eta))$. Then $\partial^k_\eta x(t; \eta)$ satisfies the $k$-th variational equation

 \begin{align}
 \begin{split}
 \partial_t \partial^k_{\eta} x(t;\eta)&= \partial^k_{\eta} F(x(t;\eta),y(t;\eta)); \\
 &= D F(x(t;\eta), y(t;\eta))\,  \partial^k_\eta \begin{pmatrix} x(t; \eta) \\ y(t;\eta) \end{pmatrix} + S^x_{k}(t; \eta), \\
\partial^k_\eta x(0; \eta) &= D^k h(\eta),
 \end{split}
 \end{align}
 where the term $S^x_{k}(t; \eta)$ follows from Fa\`a di Bruno's formula for the generalized chain rule for derivatives and contains only $\eta$-derivates for $x$ and $y$ of up to order $k-1$. We observe that
 \begin{equation}
 \begin{aligned}
 S^x_k(t; \eta) &=  \partial_{\eta} DF(x(t;\eta),y(t;\eta)) \, \partial^{k-1}_\eta \begin{pmatrix} x(t; \eta) \\  y(t;\eta) \end{pmatrix} + \partial_{\eta} S^x_{k-1}(t; \eta) \\
 &= D^2F(x(t;\eta),y(t;\eta))  \,  \partial_\eta \begin{pmatrix} x(t; \eta) \\  y(t;\eta) \end{pmatrix} \,  \partial^{k-1}_\eta \begin{pmatrix}  x(t; \eta)\\  y(t;\eta) \end{pmatrix} + \partial_{\eta} S^x_{k-1}(t; \eta),
 \end{aligned}
 \end{equation}
 where $S^x_1(t;\eta) = 0$.
 
Moreover 
 \begin{align*}
 \partial_t \partial^k_{\eta} y(t;\eta) &=  D g(x(t;\eta), y(t;\eta)) \, \partial^k_\eta  \begin{pmatrix} x(t; \eta) \\ y(t;\eta) \end{pmatrix} + S^y_k(t;\eta); \\
 \partial^k_{\eta} y(0;\eta) &=  0,
 \end{align*}
 where again the term $S^y_{k}(t; \eta)$ follows from Fa\`a di Bruno's formula and contains only $\eta$-derivates for $x$ and $y$ of up to order $k-1$. We observe
  \begin{equation}
 \begin{aligned}
 S^y_k(t; \eta) &=  \partial_{\eta} Dg(x(t;\eta),y(t;\eta)) \, \partial^{k-1}_\eta \begin{pmatrix} x(t; \eta) \\  y(t;\eta) \end{pmatrix} + \partial_{\eta} S^y_{k-1}(t; \eta) \\
 &= D^2g(x(t;\eta),y(t;\eta)) \,  \partial_\eta  \begin{pmatrix} x(t; \eta) \\  y(t;\eta) \end{pmatrix} \,  \partial^{k-1}_\eta \begin{pmatrix}  x(t; \eta)\\  y(t;\eta) \end{pmatrix} + \partial_{\eta} S^y_{k-1}(t; \eta).
 \end{aligned}
 \end{equation}
 
Therefore we introduce candidates for the $k$-th $\eta$-derivatives of $(t,\eta) \mapsto x(t;\eta)$ and $(t,\eta) \mapsto y(t;\eta)$ via the non-autonomous linear ODE
\begin{align*}
\partial_t \begin{pmatrix} w_k(t; \eta) \\ z_k(t; \eta)  \end{pmatrix} &= \begin{pmatrix} DF(x(t;\eta),y(t;\eta))  \\  Dg(x(t;\eta), y(t;\eta))    \end{pmatrix} \begin{pmatrix} w_k(t; \eta) \\ z_k(t; \eta)  \end{pmatrix}  + S_k(t; \eta),  \\
&= \begin{pmatrix} D_xF(x(t;\eta),y(t;\eta))  & D_yF(x(t;\eta),y(t;\eta))  \\  D_x g(x(t;\eta), y(t;\eta)) &  D_y g(x(t;\eta), y(t;\eta))    \end{pmatrix} \begin{pmatrix} w_k(t; \eta) \\ z_k(t; \eta)  \end{pmatrix}  + S_k(t; \eta),
\end{align*}
where
\begin{align*}
S_k(t; \eta) := \begin{pmatrix} S^x_k(t; \eta) \\ S^y_k(t; \eta) \end{pmatrix} &= \begin{pmatrix} D^2 F(x(t;\eta),y(t;\eta))  \\  D^2 g(x(t;\eta), y(t;\eta))   \end{pmatrix} \begin{pmatrix} w_{1}(t; \eta) \\  z_{1}(t;\eta) \end{pmatrix} \begin{pmatrix} w_{k-1}(t; \eta) \\  z_{k-1}(t;\eta) \end{pmatrix} \\
&\qquad \qquad \qquad \qquad \qquad \qquad \qquad \qquad + \partial_{\eta} \begin{pmatrix} S^x_{k-1}(t; \eta) \\ S^y_{k-1}(t; \eta) \end{pmatrix}.
\end{align*}

A solution such that $z_k(0;\eta) = 0$ and $w_k(\blank; \eta)$ stays bounded has to satisfy the integral equations
\begin{equation}\label{eq:k-der-int-eqn}
\begin{aligned}
w_k(t; \eta) &= \int^t_{-\infty} T_h(t,s; \eta) \left( D_yF(x(s;\eta),y(s;\eta)) z_k(s; \eta) + S^x_{k}(s; \eta) \right) \, ds  ; \\
z_k(t; \eta) &= - \int^0_t \, Dg(x(s;\eta), y(s;\eta)) \begin{pmatrix} w_k(s; \eta) \\ z_k(s; \eta)  \end{pmatrix}  + S^y_k(s; \eta) ds \\
&= - \int^0_t \,T^{sl}_h(t,s; \eta) \left( D_x g(x(s;\eta), y(s;\eta))  w_k(s; \eta) + S^y_k(s; \eta) \right) \, ds,
\end{aligned}
\end{equation}
where $T_h(t,s; \eta)$ is the same process as introduced in \cref{lem:Th-props} and $T^{sl}_h(t,s; \eta)$ is the reversible process associated to
\[
\partial_t \tilde{z}_k(t; \eta) =  D_y g(x(t;\eta), y(t;\eta)) \tilde{z}_k(t; \eta).
\]
Note that as $\| D_y g(x,y) \|$ is bounded by $N_1$ for $(x,y) \in U \times \overline{V}$, we have the estimate $\| T^{sl}_h(t,s; \eta) \| \leq e^{N_1|t-s|}$. 

We can inductively show that $\sup_{t \in \RSet} e^{-k N_1|t|} \| S_k(t; \eta) \| < \infty$, and therefore the existence of a bounded solution $w_k(\blank; \eta)$ with $z_k(0;\eta)=0$ will be guaranteed by the condition 
\[
kN_1 < \mu - KM^x_1.
\]
The full proof can be done by setting up a contraction mapping argument that is
analogous to the one presented in \cref{sec:slow-mfds-smoothness}, whereby the
contraction mapping $\Gamma_{\eta}$ itself follows the right-hand side of \cref{eq:k-der-int-eqn}, and has as domain the Banach space
\[
BC^0_{\gamma}(\RSet, L^k(\XSet \times \RSet^n, \XSet)), \quad \text{with } \gamma := -k N_1 (\rho+1).
\]

If $\tilde{w}^i_k(\blank; \eta) \in BC^0_{\gamma}(\RSet, L^k(\XSet \times \RSet^n, \XSet))$, for $i \in \{1,2\}$, then for the solutions $z^i_k(\blank; \eta)$ to the IVPs
 \[
 \partial_t z^i_k(t; \eta) = D g(x(t;\eta), y(t;\eta)) \begin{pmatrix} \tilde{w}^i_k(s; \eta) \\ z^i_k(s; \eta)  \end{pmatrix}   + S^y_{k}(t; \eta); \quad z(0; \eta) = 0,
 \]
we have for any $\rho > 0$ and $t \in \RSet$ that
\[
\| z_k^2(t; \eta) - z_k^1(t; \eta) \| \leq \frac{1}{k(\rho+1)-1} \| \tilde{w}_k^2(\blank; \eta) - \tilde{w}_k^1(\blank; \eta)\|_{-k N_1 (\rho+1),\infty} e^{k N_1 (\rho+1)|t|}.
\]

Now if we subsequently set
\[
w^i_k(t; \eta) := \int^t_{-\infty} T_h(t,s; \eta) \left( D_yF(x(s;\eta),y(s;\eta)) z^i_k(s; \eta) + S^x_{k}(s; \eta) \right) \, ds,
\]
we derive that
\begin{align*}
&\| w_k^2(\blank; \eta) - w_k^1(\blank; \eta)\|_{-k(\rho+1)N_1,\infty} \\ 
&\qquad \leq \frac{1}{k(\rho+1)-1} \cdot \frac{KM^y_1}{\mu - KM^x_1 - k N_1(\rho+1)} \| \tilde{w}_k^2(\blank; \eta) - \tilde{w}_k^1(\blank; \eta)\|_{-k N_1 (\rho+1),\infty}.
\end{align*}

Via arguments analogous to those in \cref{sec:slow-mfds-smoothness}, we therefore
have that $\Gamma_{\eta}$ has a unique fixed point in $BC^0_{\gamma}(\RSet, L^k(\XSet
\times \RSet^n, \XSet))$ under the assumptions from \cref{thm:slow-mfd-k-smoothness}.
Moreover, $\Gamma_{\eta}$ can also be shown to
depend continuously on $\eta$ when considered as a map from $BC^0_{\gamma}(\RSet,
L^k(\XSet \times \RSet^n, \XSet))$ into  $BC^0_{\tilde{\gamma}}(\RSet, L^k(\XSet
\times \RSet^n, \XSet))$, in case we have $\tilde{\gamma} := -k N_1 (\tilde{\rho} +
1)$ for some $\tilde{\rho} > \rho$.

To check that the fixed points $w_k(\blank; \eta)$ of the mappings $\Gamma_{\eta}$
(viewed as a family of contraction mappings over $\eta \in \overline{V}$) indeed
provide the $k$-th $\eta$-derivatives for the invariant manifold trajectories $x(\blank;
\eta)$ we define, again analogously to arguments in \cref{sec:slow-mfds-smoothness},
for $t \in \RSet$, $\eta \in \overline{V}$ and $\Delta \in \RSet^n$ the quantities
\begin{align*}
  p_k(t,\eta,\Delta) &:= \partial^{k-1} _{\eta}x(t; \eta+\Delta) - \partial^{k-1} _{\eta} x(t; \eta) - w_k(t;\eta) \Delta, \\
  q_k(t,\eta,\Delta) &:= \partial^{k-1} _{\eta} y(t;\eta+\Delta) - \partial^{k-1} _{\eta} y(t; \eta) - z_k(t; \eta) \Delta,
\end{align*}
where now  $z_k(\blank; \eta)$ are the solutions to the IVPs
 \[
 \partial_t z_k(t; \eta) = D g(x(t;\eta), y(t;\eta)) \begin{pmatrix} w_k(s; \eta) \\ z_k(s; \eta)  \end{pmatrix}   + S^y_{k}(t; \eta); \quad z(0; \eta) = 0.
 \]
 
 Observe that if we set $h^k(\eta) := w_k(t; \eta)$, then 
 \[
 D^{k-1}h(\eta+\Delta) - D^{k-1}h(\eta) - h^k(\eta) \Delta = p_k(0, \eta, \Delta).
 \]
 Therefore it holds that if $p_k(0,\eta,\Delta) = o(|\Delta|)$ for arbitrary $\eta
 \in \overline{V}$, then we have $h^k(\eta) = D^k h(\eta)$. The continuity and
 boundedness of $h^k$ can be proven analogously to \cref{lem:h^1-bdd-ct}. So to conclude that $D^k h$ exists and is continuous and uniformly bounded, and thereby complete our induction argument, it is sufficient to show the following Lemma:
 
 \begin{lemma}\label{lem:pqk-bounds}
 Under the assumptions from \cref{thm:slow-mfd-k-smoothness}, we have for
 arbitrary $\eta \in \overline{V}$ that
 \begin{align*}
\| p_k(\blank, \eta, \Delta) \|_{\gamma,\infty} = \sup_{t \in \RSet} e^{-kN_1(\rho+1)|t|}  \| p_k(t, \eta, \Delta)  \| &=  o(|\Delta|), \\
\| q_k(\blank, \eta, \Delta) \|_{\gamma,\infty} = \sup_{t \in \RSet} e^{-kN_1(\rho+1)|t|}  \| q_k(t, \eta, \Delta)  \| &=  o(|\Delta|).
\end{align*}
 \end{lemma}
 
This Lemma can be proven in a manner analogous to \cref{lem:p-bound,lem:q-bound} from
\cref{sec:slow-mfds-smoothness}. Namely, we first set
\begin{align*}
\hat{F}_2(t, \eta, \Delta) &:=  DF(x(t;\eta+\Delta),y(t;\eta+\Delta)) - DF(x(t;\eta),y(t;\eta)) \\
&\qquad \qquad \qquad \qquad \qquad   - D^2F(x(t;\eta),y(t;\eta))  \begin{pmatrix}  w_1(t; \eta)\\  z_1(t; \eta) \end{pmatrix} \Delta  ; \\
\hat{g}_2(t, \eta, \Delta) &:=  Dg(x(t;\eta+\Delta),y(t;\eta+\Delta)) - Dg(x(t;\eta),y(t;\eta)) \\
&\qquad \qquad \qquad \qquad \qquad     -  D^2g(x(t;\eta),y(t;\eta))   \begin{pmatrix}  w_1(t; \eta)\\  z_1(t; \eta) \end{pmatrix} \Delta    ;\\
\hat{S}^i_{k-1} (t, \eta, \Delta) &:= S^i_{k-1}(t;\eta+\Delta) - S^i_{k-1}(t;\eta) - \partial_{\eta} S^i_{k-1}(t;\eta) \Delta, \quad i \in \{x,y\}.
\end{align*}
Then we observe that
\begin{align*}
\partial_t &\begin{pmatrix} p_k(t, \eta, \Delta) \\  q_k(t,\eta,\Delta)   \end{pmatrix}  \\
&= \begin{pmatrix} DF(x(t;\eta),y(t;\eta))  \\  Dg(x(t;\eta), y(t;\eta))    \end{pmatrix} \begin{pmatrix} p_k(t; \eta, \Delta) \\ q_k(t; \eta, \Delta)  \end{pmatrix}   +    \begin{pmatrix} \hat{F}_2(t, \eta, \Delta)   \\  \hat{g}_2(t, \eta, \Delta)  \end{pmatrix} \begin{pmatrix} w_{k-1}(t; \eta+\Delta) \\ z_{k-1}(t; \eta + \Delta)  \end{pmatrix} \\
&\quad + \begin{pmatrix} D^2 F(x(t;\eta),y(t;\eta))  \\  D^2 g(x(t;\eta), y(t;\eta)) \end{pmatrix}  \begin{pmatrix} w_1(t; \eta) \\ z_1(t; \eta )  \end{pmatrix} \Delta \begin{pmatrix}  w_{k-1}(t; \eta+\Delta) - w_{k-1}(t; \eta)\\  z_{k-1}(t;\eta+\Delta) -z_{k-1}(t; \eta) \end{pmatrix}  \\
&\quad + \begin{pmatrix} \hat{S}^x_{k-1}(t, \eta, \Delta) \\  \hat{S}^y_{k-1}(t, \eta, \Delta)  \end{pmatrix}.
\end{align*}
By the variation of constants formula we can now obtain integral equations for
$p_k(t, \eta, \Delta)$ and $q_k(t,\eta,\Delta)$ analogous to those used in the proofs
of \cref{lem:p-bound,lem:q-bound}, and thereby derive \cref{lem:pqk-bounds}. Compared
to the proofs of \cref{lem:p-bound,lem:q-bound}, the integral
equations now include extra terms under the integral signs that contain the factors
given by the two rows of
\begin{align*}
&\begin{pmatrix} D^2 F(x(t;\eta),y(t;\eta))  \\  D^2 g(x(t;\eta), y(t;\eta)) \end{pmatrix}  \begin{pmatrix} w_1(t; \eta) \\ z_1(t; \eta )  \end{pmatrix} \Delta \begin{pmatrix}  w_{k-1}(t; \eta+\Delta) - w_{k-1}(t; \eta)\\  z_{k-1}(t;\eta+\Delta) -z_{k-1}(t; \eta) \end{pmatrix}   \\ 
&\qquad \qquad \qquad \qquad \qquad \qquad \qquad \qquad \qquad \qquad \qquad \qquad + \begin{pmatrix} \hat{S}^x_{k-1}(t, \eta, \Delta)  \\  \hat{S}^y_{k-1}(t, \eta, \Delta)  \end{pmatrix}.
\end{align*}
It can be checked that by the induction hypotheses these extra terms are still consistent with having 
\[
\sup_{t \in \RSet} e^{-kN_1(\rho+1)|t|}  \| p_k(t, \eta, \Delta)  \| =  o(|\Delta|), \quad \sup_{t \in \RSet} e^{-kN_1(\rho+1)|t|}  \| q_k(t, \eta, \Delta)  \| =  o(|\Delta|).
\]

We can now complete the induction step similarly to how 
\cref{thm:slow-mfd-smoothness} was eventually proven at the very end of 
\cref{sec:slow-mfds-smoothness}, and thereby we also prove 
\cref{thm:slow-mfd-k-smoothness}.


 \section{Invariant manifold reduction and straightening of invariant manifolds}\label{sec:slow-mfds-straightening}
 
 This Section and \crefrange{sec:red-map-existence}{sec:red-map-smoothness} treat
 \textit{invariant manifold reduction}, which can be interpreted as an extension of the
 classical reduction principle for centre manifolds of non-hyperbolic equilibria
 pioneered by Pliss \cite{Pl64}. This invariant manifold reduction entails that for system
 \cref{ODE-system}, under suitable assumptions, all orbits are asymptotically
 attracted to an orbit on the invariant manifold $S_h$ described by the function $h$ from
 \cref{thm:slow-mfd-k-smoothness}. 

Specifically we show that the following holds:
 
 \begin{theorem}[Reduction principle for invariant manifolds]\label{thm:red-principle-slow-mfd}
 Suppose that for system \cref{ODE-system} the conditions for
 \cref{thm:slow-mfd-k-smoothness} hold, with $k \geq 2$, and that also
 \[
 K N_1(\rho + 1) < \mu - K M^x_1.
 \]
 Then there exists a function $P \in C^0(\XSet \times \overline{V}, S_h) \cap C^{k-1}(U \times \overline{V}, S_h)$ such that for all $(\xi,\eta) \in \XSet \times \overline{V}$ and $t \geq 0$ it holds that
 \begin{enumerate}
 \item $P((x,y)(t; (\xi,\eta))) = (x,y)(t; P(\xi,\eta))$;
 \item $| (x,y)(t; (\xi,\eta)) - (x,y)(t; P(\xi,\eta)) | = \mathcal{O}\left(e^{-(\mu-K M^x_1) t}\right).$
 \end{enumerate}
 Moreover, if it holds as well that
 \[
 (k+1) N_1(\rho + 1) < \mu - K M^x_1,
 \]
 then also $P \in C^k(U \times \overline{V}, S_h)$. 
 \end{theorem}
 
 The relationship between orbits $(x,y)(\blank ; (\xi,\eta))$ of the system and the
 \textit{reduction map} $P$ can be displayed in a commutative diagram as in
 \cref{fig:com-dia-red-map}. If $\pi_y$ denotes projection on the second coordinate
 of the product space $\XSet \times {V}$, then $\pi_y$ restricted to $S_h$ provides
 us with coordinates for the invariant manifold $S_h$ in terms of the chart $\overline{V}
 \ni \eta \mapsto (h(\eta), \eta)$.
  In the $y$-coordinate orbits on $S_h$ are given by $\psi(\blank; \eta, h)$ for $\eta \in \overline{V}$, and therefore the \textit{slow flow} on $S_h$ is given by the orbits $(h(\psi(\blank; \eta,h)), \psi(\blank; \eta,h))$ with $\eta \in \overline{V}$. 
 
 \begin{figure}\label{fig:com-dia-red-map}\centering
\begin{tikzcd}[sep = 25mm, every label/.append style = {font = \normalfont}]
(\xi,\eta)  \arrow[symbol = \in]{d} & (\xi^p,\eta^p) \arrow[symbol = \in]{d} &  \eta^p \arrow[symbol = \in]{d}  \\ [-60pt]
\XSet \times \overline{V}  \arrow{d}[swap]{(x,y)(t; (\xi,\eta))} \arrow{r}{P(\xi,\eta)} & S_h \arrow{d} \arrow{r}{\pi_y (\xi^p,\eta^p)}  & \overline{V} \arrow{d}{\psi(t; \eta^p, h)} \\
\XSet \times \overline{V}  \arrow[swap]{r}{P(\xi_t,\eta_t)} \arrow[symbol = \ni]{d} & S_h   \arrow[swap]{r}{\pi_y(\xi^p_t,\eta^p_t)}  \arrow[symbol = \ni]{d}   & \overline{V} \\ [-60pt]
(\xi_t,\eta_t) & (\xi^p_t,\eta^p_t)
\end{tikzcd}
\caption{The relationship between the (semi)flow of the system and the reduction map $P$, for arbitrary $t \geq 0$. The reduction map $P$ is a topological semi-conjugation between the (semi)flow of the system and the slow flow on $S_h$ defined by $\RSet \times \overline{V} \ni (t; \eta) \mapsto (h(\psi(t; \eta, h)),\psi(t; \eta, h))$.}
\end{figure}
 
 Suppose $(\xi, \eta) \in \overline{V}$ is given, and set $Q(\xi,\eta) := (\xi,\eta) - P(\xi,\eta)$ (we prove \cref{thm:red-principle-slow-mfd} via $Q$ in the next Section). Then by \cref{thm:red-principle-slow-mfd} it holds for all $t \geq 0$ that
 \begin{align*}
 (x,y)(t; (\xi,\eta)) &= P ((x,y)(t; (\xi,\eta))) + Q ((x,y)(t; (\xi,\eta))) \\
 &= (x,y)(t; P(\xi,\eta)) + ((x,y)(t; (\xi,\eta)) - (x,y)(t; P(\xi,\eta)) ) \\
 &=: (x,y)(t; P(\xi,\eta)) + r(t;(\xi,\eta)) \\
 &= (h(\psi(t; \eta,h)),\psi(t; \eta,h)) +  r(t;(\xi,\eta)),
 \end{align*}
 where $r(t;(\xi,\eta)) =  \mathcal{O}\left(e^{-(\mu-K M^x_1) t}\right)$. So we see that the orbit $(x,y)(t; (\xi,\eta))$ equals the orbit $(x,y)(t; P(\xi,\eta))$, which lies on $S_h$, modulo some (fast) exponentially decaying term. In this sense all dynamics can be reduced to the slow flow on $S_h$.
 
 Observe that in geometric terms the reduction map $P$ from \cref{thm:red-principle-slow-mfd} gives rise to a $C^{k-1}$ (or even $C^k$) invariant
 stable foliation of $U$ where the stable stem is the invariant manifold $S_h$ and
 invariant leaves are given by the manifolds $P^{-1}((h(\eta),\eta)) \cap U$ for each
 $\eta \in \overline{V}$. The following Remark compares
 \cref{thm:red-principle-slow-mfd} to similar Theorems in the literature:
 
 \begin{remark}
In the finite-dimensional case it is known from the Fenichel-theory, see \cite{Fe79}, that for each $\eta \in \overline{V}$, the leaf $P^{-1}((h(\eta),\eta)) \cap U$ 
is actually a $C^k$ manifold even when $P$ is just $C^{k-1}$. For the infinite-dimensional case, there only seems to be a proof for the case $k=1$, see \cite{BaLuZe00}, 
which is based on Hadamard's graph transform. We believe a Lyapunov-perron type approach to $C^k$ smoothness of the leaves to be possible, as each leaf can be interpreted 
as the (strong) stable manifold of a non-autonomous system (with equilibrium at the origin), via the coordinate transformation $(x,y) \mapsto (x,y) - (h(\psi(t; \eta,h)), \psi(t; \eta,h))$. 
Therefore one can mimic existing theory on Lyapunov-perron type approaches to stable manifold theorems for equilibria of infinite-dimensional non-autonomous ODEs, 
see for example \cite{ChLiLu91}. 
Furthermore, it should be noted that Fenichel's original results on invariant foliations \cite{Fe74, Fe77, Fe79} were improved upon by Wiggins
\cite{Wi94}, in the form of less strict smoothness requirements. Particularly, Wiggins seems to only require $k \geq 1$, whereas our theorems here
suppose that $k \geq 2$. It is yet to be seen whether the proofs here can be amended to cover the case $k=1$.
 \end{remark}

To set up a proof for \cref{thm:red-principle-slow-mfd}, it is convenient to
first apply a coordinate transformation to system \cref{ODE-system} by introducing a
new variable $\tilde{x}$ via the transformation $x = \tilde{x} + h(y)$. This maps the
invariant manifold described by the function $h$  bijectively to the set $\{ \tilde x = 0 \}$.
We discuss the properties of the new system, which sets us up for
the next Section. 

We assume that $h \in BC^{k}(\overline{V},\XSet)$, so that in the new coordinates $(\tilde{x},y)$ we obtain the system
\begin{equation}\label{eq:tildeODE}
 \begin{aligned}
 \dot{\tilde{x}}(t) &= \tilde{F}(\tilde{x}(t),y(t)), \\
 \dot{y}(t) &=  \tilde{g}(\tilde{x}(t),y(t)),
 \end{aligned} 
\end{equation}
 where $(\tilde{x},y) \in \XSet \times \overline{V}$ and
 \begin{align*}
 \tilde{F}(\tilde{x},y) &:= F(\tilde{x}+h(y),y) -  Dh(y) g(\tilde{x}+h(y),y), \\
 \tilde{g}(\tilde{x},y) &:= g(\tilde{x}+h(y),y).
 \end{align*}
 By setting
 \begin{align*}
 A_h(y) : &= D_x F(h(y),y), \\
  \tilde{R}(\tilde{x},y) :&=  \tilde{F}(\tilde{x},y) - A_h(y) \tilde{x} \\
  	&= F(\tilde{x}+h(y),y) - Dh(y) g(\tilde{x}+h(y),y) - D_x F(h(y),y) \tilde{x},
  \end{align*}
  we get
   \begin{equation}\label{eq:ODE-system-straight}
 \begin{aligned}
 \dot{\tilde{x}}(t) &= A_h(y(t)) \tilde{x}(t) + \tilde{R}(\tilde{x}(t),y(t)), \\
 \dot{y}(t) &= \tilde{g}(\tilde{x}(t),y(t)).
 \end{aligned} 
 \end{equation}
Observe that
  \begin{enumerate}
\item  $\tilde{F} \in BC^{k-1}(\XSet \times \overline{V}, \XSet)$, $D_{\tilde{x}}  \tilde{F} \in BC^{k-1}(\XSet \times \overline{V}, L(\XSet \times \RSet^n,\XSet))$ and for all $y \in \overline{V}$ we have
 \[
\tilde{F}(0,y) = 0;
 \] 
 \item  For all $(\tilde{x},y) \in \XSet \times \overline{V}$ we have
 \[
 D \tilde{g}(\tilde{x},y) = D g (\tilde{x}+h(y),y) \begin{pmatrix} I_{\XSet} & Dh(y) \\ 0 & I_{\RSet^n} \end{pmatrix},
 \]
 and therefore $\| D\tilde{g} \|_{\infty} \leq (1+\|Dh\|_{\infty}) N_1$. 
 \end{enumerate}
 \vspace{0.2cm}
 
 Moreover, solutions to the system \cref{eq:ODE-system-straight} have the following property:
 \begin{lemma}\label{lem:props-fast-sol}
For every $(\tilde{\xi},\eta) \in \XSet \times \overline{V}$ and $t \geq s \geq 0$ we have
\[
  |\tilde{x}(t;(\tilde{\xi},\eta)) | \leq K e^{-(\mu - K M^x_1) (t-s)} |\tilde{x}(s;(\tilde{\xi},\eta))|.
\]
\end{lemma}

\begin{proof}
Firstly observe that we have for every $\xi \in \XSet$ and $\eta \in \overline{V}$ by
the variation of constants formula applied to \eqref{ODE-system} that for $t \geq s \geq 0$:
\[
x(t,(\xi,\eta)) = T_0(t,s; \eta, h) x(s; (\xi,\eta)) + \int_s^t T_0(t,r; \eta, h) R_0(x(r; (\xi,\eta)),y(r;(\xi,\eta))) \, dr.
\]
And therefore also for $\tilde{x} = x - h(y)$:
\begin{align*}
\tilde{x}(t,(\tilde{\xi},\eta)) &= T_0(t,s; \eta, h) \tilde{x}(s; (\tilde{\xi},\eta)) + \int_s^t T_0(t,r; \eta, h) \Delta R (x(r; (\xi,\eta)),y(r;(\xi,\eta))) \, ds,
\end{align*}
where
\[
\Delta R (x,y) :=  R_0(x,y) - R_0(h(y), y).
\]
As $|\Delta R (x,y)| \leq M^x_1 |x - h(y)| = M^x_1 |\tilde{x}|$ (see assumption (H2) from \cref{sec:slow-mfds-existence}) we now obtain the estimate
\[
|\tilde{x}(t,(\tilde{\xi},\eta)) | \leq e^{-\mu (t-s)} \left( K |\tilde{x}(s; (\tilde{\xi},\eta))| + K M^x_1 \int_s^t e^{-\mu (t-r)} |\tilde{x}(r,(\tilde{\xi},\eta)) |  \, dr   \right),
\]
which pertains to the solutions of \cref{eq:tildeODE}. Rearranging terms and applying Gr\"onwall's inequality gives
\[
e^{\mu (t-s)} |\tilde{x}(t;(\tilde{\xi},\eta)) | \leq K |\tilde{x}(s; (\tilde{\xi},\eta))|  e^{K M^x_1 (t-s)},
\]
and therefore 
\[
 |\tilde{x}(t;(\tilde{\xi},\eta)) | \leq K e^{-(\mu - K M^x_1) (t-s)} |\tilde{x}(s; (\tilde{\xi},\eta))|  \quad \text{for all } t \geq s \geq 0.
\]
\end{proof}

Via an induction proof for the variational equations, which we omit here for brevity, we also have the following
estimates, which we need for the smoothness proofs in \cref{sec:red-map-smoothness}:
\begin{lemma}\label{lem:ODE-system-derivative-bounds}
For every $(\tilde{\xi},\eta) \in U \times \overline{V}$, $1 \leq l \leq k$ and $t \geq s \geq 0$ we have
\[
	 \| \partial^l_{(\tilde{\xi},\eta)} \tilde{x}(t;(\tilde{\xi},\eta)) \| \leq K^x_l e^{-(\mu - K M^x_1 - l(\|Dh\|_{\infty} + 1) N_1) (t-s)} \| \partial^l_{(\tilde{\xi},\eta)} \tilde{x}(s;(\tilde{\xi},\eta)) \|,
\]
as well as
\[
	\| \partial^l_{(\tilde{\xi},\eta)} y(t;(\tilde{\xi},\eta)) \| \leq K^y_l e^{l (\|Dh\|_{\infty} + 1) N_1 (t-s)} \| \partial^l_{(\tilde{\xi},\eta)} y(s;(\tilde{\xi},\eta)) \|,
\]
for some constants $K^x_l, K^y_l \geq 1$. Moreover, $\partial^l_{(\tilde{\xi},\eta)} \tilde{x}(t;(\tilde{\xi},\eta))$ and $\partial^l_{(\tilde{\xi},\eta)} y(t;(\tilde{\xi},\eta))$ are continuous with respect to $(\tilde{\xi},\eta)$ for every $1 \leq l \leq k$ and $t \geq 0$. 
\end{lemma}

The conclusions from the next two sections can be applied to this transformed $(\tilde{x},y)$-system. In particular, 
\cref{thm:red-principle-slow-mfd} follows by applying \cref{thm:red-map-k-smoothness}  to the $(\tilde{x},y)$-system.


 \section{Existence of reduction map}\label{sec:red-map-existence}
 
 As by the previous Section, to prove existence of the reduction map, without loss of generality we now consider systems of the form
 \begin{equation}\label{ODE-system-red}
 \begin{aligned}
 \dot{x}(t) &= F(x(t),y(t)), \\
 \dot{y}(t) &= g(x(t),y(t)),
 \end{aligned}
 \end{equation}
 where $x \in \XSet$, $y \in \overline{V}$ with $V$ an open subset of $\RSet^n  (n \geq 1)$, $F \in C^0(\XSet \times \overline{V},\XSet)$, $g \in C^0(\XSet\times \overline{V}, \RSet^n)$ and $F,g$ are globally Lipschitz. Also we have $g(x,y)=0$ whenever $y \in \partial V$.
 
Now we introduce the following assumptions (S1)-(S2):
  \begin{hypothesis}\label{hyp:ODE-system-red}
For the system \cref{ODE-system-red} it holds that:
 \begin{enumerate}
 \item[$\mathrm{(S1)}$] There exist $\mu >0$ and $K \geq 1$ such that every solution of system \cref{ODE-system-red}, with $(x,y)(0) = (\xi,\eta)$ for some $(\xi,\eta) \in \XSet \times \overline{V}$, has the property that 
\begin{equation}\label{eq:process-exp-bdd-red}
  |x(t;(\xi,\eta)) | \leq K e^{-\mu (t-s)} |x(s;(\xi,\eta))|
\end{equation}
for all $t \geq s \geq 0$;
 \item[$\mathrm{(S2)}$]  There exists $N_1 > 0$, with $K N_1 < \mu $, such that
 \[
 | g(x_2,y_2) - g(x_1,y_1) | \leq N_1 |(x_2,y_2) - (x_1,y_1)|
 \]
 for all $(x_1,y_1), (x_2,y_2) \in \XSet \times \overline{V}$.
 \end{enumerate}
 \end{hypothesis}
Note that (S1) implies that we have $F(0,y) = 0 $ for all $y \in \overline{V}$, and therefore $\{ x = 0 \}$ is invariant under the flow of system \cref{ODE-system-red}.

The central result of this section is the following theorem:
 \begin{theorem}[Existence of Reduction Map]\label{thm:red-map-existence}
Assume System \cref{ODE-system-red} satisfies assumptions (S1)-(S2). Then there
exists a map $P \in C^0(\XSet \times \overline{V}, \overline{V})$ such that for each
$(\xi,\eta) \in \XSet \times \overline{V}$ it holds that 
\[
  \begin{aligned}
    & P(x(t; (\xi,\eta)),y(t;(\xi,\eta))) = y(t; (0,P(\xi,\eta))),  
      && \textit{for all $t \geq 0$,} \\
    & \lim_{t \rightarrow \infty} e^{\gamma t}  | y(t;(\xi,\eta)) - y(t; (0,P(\xi,\eta))) | = 0,
      && \textit{for any $0 \leq \gamma < \mu $,} \\
    & \sup \left\{ \frac{|\eta - P(\xi,\eta)|}{|\xi|} 
      \colon (\xi,\eta) \in \XSet \times \overline{V}, \xi \neq 0   \right\} < \infty.
  \end{aligned}
\]
 \end{theorem}
 We shall also refer to a map $P \in C^0(\XSet \times \overline{V}, \overline{V})$
 satisfying the first two properties from this Theorem as a reduction map for system
 \cref{ODE-system-red}. The proof of \cref{thm:red-map-existence} follows directly from two upcoming
 results, namely \cref{cor: red-map-contraction} and \cref{lem:red-map-props}, which
 we now address. 
 
 To set up a strategy for the proof of \cref{thm:red-map-existence}, first suppose that the Theorem is true. 
 Because of the first property we need to have
 \[
 P(x(t; (\xi,\eta)),y(t;(\xi,\eta))) = P(\xi,\eta) +\int_0^t g(0, P(x(s; (\xi,\eta)),y(s;(\xi,\eta))) ) \, ds.
 \]
 Then from the second property we obtain
 \begin{align*}
 \lim_{t \rightarrow \infty} e^{\gamma t}  \Big| \eta - P(\xi,\eta) + \int_0^t &g(x(s; (\xi,\eta)),y(s;(\xi,\eta))) \\
  &\quad - g(0, P(x(s; (\xi,\eta)),y(s;(\xi,\eta))) ) \, ds \Big| = 0.
 \end{align*}
 This can only possibly happen when
 \[
 P(\xi,\eta) = \eta + \int_0^{\infty} g(x(s; (\xi,\eta)),y(s;(\xi,\eta))) - g(0, P(x(s; (\xi,\eta)),y(s;(\xi,\eta))) ) \, ds.
 \]
 Set now $Q(\xi,\eta) :=  \eta - P(\xi,\eta) = (\pi_y - P)(\xi,\eta)$ to obtain
 \[
 Q(\xi,\eta) = \int_0^{\infty} g(0, (\pi_y -Q)(x(s; (\xi,\eta)),y(s;(\xi,\eta)))) - g(x(s; (\xi,\eta)),y(s;(\xi,\eta))) \, ds.
 \]
 
 To show existence of a map $Q$ which satisfies this integral equation, we set up a contraction mapping argument again. For this we naturally also restrict now to maps satisfying the property $Q(0,\eta) = 0$ for all $\eta \in \overline{V}$. 
 
 Define
 \[
 \mathcal{E}^0 := \left\{ Q \in C^{0}(\XSet \times \overline{V}, \overline{V}): Q(0, \cdot) =0;  \sup_{\substack{(\xi,\eta) \in \XSet \times \overline{V}; \\ \xi \neq 0}} \frac{|Q(\xi,\eta)|}{|\xi|} < \infty  \right\}.
 \]
This is a Banach space with norm
 \[
 \| Q \|_{\mathcal{E}} := \sup \left\{ \frac{|Q(\xi,\eta)|}{|\xi|} : (\xi,\eta) \in \XSet \times \overline{V}, \xi \neq 0   \right\},
 \]
 see also \cref{app:C}. Observe that we have for all $(\xi,\eta) \in \XSet \times \overline{V}$ that
 \[
 Q(\xi,\eta) \leq  \| Q \|_{\mathcal{E}} |\xi|.
 \]
 
 \begin{remark}
The set-up for this function space was partly inspired by the proofs from Chapter 4 of Chicone \cite{Ch06}.
\end{remark}
 
 We now define a map $\Lambda$ on (a subset of) $C^{0}(\XSet \times \overline{V})$ as
\[
\Lambda(Q)(\xi,\eta) := \int_0^{\infty} g(0, (\pi_y -Q)(x(s; (\xi,\eta)),y(s;(\xi,\eta)))) - g(x(s; (\xi,\eta)),y(s;(\xi,\eta))) \, ds.
\]

The following consequence of assumption (S1) prepares us for showing that $\Lambda$ is well-defined and a contraction mapping on the Banach space $ \mathcal{E}^0$. 
\begin{lemma}\label{lem:props-fast-sol-int}
Under (S1), for every $(\xi,\eta) \in \XSet \times \overline{V}$ and $t \geq 0$ we have
\[
\int_t^{\infty} |x(s;(\xi,\eta))| \, ds \leq \frac{K}{\mu} |x(t;(\xi,\eta))|
\]
In particular it holds for every $(\xi,\eta) \in \XSet \times \overline{V}$ that
\[
\int_0^{\infty} |x(s;(\xi,\eta))| \, ds \leq \frac{K}{\mu} |\xi|.
\]
\end{lemma}

\begin{proof}
From the inequality \cref{eq:process-exp-bdd-red} from assumption (S1) we derive immediately that for every $(\xi,\eta) \in \XSet \times \overline{V}$ and $t \geq 0$ we have
\begin{align*}
\int_t^{\infty} |x(s;(\xi,\eta))| \, ds \leq K x(t;(\xi,\eta)) \int_t^{\infty} e^{-\mu(s-t)} \, ds = \frac{K}{\mu} |x(t;(\xi,\eta))|.
\end{align*}
Setting $t=0$ gives the last statement. 
\end{proof}

We are now ready to firstly show that $\Lambda$ is well-defined on $\mathcal{E}^0$, and next that it is also a contraction on $\mathcal{E}^0$ under the assumptions (S1)-(S2). 

\begin{lemma}\label{lem:red-map-well-defined}
Under (S1), the mapping $\Lambda$ is well-defined on $\mathcal{E}^0$, and moreover
\[
  \| \Lambda(Q) \|_{\mathcal{E}} \leq \frac{K N_1}{\mu} (1+  \| Q \|_{\mathcal{E}}).
\]
\end{lemma}

\begin{proof} 
Clearly for every $Q \in \mathcal{E}^0$ we have $\Lambda(Q)(0,\cdot) = 0$. Secondly,
we have for every $(\xi,\eta) \in \XSet \times \overline{V}$, by applying \cref{lem:props-fast-sol-int}:
\begin{align*}
|\Lambda(Q)(\xi,\eta)| &\leq N_1 \int_0^{\infty}  |x(s; (\xi,\eta))| + |Q(x(s; (\xi,\eta)),y(s;(\xi,\eta)))| \, ds \\
&\leq N_1 (1+  \| Q \|_{\mathcal{E}})  \int_0^{\infty}  |x(s; (\xi,\eta))| \, ds \\
&\leq \frac{K N_1}{\mu} (1+  \| Q \|_{\mathcal{E}})  |\xi|.
\end{align*}
And therefore $\| \Lambda(Q) \|_{\mathcal{E}} \leq \frac{K N_1}{\mu} (1+  \| Q \|_{\mathcal{E}}) < \infty$ for every $Q \in \mathcal{E}^0$.

Thirdly, we still have to show that $\Lambda(Q)$ is continuous when $Q \in \mathcal{E}^0$. For this, set for arbitrary $(\xi,\eta) \in \XSet \times \overline{V}$ and $ t \geq 0$:
\[
q(t;(\xi,\eta)) := Q(x(t;(\xi,\eta)),y(t;(\xi,\eta)),
\]
so that
\begin{align*}
\Lambda (q(t;\blank)) (\xi, \eta) &= \Lambda(Q(x(t;\blank),y(t;\blank)))(\xi,\eta) \\
 &= \int_t^{\infty} g(0, (\pi_y -Q)(x(s; (\xi,\eta)),y(s;(\xi,\eta)))) \\
 &\qquad \qquad \qquad \qquad \qquad- g(x(s; (\xi,\eta)),y(s;(\xi,\eta))) \, ds.
 \end{align*}
Now for each $(\xi_0,\eta_0) \in \XSet \times \overline{V}$ we have
\begin{equation}\label{eq:int-ineq-Lambda-q}
\begin{aligned}
 &| \Lambda (q(t;\blank))(\xi,\eta) -  \Lambda (q(t;\blank))(\xi_0,\eta_0) | \\
 &\qquad  \leq 2 N_1 \int_t^{\infty} |x(s; (\xi,\eta)) - x(s; (\xi_0,\eta_0))| +  |y(s; (\xi,\eta)) - y(s; (\xi_0,\eta_0))| \\
&\qquad \qquad \qquad \qquad \qquad \qquad \qquad \qquad \qquad \quad  +  |q(s; (\xi,\eta)) - q(s; (\xi_0,\eta_0))|   \, ds.
 \end{aligned}
 \end{equation}
 For $\gamma > N_1$ we have
 \[
 \lim_{t \rightarrow \infty} e^{-\gamma t} |y(t; (\xi,\eta)) - y(t; (\xi_0,\eta_0))| = 0,
 \]
 so that
 \[
 \sup_{t \geq 0} e^{-\gamma t} |y(t; (\xi,\eta)) - y(t; (\xi_0,\eta_0))| \rightarrow 0 \quad \text{as } (\xi,\eta) \rightarrow (\xi_0,\eta_0).
 \]
Analogous statements hold when we replace $y$ by $x, q$. Thereby we derive from the inequality \cref{eq:int-ineq-Lambda-q} that
 \[
\sup_{t \geq 0} e^{-\gamma t} | \Lambda (q(t;\blank))(\xi,\eta) -  \Lambda (q(t;\blank))(\xi_0,\eta_0) | \rightarrow 0 \quad \text{as } (\xi,\eta) \rightarrow (\xi_0,\eta_0).
 \]
 
 Observe that $Q(\xi,\eta) = q(0;(\xi,\eta))$, and therefore we must have
 \[
 |\Lambda(Q)(\xi,\eta) - \Lambda(Q)(\xi_0,\eta_0)| \rightarrow 0 \quad \text{as } (\xi,\eta) \rightarrow (\xi_0,\eta_0),
 \]
 which means $\Lambda(Q)$ is continuous. 
\end{proof}

\begin{lemma}\label{lem: red-map-contraction}
  Under hypotheses (S1)--(S2) the mapping $\Lambda$ is a contraction on $\mathcal{E}^0$.
\end{lemma}

\begin{proof}
Let $Q_1, Q_2 \in \mathcal{E}^0$, now we obtain by \cref{lem:props-fast-sol-int} for each $(\xi,\eta) \in \XSet \times \overline{V}$:
\begin{align*}
|\Lambda(Q_2)(\xi,\eta) - \Lambda(Q_1)(\xi,\eta)| &\leq N_1 \| Q_2 - Q_1 \|_{\mathcal{E}} \int_0^{\infty} |x(s; (\xi,\eta))| \, ds \\
&\leq \frac{K N_1 |\xi|}{\mu}  \| Q_2 - Q_1 \|_{\mathcal{E}}.
\end{align*}
And we therefore conclude 
\[
\| \Lambda(Q_2) - \Lambda(Q_1) \|_{\mathcal{E}} \leq \frac{K N_1}{\mu}  \| Q_2 - Q_1 \|_{\mathcal{E}}.
\]
As by assumption (S2) we have $\frac{K N_1}{\mu} < 1$, the Lemma follows.
\end{proof}

\begin{corollary}\label{cor: red-map-contraction}
Suppose (S1) and (S2) hold. Then the mapping $\Lambda$ has a unique fixed point $Q \in \mathcal{E}^0$. 
\end{corollary}

\begin{proof}
Apply the Contraction Mapping Theorem to the results of
\cref{lem:red-map-well-defined,lem: red-map-contraction}.
\end{proof}

The following Lemma states that a fixed point of $\Lambda$ indeed provides a reduction map for system \cref{ODE-system-red} under assumptions (S1)-(S2):

\begin{lemma}\label{lem:red-map-props}
Assume (S1)-(S2) are valid. Suppose $Q \in  \mathcal{E}^0$ is such that $\Lambda(Q)=\Lambda$, and set $P(\xi,\eta) = \eta - Q(\xi,\eta)$. Then for every $(\xi,\eta) \in \XSet \times \overline{V}$ we have:
\begin{enumerate}
 \item For all $t \geq 0$ it holds that 
 \[
 \partial_t P(x(t; (\xi,\eta)),y(t;(\xi,\eta))) = g(0, P(x(t; (\xi,\eta)),y(t;(\xi,\eta))));
 \]
 \item $| Q(x(t; (\xi,\eta)),y(t;(\xi,\eta))) | \leq \frac{K^2 N_1}{\mu - K N_1} e^{-\mu t}  |\xi|$ for all $t \geq 0$, and thereby also
 \[
 \sup_{t \geq 0} e^{\mu t} | y(t;(\xi,\eta)) -  P(x(t; (\xi,\eta)),y(t;(\xi,\eta))) | \leq  \frac{K^2 N_1}{\mu - K N_1} |\xi|.
 \]
\end{enumerate}
\end{lemma}

\begin{proof}
We firstly derive from the assumption $\Lambda(Q) = Q$ via a change of variables that for all $(\xi,\eta) \in \XSet \times \overline{V}$ we have for $t \geq 0$ :
\begin{align*}
Q(x(t; (\xi,\eta)),y(t;(\xi,\eta))) &= \int_0^{\infty} g(0, (\pi_y-Q)(x(s+t; (\xi,\eta)),y(s+t;(\xi,\eta))) ) \\
& \qquad \qquad \quad  - g(x(s+t; (\xi,\eta)),y(s+t;(\xi,\eta)))    \, ds \\
&=  \int_t^{\infty} g(0, (\pi_y-Q)(x(s; (\xi,\eta)),y(s;(\xi,\eta))) ) \\
& \qquad \qquad \quad  - g(x(s; (\xi,\eta)),y(s;(\xi,\eta)))    \, ds.
\end{align*}
\begin{enumerate}
\item Substituting $\pi_y-Q=P$ now gives for $t \geq 0$ that
\begin{align*}
P(x(t; &(\xi,\eta)),y(t;(\xi,\eta))) \\
&= y(t;(\xi,\eta)) + \int_t^{\infty} g(x(s; (\xi,\eta)),y(s;(\xi,\eta))) \\
&\qquad \qquad \qquad \qquad \qquad - g(0, P(x(s; (\xi,\eta)),y(s;(\xi,\eta))) )   \, ds \\
&=  P(\xi,\eta) + \int_0^t  g(0, P(x(s; (\xi,\eta)),y(s;(\xi,\eta))) ) \, ds.
\end{align*}
By applying Leibniz' integral rule we then obtain the desired result.
\item By \cref{lem:red-map-well-defined} we derive that if $\Lambda(Q) = Q$, then
 \[
\| Q \|_{\mathcal{E}} \leq \frac{K N_1}{\mu - K N_1}.
\]
Therefore we have by (S1) for $t \geq 0$ that
\[
| Q(x(t; (\xi,\eta)),y(t;(\xi,\eta))) | \leq \frac{K^2 N_1}{\mu - K N_1} e^{-\mu t}  |\xi|.
\]
This implies the assertion.
\end{enumerate}
\end{proof}

We are now ready to deduce the validity of \cref{thm:red-map-existence}:

\begin{proof}[Proof of \cref{thm:red-map-existence}]
  Combine \cref{cor: red-map-contraction,lem:red-map-props}.
  The third property follows from the definition of the space $\mathcal{E}^0$.
\end{proof}


 \section{Smoothness of reduction map}\label{sec:red-map-smoothness}

Under some additional assumptions on \cref{ODE-system-red}, the map $P \in C^0(\XSet \times \overline{V}, \overline{V})$
 from \cref{thm:red-map-existence} is also contained in $C^k(U \times \overline{V}, \overline{V})$
 for some open set $U$ in $\XSet$ such that $0 \in U$, whereby $k \geq 2$ is the degree of smoothness of $g$ with respect to $(x,y)$. 
 In this section we present this statement as a Theorem, although we prove it in detail only for the case $k=2$ for expository reasons.
 We remark that the general case then follows from an induction argument
 comparable to the one in \cref{sec:k-smooth-slow-mfds}. The induction step
 follows similarly to the proof of \cref{lem:red-map-C2} under the assumption of \cref{lem:red-map-C1}, 
 as given below. The results are captured in \cref{thm:red-map-k-smoothness}, 
 which we state after providing its Hypotheses that are additional to those of the previous section.
 
 \begin{hypothesis}\label{hyp:red-map-l-derivatives}
Suppose $k \geq 2$, and that for some $1 \leq l \leq k$ we have:
\begin{enumerate}
\item 
  The functions $\partial^l_{(\xi,\eta)} x$ and $\partial^l_{(\xi,\eta) }y$
  exist and are continuous with respect to $(\xi,\eta)$. Further, there exist
  constants $K^x_l, K^y_l \geq 1$ such that for all $t \geq s \geq 0$ 
  \[
    \begin{aligned}
  & \| \partial^l_{(\xi,\eta) }x(t;(\xi,\eta)) \| 
    \leq K^x_l e^{-(\mu - l N_1) (t-s)} \|x(s;(\xi,\eta))\|, \\
  & \| \partial^l_{(\xi,\eta) }y(t;(\xi,\eta)) \| \leq 
  K^y_l e^{l N_1 (t-s)} \|y(s;(\xi,\eta))\|.
    \end{aligned}
  \]
\item 
  It holds $g \in C^k(U \times \overline{V}, \RSet^n)$. Further there exist 
  constants $N_1, \ldots, N_l$ such that $(l+1)N_1 < \mu$ and
\[
  \sup \{\|D^i g(x,y)\| : (x,y) \in U \times \overline{V} \} \leq N_i, \qquad \text{for all $ 1 \leq i \leq l$.}
\]
\end{enumerate}
\end{hypothesis}

\begin{remark}
These Hypotheses hold for system \cref{ODE-system} as stated in \cref{lem:ODE-system-derivative-bounds}, and therefore our results here prove \cref{thm:red-principle-slow-mfd}. 
Notice that the estimates from \cref{lem:ODE-system-derivative-bounds} can be translated into those of the above Hypotheses by picking $\mu$ and $N_1$ appropriately. 
\end{remark}
 
 \begin{theorem}\label{thm:red-map-k-smoothness}
Suppose system \cref{ODE-system-red} satisfies the assumptions (S1)-(S2) from \cref{thm:red-map-existence}. 
Assume additionally that \cref{hyp:red-map-l-derivatives} is satisfied for certain $k \geq 2$ and $1 \leq l \leq k$.
Then for the map $P \in C^0(\XSet \times \overline{V}, \overline{V})$ from \cref{thm:red-map-existence} it holds as well that
\[
P \in C^{l}(U \times \overline{V}, \overline{V}).
\]
\end{theorem}

To prove the particular case $k=l=2$ of this Theorem, we first prove that 
$(\xi,\eta) \mapsto P(\xi,\eta)$ is continuously differentiable when $D_y g$ 
satisfies a uniform Lipschitz condition as stated in the following Lemma:
 
 \begin{lemma}\label{lem:red-map-C1}
Suppose system \cref{ODE-system-red} satisfies the assumptions (S1)-(S2) from \cref{thm:red-map-existence}. Assume additionally that
\begin{enumerate}
\item \cref{hyp:red-map-l-derivatives} holds for $k=2$ and $l=1$
\item The map
 \begin{align*}
 D_y g : U \times \overline{V} &\rightarrow L(\RSet^n,\RSet^n) \, : \\
U \times \overline{V} \ni (x,y) &\mapsto D_y g(x, y) \in L( \RSet^n,\RSet^n)
 \end{align*}
is continuously differentiable at every point $(0,y) \in U \times V$ and there exists $N_2 > 0$ such that
\[
\sup_{x \in U, \, y_1,y_2 \in V} \| D_y g(x,y_2) - D_y g(0,y_1)  \| \leq N_2 (|x| + |y_2 - y_1|).
\]
\end{enumerate}
 Then for the map $P \in C^0(\XSet \times \overline{V}, \overline{V})$ from \cref{thm:red-map-existence} it holds moreover that $P \in  C^1(U \times \overline{V}, \overline{V})$.
 \end{lemma}
 
To construct a proof for this Lemma, henceforth for $(\xi,\eta) \in \XSet \times
\overline{V}$ we set $q(t;(\xi,\eta)) := Q(x(t; (\xi,\eta)),y(t;(\xi,\eta)))$. From
\cref{lem:red-map-props} we know that $q$ satisfies the IVP
\begin{align*}
\partial_t q(t;(\xi,\eta)) &= g(x(t;(\xi,\eta)),y(t;(\xi,\eta))) - g(0,y(t;(\xi,\eta)) - q(t;(\xi,\eta))); \\
q(0;(\xi,\eta)) &= Q(\xi,\eta),
\end{align*}
where $Q(\xi, \eta) = \eta - P(\xi,\eta)$. 

In order to set up a proof for the smoothness of $Q$ (and thereby \cref{lem:red-map-C1}), first suppose that $Q$ is differentiable. Then it is straightforward to verify that $w(t;(\xi,\eta)) := \partial_{(\xi,\eta)} q(t;(\xi,\eta))$ exists for all $t \geq 0$ and satisfies the variational equation
\begin{align*}
\partial_t w(t;(\xi,\eta)) &= D_y g(0,p(t;(\xi,\eta))) w(t;(\xi,\eta)) \\
&\qquad + D g(x(t;(\xi,\eta)),y(t;(\xi,\eta))) \begin{pmatrix} D_{(\xi,\eta)}x(t; (\xi,\eta)) \\ D_{(\xi,\eta)}y(t; (\xi,\eta)) \end{pmatrix}  \\
&\qquad -  D_y g(0,p(t;(\xi,\eta))) D_{(\xi,\eta)}y(t; (\xi,\eta)),
\end{align*}
where $p(t; (\xi,\eta)) := y(t;(\xi,\eta)) - q(t; (\xi,\eta))$, and $w(s;(\xi,\eta)) \in L(\XSet \times \RSet^n, \RSet^n)$ for arbitrary $s \geq 0$.

Now let $Z(t,s; (\xi,\eta))$ be the reversible process associated to the non-autonomous linear system
\[
\partial_t \tilde{w}(t) = D_y g(0,p(t;(\xi,\eta))) \tilde{w}(t).
\]
Observe that $ \|Z(t,s; (\xi,\eta))\| \leq e^{N_1 |t-s|}$ for $t,s \geq 0$ because of (S2). 

By the variation of constants formula we obtain for $t \geq s \geq 0$:
\begin{align*}
w(t; (\xi,\eta)) &= Z(t,s; (\xi,\eta)) w(s;(\xi,\eta)) \\
&\qquad - \int_s^t Z(t,r; (\xi,\eta)) \Bigg( D_y g(0,p(r;(\xi,\eta))) \partial_{(\xi,\eta)}y(r; (\xi,\eta))  \\
&\qquad \qquad \quad - D g(x(r;(\xi,\eta)),y(r;(\xi,\eta))) \begin{pmatrix} \partial_{(\xi,\eta)}x(r; (\xi,\eta)) \\ \partial_{(\xi,\eta)}y(r; (\xi,\eta)) \end{pmatrix} \Bigg) \, dr.
\end{align*}
Under the assumptions of \cref{lem:red-map-C1} we can derive for $s \geq 0$ that
 \[
 \lim_{t \rightarrow \infty}  \|  Z(s,t, (\xi,\eta)) w(t; (\xi,\eta)) \| = 0,
 \]
  and thereby for any $t \geq 0$ we must have
\begin{align*}
w(t;(\xi,\eta))  &=  \int_t^{\infty} Z(t,s; (\xi,\eta)) \Bigg( D_y g(0,y(s;(\xi,\eta)) - q(s;(\xi,\eta))) \partial_{(\xi,\eta)}y(s; (\xi,\eta))  \\
&\qquad \qquad - D g(x(s;(\xi,\eta)),y(s;(\xi,\eta))) \begin{pmatrix} \partial_{(\xi,\eta)}x(s; (\xi,\eta)) \\ \partial_{(\xi,\eta)}y(s; (\xi,\eta)) \end{pmatrix} \Bigg) \, ds.
\end{align*}

This motivates the following proof approach to \cref{lem:red-map-C1}:
\begin{proof}[Proof of \cref{lem:red-map-C1}]
We firstly define a map
\[
Q^1 \in C^0(U \times \overline{V}, L(\XSet \times \RSet^n, \RSet^n)),
\]
which is now our candidate for the derivative of $(\xi,\eta) \mapsto \eta - P(\xi,\eta)$, as
\begin{align*}
Q^1(\xi,\eta) &:=  \int_0^{\infty} Z(0,s; (\xi,\eta)) \Bigg( D_y g(0,y(s;(\xi,\eta)) - q(s;(\xi,\eta))) \partial_{(\xi,\eta)}y(s; (\xi,\eta))  \\
&\qquad \qquad - D g(x(s;(\xi,\eta)),y(s;(\xi,\eta))) \begin{pmatrix} \partial_{(\xi,\eta)}x(s; (\xi,\eta)) \\ \partial_{(\xi,\eta)}y(s; (\xi,\eta)) \end{pmatrix} \Bigg) \, ds.
\end{align*}
We firstly have to show that $Q^1$ is well-defined in this manner. Using the
assumptions from \cref{lem:red-map-C1} we have the estimate
\begin{align*}
\| Q^1(\xi,\eta) \| &\leq  \int_0^{\infty} \| Z(0,s, (\xi,\eta)) \| \bigg( N_1 \| \partial_{(\xi,\eta)}x(s; (\xi,\eta)) \| \\
&\qquad \qquad \qquad + N_2 \left( |x(s; (\xi,\eta))| + |q(s; (\xi,\eta))| \right) \| \partial_{(\xi,\eta)}y(s; (\xi,\eta)) \| \bigg)  \, ds \\
&\leq \left( \frac{K^x_1 N_1}{\mu -  2 N_1} + \frac{K N_2 }{\mu - 2N_1} + \frac{K^2 K^y_1 N_1 N_2}{(\mu - K N_1)(\mu - 2 N_1)} \right) |\xi|,
\end{align*}
and thereby
\[
\sup_{\substack{(\xi,\eta) \in U \times \overline{V}; \\ \xi \neq 0}} \frac{\|Q^1(\xi,\eta)\|}{|\xi|} < \infty.
\]

It can be shown that $Q^1$ depends continuously on its arguments in a manner
comparable to the proof of \cref{lem:red-map-well-defined}. Therefore
\[
Q^1 \in \left\{ \tilde{Q} \in C^{0}(U \times \overline{V}, L(\XSet \times \RSet^n, \RSet^n)): \tilde{Q}(0, \blank) = 0 ;  \sup_{\substack{(\xi,\eta) \in U \times \overline{V}; \\ \xi \neq 0}} \frac{\|\tilde{Q}(\xi,\eta)\|}{|\xi|} < \infty  \right\}.
\]

Now for $q^1(t;\xi,\eta) := Q^1(x(t;(\xi,\eta)),y(t;(\xi,\eta)))$ set
\[
p^1(t;(\xi,\eta)) := \begin{pmatrix} 0 \\ I_{\RSet^n} \end{pmatrix} - q^1(t;(\xi,\eta)), \quad P^1(\xi,\eta) := p^1(0;(\xi,\eta)).
\]
It remains to show that $P^1$ is indeed the derivative of $P$. For this, we define for $t \geq 0$, $(\xi,\eta) \in U \times \overline{V}$ and $\Delta \in \XSet \times \RSet^n$:
\begin{align*}
a(t; (\xi,\eta), \Delta) :&= p(t; (\xi,\eta) + \Delta) - p(t; (\xi,\eta)) - p^1(t; (\xi,\eta)) \, \Delta.
\end{align*}
The quantity $a(t; (\eta,\xi), \Delta)$ has to satisfy the ODE
\begin{equation*}
\begin{aligned}
\partial_t \, a(t; (\xi,\eta), \Delta) &= g(0,p(t;(\xi,\eta)+\Delta)) - g(0,p(t;(\xi,\eta))) \\
&\qquad \qquad \qquad - D_y g(0,p(t;(\xi,\eta))) \, p^1(t;(\xi,\eta)) \, \Delta.
\end{aligned}
\end{equation*}

If we set
\begin{align*}
\hat{g}(t; (\xi,\eta), \Delta) :&= g(0,p(t;(\xi,\eta)+\Delta)) - g(0,p(t;(\xi,\eta))) \\
&\qquad \qquad \qquad - D_y g(0,p(t;(\xi,\eta))) \left( p(t;(\xi,\eta)+\Delta) - p(t;(\xi,\eta))  \right),
\end{align*}
we can rewrite the ODE for $a(t; (\xi,\eta), \Delta)$ as
\begin{equation}\label{eq:a-ODE}
\begin{aligned}
\partial_t \, a(t; (\xi,\eta), \Delta) &= D_y g(0,p(t;(\xi,\eta))) \, a(t; (\xi,\eta), \Delta) + \hat{g}(t; (\xi,\eta), \Delta) .
\end{aligned}
\end{equation}
Observe that by applying the Mean Value Theorem in integral form we may obtain
\[
| \hat{g}(t; (\xi,\eta), \Delta) | \leq \frac{N_2}{2} \left| p(t;(\xi,\eta)+\Delta) - p(t;(\xi,\eta))  \right|^2.
\]
By the first property \cref{lem:red-map-props}, and because of the second item of \cref{hyp:red-map-l-derivatives}, it holds that
\[
| p(t;(\xi,\eta)) | \leq | P(\xi,\eta) | e^{N_1|t|}.
\]
We therefore have for $\gamma > 2N_1$:
\[
\sup_{t \geq 0} e^{-\gamma t} | \hat{g}(t; (\xi,\eta), \Delta) | = o(|\Delta|).
\]

The relevant solutions to the ODE \cref{eq:a-ODE} satisfy the integral equation
\[
 a(t; (\xi,\eta), \Delta) =   \int_t^{\infty} Z(t,s; (\xi,\eta)) \, \hat{g}(s; (\xi,\eta), \Delta) \, ds,
\]
and therefore
\[
\sup_{t \geq 0} e^{-\gamma t} |  a(t; (\xi,\eta), \Delta)  | = o(|\Delta|).
\]
As $a(0; (\xi,\eta), \Delta) = P((\xi,\eta)+\Delta) - P(\xi,\eta) - P^1(\xi,\eta) \Delta$ we in particular have
\[
\left| P((\xi,\eta)+\Delta) - P(\xi,\eta) - P^1(\xi,\eta) \Delta \right| =   o(|\Delta|),
\]
and thereby $P^1 = DP$, as was desired. 
\end{proof}

Having proven \cref{lem:red-map-C1}, we are ready to prove \cref{thm:red-map-k-smoothness} for the case $k=2$. What remains to be shown is captured by the following Lemma.

\begin{lemma}\label{lem:red-map-C2}
Suppose system \cref{ODE-system-red} satisfies the assumptions (S1)-(S2) from \cref{thm:red-map-existence}. 
Assume additionally that \cref{hyp:red-map-l-derivatives} holds for $k=l=2$. Then for the map 
$P \in C^0(\XSet \times \overline{V}, \overline{V})$ from \cref{thm:red-map-existence} it holds moreover that 
$P \in  C^2(U \times \overline{V}, \overline{V})$.
\end{lemma}

To set up a proof, we first let $q^1(t; (\xi,\eta)) := \partial_{(\xi,\eta)}
q(t;(\xi,\eta))$, and the existence of this first derivative follows now from \cref{lem:red-map-C1}. Then we set $w_2(t; (\xi,\eta)) := \partial^2_{(\xi,\eta)} q(t;(\xi,\eta))$ 
for $t \geq 0$. We can derive that $w_2(t; (\xi,\eta))$ exists for $t \geq 0$ and has to satisfy the second-order variational equation
\begin{align*}
\partial_t w_2(t;(\xi,\eta)) &= D_y g(0,p(t;(\xi,\eta))) w_2(t;(\xi,\eta)) \\
&\qquad + D g(x(t;(\xi,\eta)),y(s;(\xi,\eta))) \begin{pmatrix} \partial^2_{(\xi,\eta)}x(t; (\xi,\eta)) \\ \partial^2_{(\xi,\eta)}y(t; (\xi,\eta)) \end{pmatrix} \\
&\qquad -  D_y g(0,p(t;(\xi,\eta))) \partial^2_{(\xi,\eta)}y(t; (\xi,\eta)) + H(t; (\xi,\eta)), 
\end{align*}
where
\[
\begin{aligned}
	H(t; (\xi,\eta))  
		&:=D^2 g(x(t;(\xi,\eta)),y(s;(\xi,\eta))) \begin{pmatrix} \partial_{(\xi,\eta)}x(t; (\xi,\eta)) \\ \partial_{(\xi,\eta)}y(t; (\xi,\eta)) \end{pmatrix}  
						\begin{pmatrix} \partial_{(\xi,\eta)}x(t; (\xi,\eta)) \\ \partial_{(\xi,\eta)}y(t; (\xi,\eta)) \end{pmatrix}  \\
		& \qquad  - D_{yy} g(0,p(t;(\xi,\eta)) \left( q^1(t;(\xi,\eta)) - \partial_{(\xi,\eta)}y(t; (\xi,\eta))  \right)^2.
\end{aligned}
\]

By the variation of constants formula we now obtain
\begin{align*}
w_2(t;(\xi,\eta))  &=  \int_t^{\infty} Z(t,s; (\xi,\eta)) \Bigg( D_y g(0,y(s;(\xi,\eta)) - q(s;(\xi,\eta))) \partial^2_{(\xi,\eta)}y(s; (\xi,\eta))  \\
&\qquad - D g(x(s;(\xi,\eta)),y(s;(\xi,\eta))) \begin{pmatrix} \partial^2_{(\xi,\eta)}x(s; (\xi,\eta)) \\ \partial^2_{(\xi,\eta)}y(s; (\xi,\eta)) \end{pmatrix} + H(s;(\xi,\eta)) \Bigg) \, ds.
\end{align*}

By elementary estimates we may obtain the following Lemma:
\begin{lemma}
It holds for some $C>0$ that
\[
\| H(t; (\xi,\eta)) \|  \leq C e^{-(\mu - 2 N_1) t} \quad \text{for } t \geq 0.
\]
\end{lemma}
To obtain this Lemma one uses the assumptions from \cref{hyp:red-map-l-derivatives}, and then the estimates work similarly to how 
$\|Q^1(\xi,\eta)\|$ was bounded in the proof of \cref{lem:red-map-C1}. The constant $C$ depends on $K, K_1^x, K_1^y, N_1, N_2$ and $\mu$.

We are now ready to prove \cref{lem:red-map-C2}.

\begin{proof}[Proof of \cref{lem:red-map-C2}]
We mimic the approach of the proof of \cref{lem:red-map-C1}. So we define
\[
Q^2 \in C^0(U \times \overline{V}, L(\XSet \times \RSet^n, \RSet^n)),
\]
our candidate second derivative for $Q$, by
\begin{align*}
Q^2(\xi,\eta) &:=  \int_0^{\infty} Z(t,s; (\xi,\eta)) \Bigg( D_y g(0,y(s;(\xi,\eta)) - q(s;(\xi,\eta))) \partial^2_{(\xi,\eta)}y(s; (\xi,\eta))  \\
&\qquad - D g(x(s;(\xi,\eta)),y(s;(\xi,\eta))) \begin{pmatrix} \partial^2_{(\xi,\eta)}x(s; (\xi,\eta)) \\ \partial^2_{(\xi,\eta)}y(s; (\xi,\eta)) \end{pmatrix} + H(s;(\xi,\eta)) \Bigg) \, ds.
\end{align*}
We have that $Q^2$ is well-defined now because for $\xi \in U$ it holds that
\[
Q^2(\xi,\eta) \leq \left( \frac{K^x_2 N_1}{\mu -  3 N_1} + \frac{K N_2 }{\mu - 3N_1} + \frac{K^2 K^y_2 N_1 N_2}{(\mu - K N_1)(\mu - 3 N_1)} + \frac{C}{\mu-3N_1}   \right) |\xi|,
\]
and again continuity of $Q^2$ can be proven in a manner comparable to the proof of \cref{lem:red-map-well-defined}. Remark that $3 N_1 < \mu$ by the second item of \cref{hyp:red-map-l-derivatives}.

Next for $q^2(t;\xi,\eta) := Q^2(x(t;(\xi,\eta)),y(t;(\xi,\eta)))$ set 
\[
p^2(t;(\xi,\eta)) := - q^2(t;(\xi,\eta)), \quad P^2(\xi,\eta) := p^2(0;(\xi,\eta)).
\]
It remains to show that $P^2$ is indeed the second derivative of $P$. From the proof
of \cref{lem:red-map-C1} we may use that $DP$ exists and is continuous,
and that $p^1(t;(\xi,\eta)) := DP(x(t;(\xi,\eta)),y(t;(\xi,\eta)))$ satisfies the ODE
\[
\partial_t \, p^1(t;(\xi,\eta)) = D_y g(0,p(t;(\xi,\eta))) p^1(t;(\xi,\eta)).
\]
Now we define $t \geq 0$, $(\xi,\eta) \in U \times \overline{V}$ and $\Delta \in \XSet \times \RSet^n$:
\begin{align*}
a_2(t; (\xi,\eta), \Delta) :&= p^1(t; (\xi,\eta) + \Delta) - p^1(t; (\xi,\eta)) - p^2(t; (\xi,\eta)) \, \Delta.
\end{align*}
The quantity $a_2(t; (\eta,\xi), \Delta)$ has to satisfy the ODE
\begin{equation}
\begin{aligned}\label{eq:a2-ODE}
\partial_t  \, a_2(t; (\xi,\eta), \Delta) & = D_y   g(0,p(t;(\xi,\eta)+\Delta)) p^1(t; (\xi,\eta) + \Delta) \\
&\qquad \quad - D_y g(0,p(t;(\xi,\eta))) p^1(t; (\xi,\eta)) \\
&\qquad \quad  - D_{yy} g(0,p(t;(\xi,\eta))) p^1(t;(\xi,\eta)) \Delta \, p^1(t;(\xi,\eta))  \\
&\qquad \quad  -    D_y g(0,p(t;(\xi,\eta)))  p^2(t;(\xi,\eta)) \Delta.
\end{aligned}
\end{equation}
By setting
\begin{align*}
\hat{g}_2(t; (\xi,\eta), \Delta) :&= D_y g(0,p(t;(\xi,\eta)+\Delta)) - D_y g(0,p(t;(\xi,\eta)))   \\
&\qquad \qquad \qquad \qquad - D_{yy} g(0,p(t;(\xi,\eta)))  p^1(t;(\xi,\eta))  \Delta,
\end{align*}
we can rewrite the ODE for $a(t; (\xi,\eta), \Delta)$ to
\begin{align*}
&\partial_t \, a_2(t; (\xi,\eta), \Delta) =  D_y g(0,p(t;(\xi,\eta))) a(t; (\eta,\xi), \Delta) + \hat{g}_2(t; (\eta,\xi), \Delta) p^1(t; (\xi,\eta) + \Delta)  \\
&\qquad \qquad + D_{yy} g(0,p(t;(\xi,\eta))) \left( p^1(t; (\xi,\eta) + \Delta)  - p^1(t;(\xi,\eta)) \right) \Delta \, p^1(t;(\xi,\eta)).
\end{align*}
As $D_{yy} g(0,p(t;(\xi,\eta))) p^1(t;(\xi,\eta)) = \partial_{(\xi,\eta)} D_y g(0,p(t;(\xi,\eta)))$ we have for $\gamma > 2N_1$:
\[
\sup_{t \geq 0} e^{-\gamma t} \| \hat{g}_2(t; (\xi,\eta), \Delta) p^1(t; (\xi,\eta) + \Delta)  \| = o(|\Delta|),
\]
and as by assumption $\sup \{\|D^2 g(x,y)\| : (x,y) \in U \times \overline{V} \} \leq N_2$ we also have
\[
\sup_{t \geq 0} e^{-\gamma t}  \| D_{yy} g(0,p(t;(\xi,\eta))) \left( p^1(t; (\xi,\eta) + \Delta)  - p^1(t;(\xi,\eta)) \right) \Delta \,  p^1(t;(\xi,\eta)) \| = o(|\Delta|).
\]

The relevant solutions to the ODE \cref{eq:a2-ODE} satisfy the integral equation
\begin{align*}
 &a_2(t; (\xi,\eta), \Delta) =   \int_t^{\infty} Z(t,s; (\xi,\eta)) \, \bigg( \hat{g}(s; (\xi,\eta), \Delta) p^1(t; (\xi,\eta) + \Delta)  \\
 &\qquad \qquad + D_{yy} g(0,p(t;(\xi,\eta))) \left( p^1(t; (\xi,\eta) + \Delta)  - p^1(t;(\xi,\eta)) \right) \Delta \, p^1(t;(\xi,\eta))  \bigg) \, ds,
\end{align*}
and therefore
\[
\sup_{t \geq 0} e^{-\gamma t} \|  a_2(t; (\xi,\eta), \Delta)  \| = o(|\Delta|).
\]
As $a_2(0; (\xi,\eta), \Delta) = DP((\xi,\eta)+\Delta) - DP(\xi,\eta) - P^2(\xi,\eta) \Delta$ we in particular have
\[
\left\| DP((\xi,\eta)+\Delta) - DP(\xi,\eta) - P^2(\xi,\eta) \Delta \right\| =   o(|\Delta|),
\]
and thereby $P^2 = D^2 P$, as was desired. 
\end{proof}


 \section{Application to local theory for critical manifolds of fast-slow systems of ODEs on Banach space}\label{sec:application}
 
In this Section we discuss how the main theorems from the introduction,
\cref{thm:slow-mfd-intro,thm:red-map-intro}, lead to a local theory for
compact attracting submanifolds of critical manifolds of fast-slow systems of ODEs
where the fast system lives on a Banach space. Our starting point is a fast-slow
system like system \cref{eq:fs-system}, so of the form
 \begin{align}
 \begin{split}\label{eq:fs-system-app}
 \dot{x}(t) &= F(x(t),y(t),\eps), \\
 \dot{y}(t) &= \eps \cdot g(x(t),y(t), \eps),
 \end{split}
 \end{align}
where $x \in \XSet$, $y \in \RSet^n$, $\eps \geq 0$, and $F \in C^{k}(\XSet \times \RSet^n \times \RSet_{\geq 0},\XSet)$, $g \in C^k(\XSet \times \RSet^n \times \RSet_{\geq 0},\RSet^n)$ for some $k \geq 2$. The local theory we derive in this Section can be compared to that from Fenichel's 1979 paper \cite{Fe79, Ku15}, with the generalization that we allow the fast variable $x$ to possibly live on an infinite-dimensional space. 

Let $C_0$ denote the critical manifold of system \cref{eq:fs-system-app}, so
\[
C_0 = \{(x,y) \in \XSet \times \RSet^n : F(x,y,0) = 0 \}.
\]
Suppose that $K_0$ is a compact $\mu$-attracting submanifold of $C_0$. This means that  for some compact $K^y \subset \RSet^n$ and some  $h_0 \in C^k(K^y, \XSet)$ we have
\[
K_0 = \{ (h_0(y),y) \in \XSet \times K^y  \} \subseteq C_0,
\]
such that for some $\mu > 0$ it holds that
\[
 \sup \{ \real \lambda : y \in K^y, \lambda  \in \sigma(D_x F(h_0(y),y,0)) \} < -\mu. 
\]
This last condition is also called the spectral gap assumption.

By applying a coordinate transformation as well as suitable cut-off functions (see
\cite{VaIo92} for a detailed explanation on the construction of cut-offs in
Banach space setting) we can find a system that satisfies the assumptions of
\cref{thm:slow-mfd-intro,thm:red-map-intro} and which is identical (after
coordinate transformation) to system \cref{eq:fs-system-app} on some tubular
neighborhood of $K_0 \times \{ 0 \}$ in $\XSet \times \RSet^n \times \RSet_{\geq 0}$.
Therefore \cref{thm:slow-mfd-intro,thm:red-map-intro} can be
applied locally to this tubular neighborhood. In order to translate the spectral gap
assumption to assumption (H1) from \cref{sec:slow-mfds-existence}, one needs to apply
\cref{cor:slow-A-stable} from \cref{app:B}, see also \cref{rem:spectral-gap-to-H1}.
This leads to a Theorem on persistence of $K_0$ for sufficiently small $\eps
>0$: 
 
 \begin{theorem}[Local existence of smooth compact slow manifolds]\label{thm:slow-mfd-app}
Suppose that $K_0$ is a compact $\mu$-attracting submanifold of the critical manifold $C_0$ of system \cref{eq:fs-system-app}. Then there exist $\eps_0 > 0$ and a function $h \in C^k(K^y \times [0,\eps_0], \XSet)$  such that:
\begin{enumerate}
\item For all $\eps \in (0,\eps_0]$, the submanifold with boundary
\[
K_{\eps} := \{ (h(y,\eps),y) \in \XSet \times K^y \}
\]
is locally invariant under the flow of \cref{eq:fs-system-app};
\item It holds that $h(\cdot, 0) = h_0$ and $\sup_{y \in K^y} \| h(y,\eps) - h(y,0) \|_{\XSet} \rightarrow 0$ as $\eps \downarrow 0$.
\end{enumerate}
\end{theorem}

We also denote $h(\cdot,\eps) =: h_{\eps}$, and we observe that the flow on the slow manifolds $K_{\eps}$ is given by
\begin{align*}
 \begin{split} 
 \dot{x}(t) &= \eps \cdot D_y h_{\eps}(y(t)) \, g(h_{\eps}(y(t)),y(t), \eps), \\
 \dot{y}(t) &= \eps \cdot g(h_{\eps}(y(t)),y(t), \eps),
 \end{split}
 \end{align*}
 which in the slow timescale $\tau$ obtained via the transformation of time $\tau = \eps t$ gives 
 \begin{align*}
 \begin{split} 
 \dot{x}(\tau) &= D_y h_{\eps}(y(\tau)) \, g(h_{\eps}(y(\tau)),y(\tau), \eps), \\
 \dot{y}(\tau) &= g(h_{\eps}(y(\tau)),y(\tau), \eps).
 \end{split}
 \end{align*}
 From this last system we see that the flow on $K_{\eps}$ is a perturbation of the flow on $K_0$ that is defined by the differential-algebraic equation
 \begin{equation}\label{fs-system-bdd-vf-reduced}
 \begin{aligned}
 x(\tau) &= h_0(y(\tau)), \\
 \dot{y}(\tau) &= g(x(\tau),y(\tau), 0),
 \end{aligned}
 \end{equation}
 which is often referred to as the reduced system in theory for fast-slow systems. This corresponds to results by Fenichel about the existence of slow manifolds and the flow on them. 

The second result concerns the existence of stable invariant foliations near
$K_{\eps}$ for sufficiently small $\eps > 0$, which tells us that orbits that start
near $K_{\eps}$ are locally exponentially attracted to some orbit that lies on
$K_{\eps}$. We retrieve this result from \cref{thm:red-map-intro} by applying
the Implicit and Inverse Function Theorems to the reduction map described there.

\begin{theorem}[Local existence of invariant foliations near compact slow manifolds]\label{thm:red-map-app} 
Suppose that $K_0$ is a compact $\mu$-attracting submanifold of the critical manifold
$C_0$ of system \cref{eq:fs-system-app}, and let $h_{\eps}$ be as in \cref{thm:slow-mfd-app}. 
Then there exist $0 < \eps_1 \leq \eps_0$, an open set $U \subset \XSet$ 
such that $0 \in U$, and a map $\Phi \in C^{k}(U \times K^y  \times [0,\eps_1], \XSet \times \RSet^n)$ such that
\begin{enumerate}
\item(Invariance of leaves) For each $(x_0,y_0, \eps) \in U \times K^y \times (0,\eps_1]$ we have
\[
(x,y)(t; \Phi(x_0,y_0,\eps),\eps) \in \Phi (U, y(t; (h_{\eps}(y_0),y_0),\eps), \eps)
\]
for all $t \geq 0$ such that $(x,y)(t; (h_{\eps}(y_0),y_0),\eps) \in K_{\eps}$;
\item(Contractivity of leaves) There exists $C \geq 1$ such that for every $(\hat{x}_0, \eps) \in (U \times K^y)  \times (0,\eps_0]$ we have
\begin{align*}
& \| (x,y)(t; \Phi(x_0, y_0, \eps), \eps) -  (x,y)(t; (h_{\eps}(y_0),y_0),\eps)\|_{\XSet \times \RSet^n} \\
&\qquad \qquad \qquad \qquad \leq C e^{-\mu t / 2} \| \Phi(x_0,y_0,\eps) - (h_{\eps}(y_0),y_0)  \|_{\XSet \times \RSet^n} 
\end{align*}
for all $t \geq 0$ such that $(x,y)(t; (h_{\eps}(y_0),y_0),\eps) \in K_{\eps}$.
\end{enumerate}
\end{theorem}

A  $C^k$ reduction map $P_{\eps}$ like in \cref{thm:red-map-intro} can for each $\eps \in [0,\eps_1]$ be defined on $\Phi(U \times K^y \times \{ \eps\})$. The connection between $P_{\eps}$ and $\Phi$ is given by
\[
\Phi(U \times \{ y_0 \} \times \{\eps\}) = P^{-1}_{\eps} (h_{\eps}(y_0),y_0)
\]
for each $(h_{\eps}(y_0),y_0) \in K_{\eps}$. 

 Now set $Q_{\eps}(\xi,\eta) := (\xi,\eta) - P_{\eps} (\xi,\eta)$ for $(\xi,\eta) \in \Phi(U \times K^y \times \{ \eps\})$. Then we obtain
\begin{align*}
(x,y)(t; (\xi,\eta),\eps) =  (x,y)(t; P_{\eps}(\xi,\eta),\eps) + Q_{\eps}((x,y)(t; (\xi,\eta),\eps)),
\end{align*}
as long as $(x,y)(t; (h_{\eps}(y_0),y_0),\eps) \in K_{\eps}$. This can be interpreted in the terminology of matched asymptotic expansions (see for example \cite{Ve05}). Namely, the terms $(x,y)(t; (h_{\eps}(y_0),y_0),\eps)$ for $y_0 \in K^y$ are outer solutions, whilst picking $(h_{\eps}(y_0),y_0) = P_{\eps}(\xi,\eta)$ and adding the correction $Q_{\eps}((x,y)(t; (\xi,\eta),\eps))$ accounts for matched interior layer solutions for the initial value problem defined by $(x,y)(0) = (\xi,\eta)$. It moreover holds that
\[
 \| Q_{\eps}((x,y)(t; (\xi,\eta),\eps)) \|_{\XSet \times \RSet^n} \leq  C e^{-\mu t / 2} \| (\xi,\eta) - P_{\eps}(\xi,\eta) \|_{\XSet \times \RSet^n},
\]
so the interior layer-related correction is exponentially decaying. This way,
\cref{thm:slow-mfd-app,thm:red-map-app} provide a justification for the
method of matched asymptotic expansions in the context of spatially extended systems
with slowly evolving parameters. Further details on this are similar to the
finite-dimensional setting and can be found for example in \cite{Sa90, DJFu96, Ve05}.

  
  \section{Finite-dimensional center-unstable manifolds}\label{sec:cu-manifolds}
 Whilst the Theorems from the Introduction concern attracting critical manifolds, the Theorems and proofs in this paper can be generalized to apply to normally hyperbolic critical manifolds with finite-dimensional center-unstable manifolds. Finite-dimensionality of unstable directions occurs in many physical applications, so this is often a relevant restriction. In these generalizations, the slow manifolds from the Theorems from the Introduction are replaced by center-unstable manifolds. Note that as these center-unstable manifolds are finite-dimensional, the dynamics on the center-unstable manifolds can be further characterized by applying the theories of Fenichel. See also Theorem 9.1 in \cite{Fe79} for the finite-dimensional analog of center-unstable manifold existence. In specialized literature on invariant manifold theorems for differential equations, the extension of proofs for Center Manifold Theorems to existence of center-unstable manifolds is a common feature, see for example \cite{Ke67, HiPuSh77, Sh-ea98, HoKn24}.
 
 The details for the extension of the material of \cref{sec:slow-mfds-existence} to cover existence of center-unstable manifolds for differential equations on Banach space with slowly evolving parameters are left for future work. One needs a generalization of the set-up at the start of \cref{sec:slow-mfds-existence} which accounts for the presence of unstable directions, and this generalized system will involve additional hypotheses. However, results from \cref{lem:sol-slow-subsystem} onwards as well as the construction of a contraction mapping, whose fixed point is a center-unstable manifold, are expected to be largely similar, albeit with some minor modifications. 
 
 In applications (we use the set-up of \cref{sec:application} for \cref{eq:fs-system-app}), normal hyperbolicity for compact critical manifolds $K_0$ can be characterized as follows. We define stable and unstable parts of the spectrum of the linearized fast subsystem along $K_0$ as follows:
 \begin{align*}
 \Sigma^s(y) &:= \left\{ \lambda \in \sigma(D_x F(h_0(y),y)) : \real \lambda < 0 \right\}, \\
 \Sigma^u(y) &:= \left\{ \lambda \in \sigma(D_x F(h_0(y),y)) : \real \lambda > 0 \right\},
 \end{align*}
 where $y \in K^y$. Now we call $K_0$ normally hyperbolic (with spectral gap constants $\mu,\nu$)  if there exist $\mu, \nu > 0$ such that
 \[
 \sup \left\{\real \lambda : y \in K^y, \lambda \in \Sigma^s(y) \right\} < -\mu < 0 < \nu <   \inf \left\{ \real \lambda : y \in K^y, \lambda \in \Sigma^u(y) \right\}.
 \]
 Observe that under the assumption that $\sigma^u(y)$ is finite-dimensional, the infimum on the right-hand side can be replaced by a minimum. 
 
 For each $(h_0(y),y) \in K_0$ the unstable eigenspace is defined via the projection
 \[
 \Pi^{u}(y) := \frac{1}{2\pi i} \int_{\Gamma_y} \left(\lambda I_{\XSet} - D_x F(h_0(y),y)  \right)^{-1} \, d\lambda,
 \]
 where $\Gamma_y$ is a simple, counterclockwise oriented Jordan curve that surrounds $\Sigma^u(y)$ in the complex plane and that lies entirely to the right of the line $\real \lambda = 0$. Now the set
 \[
E^{cu} := \{ ((h_0(y),y), x) \in K_0 \times \XSet : x \in \range{\Pi^u(y)} \}
 \]
 is a finite-dimensional vector bundle with base space $K_0$ (see for example Chapter 2 in \cite{El12}). 
 
Now a claim analogous to \cref{thm:slow-mfd-app}, and of which the proof should go mostly along lines similar to \crefrange{sec:slow-mfds-existence}{sec:k-smooth-slow-mfds}, is that there exists locally (near $K_0$) a center-unstable manifold $W^{cu}_{\text{loc},0}$ which is tangent to $E^{cu}$ along its base space $K_0$. Moreover, the manifold $W^{cu}_{\text{loc},0}$ is perturbed regularly for sufficiently small $\eps > 0$. This means that there exists $\eps_0 > 0$ such that
\begin{enumerate}
\item There exists a $C^k$ function $h^{cu}: E^{cu}_{\text{loc}} \times [0,\eps_0] \rightarrow \XSet \times \RSet^n$, which paramaterizes an invariant manifold
\[
W_{\text{loc},\eps}^{cu} := \left \{ h^{cu}(\hat{x}, \eps) : \hat{x} \in E^{cu}_{\text{loc}}      \right \}.
\]
The center-unstable manifold $W_{\text{loc},\eps}^{cu}$ is locally invariant under the flow of \cref{eq:fs-system-app}, and $W_{\text{loc},0}^{cu}$ is tangent to $E^{cu}$ along its base space $K_0$;
\item It holds that 
\[
\sup_{\hat{x} \in E^{cu}_{\text{loc}}}  \| h^{cu}(\hat{x},\eps) - h^{cu}(\hat{x},0) \|_{\XSet \times \RSet^n} \rightarrow 0 \quad \text{as } \eps \downarrow 0.
\]
\end{enumerate}

Furthermore, a local $C^0$ reduction map $P_{\eps}$ on $\XSet \times \RSet^n$ with $W_{\text{loc},\eps}^{cu}$ as stem exists (for sufficiently small $\eps \geq 0$) analogously to \cref{thm:red-map-app} (but without further smoothness claim). All points which are mapped under $P_{\eps}$ to the slow manifold within $W_{\text{loc},\eps}^{cu}$ form a $C^k$ foliated center-stable manifold; proofs for this are mostly analogous to the material from \crefrange{sec:red-map-existence}{sec:red-map-smoothness}. This way the classical results from \cite{Fe79} can be reproduced in the setting that the fast system lives on a Banach space via the functional analytic proof methods from this paper. Future research should flesh out the remaining proof details.

The extension of the theory in this paper to (finite-dimensional) center-unstable manifolds is relevant to problems prevalent in literature on fast-slow systems, whereby one
attempts to construct non-local orbit segments passing near non-hyperbolic points on a critical manifold (see for instance \cite{SzWe04, We12}). 
This typically involves the connection of orbit segments near normally hyperbolic critical manifolds with local dynamics near non-hyperbolic points. In the ODE literature for fast-slow systems, 
this `connection problem' is handled e.g. via the utilization of geometric blow-up, such as in \cite{KrSz01a,KrSz01c}. 
These type of research problems are for instance prominent in literature on canard orbits/cycles \cite{KrSz01b,We12, Ku15}. The theory in this paper is intended to aid research into extending
such global orbit constructions (rigorously) towards spatially extended systems of fast-slow differential equations. Center manifold theorems near non-hyperbolic points of infinite-dimensional
fast-slow systems were discussed in \cite{AvDe20}, and this paper and \cite{AvDe20} can be considered complementary to eachother. For numerical examples of orbit segments
in spatially extended fast-slow systems that include passages near normally hyperbolic invariant manifolds, as well as canard segments and delayed bifurcations, 
we also refer to \cite{AvDe17,KaVo18,AvDe20}. 


\section*{Acknowledgements}
Dirk Doorakkers and Daniele Avitabile acknowledge funding from NWO grant 613.009.136: Spatio-temporal canards and delayed bifurcations.

 
\bibliographystyle{siamplain}
\bibliography{references}


\newpage
\appendix


\section \nopunct

 The next two Lemmas lead to a Corollary that we use to prove smoothness of invariant manifolds in \cref{sec:slow-mfds-smoothness}.
 
 \begin{lemma}\label{lem:uniformdiff}
 Let $\XSet$ be a Banach space and assume $F \in C^1(\RSet \times (\XSet \times \RSet^{n}), \XSet \times \RSet^n)$ as well as 
 $\theta \in C^0(\RSet \times \RSet^n,\XSet \times \RSet^{n})$. If the arguments of $F$ are given by $(t,x) \in \RSet \times (\XSet \times \RSet^{n})$,
 then we denote by $D_x F$ the first Fr\'echet derivative of $F$ with respect to the second argument. 
 Suppose now that for some $a \in \RSet^n$ it holds that 
  \[
  \lim_{h \rightarrow 0} \sup_{t \in \RSet} \| D_x F(t,\theta(t,a+h)) - D_x F(t,\theta(t,a)) \| = 0.
  \]
  Then it holds that for every $\gamma > 0$ there exists $\delta(\gamma)>0$ such that $|h| < \delta$ implies
  \begin{align*}
  |F(t,\theta(t,a+h)) - F(t,\theta(t,a)) - D_xF(t,\theta(t,a)) &(\theta(t,a+h)-\theta(t,a))| \\
   &\leq \gamma \cdot |\theta(t,a+h)-\theta(t,a)|
  \end{align*}
  for all $t \in \RSet$.
 \end{lemma}
 
  \begin{proof}
  Set for $t \in \RSet$, $a,h \in \RSet^n$ and $r \in [0,1]$:
  \begin{align*}
  p(t,a,h) &:= F(t,\theta(t,a+h)) - F(t,\theta(t,a)) - D_x F(t,\theta(t,a)) (\theta(t,a+h)-\theta(t,a)), \\
  q(t,a,h,r) &:= D_x F(t,\theta(t,a) + r (\theta(t,a+h)-\theta(t,a))) - D_x F(t,\theta(t,a)).
  \end{align*}
  By the Mean Value Theorem we derive for every $t \in \RSet$ that
  \begin{align*}
  |p(t,a,h)| &\leq \int_0^1\sup_{t \in \RSet}  \|q(t,a,h,r) \| \, dr \cdot |\theta(t,a+h)-\theta(t,a)|.
  \end{align*}
  By assumption, for every $\gamma > 0$, we can pick $\delta(\gamma) > 0$ such that $|h| < \delta$ implies $\sup_{t \in \RSet} \|q(t,a,h,r) \| < \gamma$ for all $r \in [0,1]$. Then also for $|h| < \delta$ we have for every $t \in \RSet$ that
    \begin{align*}
  |p(t,a,h)| &\leq \gamma \cdot |\theta(t,a+h)-\theta(t,a)|,
  \end{align*}
  which implies the Lemma.
  \end{proof}
  
  \begin{lemma}
  Let $G \in C^0(\RSet \times \RSet^n, B)$ where $B$ is some Banach space and such that $G(t,0)=0$ for all $t \in \RSet$. Suppose that for some open neighborhood $U$ of $0 \in \RSet^n$ we have
   \[
  \lim_{t \rightarrow \pm \infty} \sup_{h \in U} \| G(t,h) \| = 0.
  \]
  Then also
   \[
  \lim_{h \rightarrow 0} \sup_{t \in \RSet} \| G(t,h) \| = 0.
  \]
  \end{lemma}
  
  \begin{proof}
  Pick $\gamma > 0$ arbitrarily. By assumption, there exists $t_\gamma > 0$ such that $|t| > t_\gamma$ implies $\| G(t,h) \| < \gamma $ for all $h \in U$. 
  
  Continuity of $G$ implies that $ \lim_{h \rightarrow 0} \| G(t,h) \| = 0$ for every $t \in [-t_{\gamma},t_{\gamma}]$, so surely also
  \[
  \lim_{h \rightarrow 0} \sup_{t \in [-t_{\gamma},t_{\gamma}]} \| G(t,h) \| = 0.
  \] 
  Therefore there exists $\delta(\gamma)>0$ such that $|h| < \delta$ implies $ \| G(t,h) \| < \gamma$ for all $t \in [-t_{\gamma},t_{\gamma}]$. 
  
  Assume without loss of generality that $\{h \in \RSet^n: |h| < \delta \} \subset U$. Combining results we then obtain that $|h| < \delta$ implies $ \| G(t,h) \| < \gamma$ for all $t \in \RSet$, which proves the Lemma. 
  \end{proof}
  
  \begin{corollary}\label{cor:uniformdiff}
Assume $\XSet$ is a Banach space, $F \in C^1(\RSet \times (\XSet \times \RSet^{n}), \XSet \times \RSet^n)$ and $\theta \in C^0(\RSet \times \RSet^n,\XSet \times \RSet^{n})$. If the arguments of $F$ are given by $(t,x) \in \RSet \times (\XSet \times \RSet^{n})$, then we denote by $D_x F$ the first Fr\'echet derivative of $F$ with respect to the second argument. Now suppose for some $a \in \RSet^n$ and some neighborhood $U$ of $0 \in \RSet^n$ that
  \[
  \lim_{t \rightarrow \pm \infty} \sup_{h \in U} \| D_x f(t,\theta(t,a+h)) - D_x f(t,\theta(t,a)) \| = 0.
  \]
  Then it holds that for every $\gamma > 0$ there exists $\delta(\gamma)>0$ such that $|h| < \delta$ implies
  \begin{align*}
  |f(t,\theta(t,a+h)) - f(t,\theta(t,a)) - D_xf(t,\theta(t,a)) &(\theta(t,a+h)-\theta(t,a))| \\
   &\leq \gamma \cdot |\theta(t,a+h)-\theta(t,a)|
  \end{align*}
  for all $t \in \RSet$.
  \end{corollary}

  
 \section \nopunct \label{app:B}
 
 This Appendix treats theory for parameter-dependent uniformly continuous semigroups, and its relationship to results for the two-parameter semigroups that are generated when said parameters may also evolve slowly in time. The main result of this Appendix is \cref{cor:slow-A-stable}, which gives conditions under which bounded linear operators depending on slowly evolving parameters generate two-parameter semigroups that possess exponential bounds. This is relevant to the application of the main results of this paper, as it allows one to validate assumption (H1) from \cref{sec:slow-mfds-existence} see also \cref{rem:spectral-gap-to-H1} and \cref{sec:application}.
 
 \begin{definition}
 A semigroup $T: \RSet_{\geq 0} \rightarrow L(\XSet, \XSet)$ is a family of bounded linear operators on a Banach space $\XSet$ such that 
 \begin{enumerate}
 \item[(i)] $T(0) = I_{\XSet}$ ;
 \item[(ii)] $T(t+s) = T(t) T(s)$ for all $t,s \geq 0$.
 \end{enumerate}
 Such a semigroup is moreover called uniformly continuous if 
  \begin{enumerate}
 \item[(iii)]  The mapping 
 \[
 \RSet_{\geq 0} \ni t \mapsto T(t) \in L(\XSet,\XSet)
 \]
  is continuous with respect to the operator norm on $L(\XSet,\XSet)$.
 \end{enumerate}

 \end{definition}
 
 For an operator $A \in L(\XSet,\XSet)$, the semigroup generated by $A$ consists of the family of bounded linear operators on $\XSet$ given by
 \[
 e^{At} := \sum_{n=0}^{\infty} \frac{A^n t^n}{n !},
 \]
with $t \geq 0$. Such a semigroup is uniformly continuous and is the fundamental solution (in forward time) to the linear ODE
\begin{equation}\label{eq:linear-ode-banach}
\dot{x}(t) = A \, x(t), \quad x \in \XSet.
\end{equation}
We denote by $\sigma(A)$ the \textit{spectrum} of $A$, and then $\rho(A) := \mathbb{C} \setminus \sigma(A)$ is the \textit{resolvent set} of $A$. For each $\lambda \in \rho(A)$, we denote by $R(\lambda,A)$ the corresponding \textit{resolvent} which is defined as $R(\lambda,A) := (\lambda I_{\XSet} - A)^{-1}$.

By the following Lemma, the uniformly continuous semigroups on $\XSet$ are precisely those semigroups that are generated by bounded linear operators on $\XSet$;

\begin{lemma}
A semigroup $T: \RSet_{\geq 0}  \rightarrow L(\XSet, \XSet)$ is uniformly continuous if and only if there exists $A \in L(\XSet, \XSet)$ such that $T(t) = e^{At}$ for all $t \geq 0$.
\end{lemma}

\begin{proof}
See for example Chapter I, Theorem 3.7 in \cite{EnNa00}.
\end{proof}

For $A \in L(\XSet,\XSet)$, if $\mu > 0$ is chosen such that $\sup \{ \real \lambda : \lambda \in \sigma(A) \} < -\mu$, it is well-known that there exists $K \geq 1$ such that
\[
| e^{At} \xi | \leq K e^{-\mu t} |\xi| \quad \text{for all } \xi \in \XSet, \, t \geq 0.
\]
This means that the condition $\sup \{ \real \lambda : \lambda \in \sigma(A) \} < -\mu$ ensures exponential stability (at rate of exponential decay $\mu$) of the origin for the linear ODE \cref{eq:linear-ode-banach}. 

Now suppose $A \in C^0(\overline{V},L(\XSet,\XSet))$ is dependent on some parameter $\nu \in \overline{V}$. Then for every $\nu \in \overline{V}$ there exists $K_{\nu} \geq 1$, such that
\[
| e^{A(\nu)t} \xi | \leq K_\nu e^{-\mu t} \xi \quad \text{for all } \xi \in \XSet, \, t \geq 0.
\]
In general it is not guaranteed that $\sup_{\nu \in \overline{V}} K_{\nu}$ is finite, however this can be desirable for certain applications, as for example in this article or see \cite{LaPo24}.

The following Lemma provides conditions that guarantee the existence of uniform exponential bounds on families of semigroups generated by parameter-dependent bounded linear operators on $\XSet$. 
 
 \begin{lemma}\label{lem:un-sec-un-bounds}
Let $A \in C^0(\overline{V},L(\XSet,\XSet))$ and assume there exist $\mu >0, \tilde{K} \geq 1$ and $\theta \in (\pi/2,\pi)$ such that for all $\nu \in \overline{V}$
 \[
\Omega_{-\mu, \theta} := \{ \lambda \in \mathbb{C}: \lambda \neq -\mu, |\arg(\lambda + \mu)| < \theta \} \subseteq  \rho(A(\nu)),
 \]
 and
 \[
\| R(\lambda, A(\nu)) \| \leq \frac{\tilde{K}}{|\lambda + \mu|} \quad \text{for all } \lambda \in \Omega_{-\mu, \theta}.
\]
 Then there exists a constant $K \geq 1$, such that for all $\nu \in \overline{V}$ the semigroup generated by $A(\nu) \in L(\XSet,\XSet)$ satisfies the bounds
 \[
 | e^{A(\nu)t} \xi | \leq K e^{-\mu t} |\xi|, \quad \text{for all } \xi \in \XSet, \, t \geq 0.
 \]
 \end{lemma}
 
 \begin{proof}
 The proof is based on checking that a well-known proof for exponential bounds on analytic semigroups provides uniform bounds if the generators $A(\nu) \in L(\XSet,\XSet)$ are uniformly sectorial for all $\nu \in \overline{V}$, see also Lemma 2.3 in recent work by Latushkin \cite{LaPo24}.
 
Firstly, note that for each $\nu \in \overline{V}$ and $t \geq 0$, the semigroup $e^{A(\nu) t}$ satisfies the formula
 \[
 e^{A(\nu) t} = \frac{1}{2\pi i} \int_{\Gamma} R(\lambda, A(\nu)) \, e^{\lambda t} \, d\lambda,
 \]
 where $\Gamma$ is any piecewise smooth curve in  $\Omega_{-\mu, \theta}$ that has a parameterization $\gamma : \RSet \rightarrow \Omega_{-\mu, \theta}$ such that $\arg(\gamma(s) + \mu) \rightarrow \pm \varphi$ as $s \rightarrow \pm \infty$ for some $\varphi \in (\pi/2,\theta)$. 
 
 Now for some $r > 0$ and $\varphi \in (\pi/2,\theta)$, we consider the curve $\Gamma = \Gamma^-_{r,\varphi} \cup \Gamma^c_{r,\varphi} \cup \Gamma^+_{r,\varphi}$,
 where
 \[
 \Gamma^{\pm}_{r,\varphi} := \{ \lambda \in \Omega_{-\mu, \theta}: \arg(\lambda + \mu) = \pm \varphi, |\lambda + \mu| \geq r \} = \{ -\mu + se^{\pm i \varphi}: s \geq r \},
 \]
 and
 \[
  \Gamma^{c}_{r,\varphi} := \{ \lambda \in \Omega_{-\mu, \theta}: | \arg(\lambda + \mu)| \leq \varphi, |\lambda + \mu| = r \} = \{ -\mu + re^{i \phi}: -\varphi \leq \phi \leq \varphi  \}.
 \]
 We then obtain, for $t > 0$, by changing variables according to parameterizations for the pieces $\Gamma^{\pm}_{r,\varphi}$:
 \begin{align*}
 \left\| \int_{ \Gamma^{\pm}_{r,\varphi}} R(\lambda, A(\nu)) \, e^{\lambda t} \, d\lambda \right\| &=   \left\|  \int_r^{\infty} R(-\mu+se^{\pm i \varphi}, A(\nu)) \, e^{-\mu t + s t e^{\pm i \varphi}}  \, e^{\pm i \varphi} \, ds \right\| \\
 &\leq \tilde{K} e^{-\mu t}\int_r^{\infty} \frac{e^{s t \cos \varphi}}{s} \, ds \leq \frac{\tilde {K} e^{rt \cos \varphi}}{- r t \cos \varphi} e^{-\mu t}.
 \end{align*}
 Setting $r = 1 / t$ we get
 \[
  \left\| \int_{ \Gamma^{\pm}_{r,\varphi}} R(\lambda, A(\nu)) \, e^{\lambda t} \, d\lambda \right\| \leq  \frac{\tilde {K} e^{\cos \varphi}}{-\cos \varphi} e^{-\mu t}.
 \]
Similarly, for $t > 0$, by changing variables according to a parameterization for $\Gamma^{c}_{r,\varphi}$ we have
 \begin{align*}
 \left\|  \int_{ \Gamma^{c}_{r,\varphi}} R(\lambda, A(\nu)) \, e^{\lambda t} \, d\lambda  \right\| &=   \left\|   \int_{-\varphi}^{\varphi} R(-\mu+re^{ i \phi}, A(\nu)) \, e^{-\mu t + r t e^{ i \phi}}  \, r e^{ i \phi} \, d\phi \right\| \\
 &\leq \tilde{K} e^{-\mu t}\int_{-\varphi}^{\varphi} e^{r t \cos \phi} \, d\phi \leq \tilde {K} 2 e^{rt} \varphi e^{-\mu t},
 \end{align*}
 and thereby for $r = 1 / t$:
 \[
  \left\|  \int_{\Gamma^{c}_{r,\varphi}} R(\lambda, A(\nu)) \, e^{\lambda t} \, d\lambda  \right\|  \leq \tilde {K} 2 e \varphi e^{-\mu t}.
 \]
 By combining above estimates we now see that for each $t > 0$:
 \begin{align*}
 \left\| e^{A(\nu) t} \right\| =  \left\|  \frac{1}{2\pi i}  \int_{\Gamma} R(\lambda, A(\nu)) \, e^{\lambda t} \, d\lambda  \right\| \leq \frac{\tilde{K}}{\pi} \left(  e \varphi - e^{\cos \varphi} \sec \varphi  \right) e^{-\mu t}.
 \end{align*}
 As all these estimates are independent of the choice of $\nu$, we see that the Lemma is proven with
 \[
K :=  \frac{\tilde{K}}{\pi} \left(  e \varphi - e^{\cos \varphi} \sec \varphi  \right)
 \]
 for arbitrary choice of $\varphi \in (\pi/2, \theta)$.
 \end{proof}
 
 \begin{definition}
 For $A \in C^0(\overline{V},L(\XSet,\XSet))$, if the family of operators given by
 $\{ A(\nu): \nu \in \overline{V} \}$ satisfies the conditions of the above \cref{lem:un-sec-un-bounds}, we shall also call such a family uniformly sectorial, or $(\mu, \theta, \tilde{K})$-sectorial uniformly with respect to $\nu \in \overline{V}$. 
 \end{definition}
 
 \begin{remark}
 As an example of what can go wrong if the set of operators $A(\nu) \in L(\XSet,\XSet)$ from the above Lemma is not uniformly sectorial, consider the planar ODE system
 \begin{equation}
 \begin{aligned}
 \dot{x} = A(\nu) x := \begin{pmatrix} -1 & -1 \\ \nu^2 & -1 \end{pmatrix} x, \quad x \in \RSet^2,
 \end{aligned}
 \end{equation}
 with $\nu \in \RSet$. The eigenvalues of $A(\nu)$ are $\lambda_{\pm} = -1 \pm \nu i$, and its general solution is given by
 \[
 x(t) = e^{-t} \left( C_1 \begin{pmatrix} \cos \nu t \\ \nu \sin \nu t \end{pmatrix} +  C_2 \begin{pmatrix} -\sin \nu t \\ \nu \cos \nu t \end{pmatrix}    \right).
 \]
 Pick $C_1 = 1, \, C_2=0$ to obtain the solution for initial condition $x_0 := x(0) = \begin{pmatrix} 1 & 0 \end{pmatrix}^T$, which for odd $\nu = k \in \mathbb{Z}$ leads to
 \[
 x(\pi/2; x_0) =  \begin{pmatrix} 0 \\ \pm k e^{-\pi/2}  \end{pmatrix}.
 \]
 So we see that, for a diverging sequence of odd $k$, as $|k| \rightarrow \pm \infty$ we also have that $| x(\pi/2; x_0) | \rightarrow \infty$. Therefore the solutions do not have an exponential bound that holds for all $\nu \in \RSet$.
 \end{remark}
 
  \begin{remark}
 As a second example of what can go wrong if the set of operators $A(\nu) \in L(\XSet,\XSet)$ from the above Lemma is not uniformly sectorial, consider the planar ODE system
 \begin{equation}
 \begin{aligned}
 \dot{x} = A(\nu) x := \begin{pmatrix} -1 & \nu \\ 0 & -1 \end{pmatrix} x, \quad x \in \RSet^2,
 \end{aligned}
 \end{equation}
 with $\nu \in \RSet$. The matrix $A(\nu)$ has a single eigenvalue $\lambda = -1$ with geometric multiplicity one, and its general solution is given by
 \[
x(t) = e^{-t} \begin{pmatrix} 1 & \nu t \\ 0 & 1 \end{pmatrix} x_0, \quad x_0 = x(0) \in \RSet^2.
 \]
Let now $x_0 = \begin{pmatrix} 0 & 1 \end{pmatrix}^T$, so that $x(1; x(0)) = e^{-1}  \begin{pmatrix} \nu & 1 \end{pmatrix}^T$. We see that $|x(1; x_0)| \rightarrow \infty$ as $|\nu| \rightarrow \infty$, so there cannot exist an exponential bound that is uniform for all $\nu \in \RSet$.

Observe that 
\[
A(\nu)^{-1} = \begin{pmatrix} -1 & -\nu \\ 0 & -1 \end{pmatrix} ,
\]
so that for each $a > -1$ we have
\[
\| (0 - A(\nu))^{-1} \| \geq \frac{ |a \nu|} {|0 - a|},
\]
which shows that $A(\nu)$ cannot be uniformly sectorial with respect to $\nu \in \RSet$. 
 \end{remark}
 
 Next we turn our attention to settings in which the parameter $\nu \in \overline{V}$
 evolves in time. As discussed in \cref{sec:slow-mfds-existence}, if $A \in
 C^0(\overline{V},L(\XSet,\XSet))$ and $\tilde{\psi} \in C^0(\RSet, \overline{V})$,
 then we can associate a two-parameter semigroup or process $\{ T(t,s ;
 \tilde{\psi}): t \geq s \}$ to the time-dependent family of operators
 $\{A(\tilde{\psi}(t)): t \in \RSet \}$, which describes the solutions to the
 non-autonomous linear ODE
 \begin{equation}\label{linear-ode-banach-na}
 \dot{x}(t) = A(\tilde{\psi}(t)) \, x(t).
 \end{equation}
 In general, even if the family $\{ A(\nu) : \nu \in \overline{V} \}$ is uniformly sectorial, exponential stability of the origin of this ODE is not guaranteed. The following Lemma leads to a Corollary which intuitively states that if $\tilde{\psi}$ is changing sufficiently slowly, and $\XSet \times \overline{V} \ni (x,\nu) \mapsto A(\nu) x$ is sufficiently close to a function that is uniformly Lipschitz in $\nu$, then the process $\{ T(t,s,\tilde{\psi}): t \geq s \}$ has an exponential bound which implies exponential stability of the origin of the ODE \cref{linear-ode-banach-na} in case the family $\{ A(\nu) : \nu \in \overline{V} \}$ is uniformly sectorial.
 
  \begin{lemma}\label{lem:un-bounds-stable-process}
For $A \in C^0(\overline{V},L(\XSet,\XSet))$, consider the family of uniformly continuous semigroups $\{ e^{A(\nu)t} : \nu \in \overline{V} \}$, and suppose there exist $\mu > 0$ and $K \geq 1$ such that for every $\nu \in \overline{V}$
 \[
 | e^{A(\nu)t} \xi | \leq K e^{-\mu t} |\xi|, \quad \text{for all } \xi \in \XSet,  \, t \geq 0.
 \]
 
For arbitrary $0 < \eps < \mu$, choose $l > 0$ such that $K e^{-\mu l }  \leq e^{-(\mu - \eps) l }$. If for $\tilde{\psi} \in C^0(\RSet, \overline{V})$ it holds that
\[
\left\|  A(\tilde{\psi}(t)) - A(\tilde{\psi}(s)) \right\|   \leq  \eps / K \quad \text{as long as } |t-s| \leq l,
\] 
 then for the process $\{ T(t,s ; \tilde{\psi}): t \geq s \}$ associated to $A(\tilde{\psi}(t))$ we have
 \[
 | T(t,s ; \tilde{\psi}) \xi | \leq K e^{-(\mu - \eps) (t-s) } |\xi|, \quad \text{for all } \xi \in \XSet, t \geq s.
 \]
 \end{lemma}
 
 \begin{proof}
Let $A_s := A(\tilde{\psi}(s))$, then by the variation of constants formula a solution of
\[
\dot{x}(t) = A(\tilde{\psi(t)}) x(t) = A_s x(t) + (A(\tilde{\psi(t)}) - A_s) x(t)
\]
with $x(s) = \xi$, satisfies the integral equation
\begin{align*}
x(t) &= e^{A_s (t-s)} \xi + e^{A_s t} \int_s^t e^{-A_s r} ( A(\tilde{\psi}(r)) - A_s ) x(r) \, dr.
\end{align*}
Therefore, for $t \geq s$:
\begin{align*}
| x(t) | &\leq K e^{-\mu (t-s)} |\xi| + K e^{-\mu t} \int_s^t e^{-\mu r} \| A(\tilde{\psi}(r)) - A_s \| |  x(r) | \, dr.
\end{align*}
Then by Gr\"onwall's inequality we obtain for $t \in [s,s+l]$:
\begin{align*}
| x(t) | &\leq K e^{-\mu (t-s)} |\xi| e^{K \int_s^t \| A(\tilde{\psi}(r)) - A_s \|  \, dr} \\
&\leq K e^{-(\mu - \eps) (t-s) } |\xi|.
\end{align*}
Now for $t \in [s,\infty)$, we can write $t = s + n l + \tau$ for some $n \in \mathbb{N}_{\geq 0}$ and $\tau \in [0,l)$; and then
\begin{align*}
| T(t,s; \tilde{\psi}) \xi | &= \left| T(t, s+ nl; \tilde{\psi}) \prod_{i = 1}^{n} T(s+il, s+(i-1)l; \tilde{\psi}) \xi \right| \\
&\leq K e^{-(\mu - \eps) (t-(s+nl)) } e^{-(\mu - \eps) nl }  |\xi| = K e^{-(\mu - \eps) (t-s)}  |\xi|.
\end{align*}
\end{proof}

\begin{corollary}\label{cor:slow-A-stable}
For $A \in C^0(\overline{V},L(\XSet,\XSet))$, let $A(\nu)$ with $\nu \in \overline{V}$ be a family of bounded linear operators on $\XSet$ 
that is $(-\tilde{\mu},\theta,\tilde{K})$-sectorial uniformly with respect to $\nu \in \overline{V}$, 
for some triplet of constants $\tilde{\mu} > 0, \, \theta \in (\pi/2,\pi)$ and $\tilde{K} \geq 1$.

Suppose the function $\overline{V} \rightarrow L(\XSet,\XSet) : \nu \mapsto A(\nu)$ is uniformly continuous in the operator norm.
Then for every $0 < \mu < \tilde{\mu}$ there exist $N_0 > 0$ and $K \geq 1$ such that if
\[
\tilde{\psi} \in \{\psi \in C^1(\RSet,\overline{V}): | \dot{\psi}(t) | \leq N_0 \quad \forall \, t \in \RSet \},
\]
then for the process $\{ T(t,s; \tilde{\psi}): t \geq s \}$ associated to $A(\tilde{\psi}(\blank))$ it holds that
 \[
 | T(t,s; \tilde{\psi}) \xi | \leq K e^{-\mu (t-s) } |\xi|, \quad \text{for all } \xi \in \XSet, t \geq s.
 \]
\end{corollary} 

\begin{proof}
By \cref{lem:un-sec-un-bounds}, we have that for the family of uniformly continuous semigroups $\{ e^{A(\nu)t} : \nu \in \overline{V} \}$ there exists $K \geq 1$ such that for every $\nu \in \overline{V}$
 \[
 | e^{A(\nu)t} \xi | \leq K e^{-\mu t} |\xi|, \quad \text{for all } \xi \in \XSet,  \, t \geq 0.
 \]
 For arbitrary $0 < \mu < \tilde{\mu}$, pick $l$ such that $Ke^{-\tilde{\mu} l} \leq e^{-\mu l}$, and next choose $ N_0 > 0$ sufficiently small such that
 $|\nu_2 - \nu_1| < N_0 l$ implies $\| A(\nu_2) - A(\nu_1) \| < (\tilde{\mu}-\mu) / K$ (this is possible because of the assumed uniform continuity on
 $\nu \mapsto A(\nu)$ ).
An application of \cref{lem:un-bounds-stable-process} now gives the desired result.
\end{proof} 
The above \cref{cor:slow-A-stable} is a more general case of a theorem that is analogous to theorems in texts by Coppel \cite{Co78} and Henry \cite{He81}:

\begin{corollary}\label{cor:slow-A-Henry}
Let $\mathcal{A} := \{ A(\nu) \in L(\XSet,\XSet): \nu \in \overline{V}\}$ be a pre-compact subset of $L(\XSet,\XSet)$, such that for some $\mu > 0$ it holds that
\[
\sup \{ \real \lambda: \nu \in \overline{V}, \lambda \in \sigma(A(\nu)) \} < -\mu.
\]
Assume also that $A$ depends uniformly Lipschitz on $\nu \in \overline{V}$, i.e. there exists $M_1^{\nu}$ such that for all $\nu_1, \nu_2 \in \overline{V}$ it holds that
\[
\| A(\nu_2) - A(\nu_1) \| \leq M_1^{\nu} |\nu_2 - \nu_1|.
\]
Then there exist $N_0 > 0$ and $K \geq 1$ such that for every
\[
\tilde{\psi} \in \{\psi \in C^1(\RSet,\overline{V}): | \dot{\psi}(t) | \leq N_0 \quad \forall \, t \in \RSet \},
\]
the process $\{ T(t,s ; \tilde{\psi}): t \geq s \}$ associated to $A(\tilde{\psi}(t))$ satisfies the property
 \[
 | T(t,s ; \tilde{\psi}) \xi | \leq K e^{-\mu (t-s) } |\xi|, \quad \text{for all } \xi \in \XSet, t \geq s.
 \]
\end{corollary}

\begin{proof}
As shown for example in Chapter III Lemma 6.2 from Daleckii and Krein \cite{DaKr74}, as $\mathcal{A}$ is pre-compact there exist $\eps>0$ and $K \geq 1$ such that it holds for all $\nu \in \overline{V}$ that
\[
 | e^{A(\nu)t} \xi | \leq K e^{-(\mu+\eps) t} |\xi|, \quad \text{for all } \xi \in \XSet,  \, t \geq 0.
\]
Now pick $l$ such that $Ke^{-(\mu+\eps) l} \leq e^{-\mu l}$, and next choose $N_0 > 0$ sufficiently small such that $K M_1^{\nu} N_0 l  \leq \eps$. Then if
\[
\tilde{\psi} \in \{\psi \in C^1(\RSet,\overline{V}): | \dot{\psi}(t) | \leq N_0 \quad \forall \, t \in \RSet \},
\] 
we have for $t \geq s$ with $|t-s| \leq l$ that
 \begin{align*}
\left\| A(\tilde{\psi}(t)) - A(\tilde{\psi}(s)) \right\| \leq   M_1^{\nu} |\tilde{\psi}(t) - \tilde{\psi}(s)| \leq M_1^{\nu} N_0 l \leq \eps / K.
\end{align*}
An application of \cref{lem:un-bounds-stable-process} lastly gives the desired result.
\end{proof}


\section \nopunct  \label{app:C}

This appendix proves that the space $\mathcal{E}^{0}$ introduced in \cref{sec:red-map-existence}, is a Banach space when equipped with the $\mathcal{E}$-norm. It is partly similar to a proof in Chapter 4 of \cite{Ch06}.

\begin{lemma}
The space
 \[
 \mathcal{E}^0 := \left\{ Q \in C^{0}(\XSet \times \overline{V}, \overline{V}): Q(0, \cdot) =0;  \sup_{\substack{(\xi,\eta) \in \XSet \times \overline{V}; \\ \xi \neq 0}} \frac{|Q(\xi,\eta)|}{|\xi|} < \infty  \right\}
 \]
 is a Banach space with norm
 \[
 \| Q \|_{\mathcal{E}} := \sup \left\{ \frac{|Q(\xi,\eta)|}{|\xi|} : (\xi,\eta) \in \XSet \times \overline{V}, \xi \neq 0   \right\}.
 \]
\end{lemma}

\begin{proof} 
We let $(Q_n)_{n \in \mathbb{N}}$ be a Cauchy sequence in $\mathcal{E}^{0}$. For each $(\xi,\eta) \in \XSet \times \overline{V}$ we have then that  $(Q_n(\xi,\eta))_{n \in \mathbb{N}}$ is a Cauchy Sequence in $\overline{V}$, and therefore has a limit in $\overline{V}$ that we denote by $Q(\xi,\eta)$. We now need to show that $Q \in \mathcal{E}^{0}$ and that
\[
\lim_{n \rightarrow \infty} \|Q_n - Q \|_{\mathcal{E}} = 0.
\]

To show the convergence, pick $\eps > 0$ arbitrarily and $m, n$ sufficiently large such that
\[
 \|Q_n - Q_m \|_{\mathcal{E}} < \frac{\eps}{2},
\]
then we have for each $(\xi,\eta) \in \XSet \times \overline{V}$ that
\begin{align*}
|Q_n(\xi,\eta) - Q(\xi,\eta) | &\leq |Q_n(\xi,\eta) - Q_m(\xi,\eta) | + |Q_m(\xi,\eta) - Q(\xi,\eta) | \\
&\leq \frac{\eps}{2} |\xi| +  |Q_m(\xi,\eta) - Q(\xi,\eta) |.
\end{align*}
For sufficiently large $n$ we have then for all $(\xi,\eta) \in \XSet \times \overline{V}$, by letting $m \rightarrow \infty$, that
\[
|Q_n(\xi,\eta) - Q(\xi,\eta) | < \eps |\xi|,
\]
and therefore
\[
\|Q_n - Q \|_{\mathcal{E}} < \eps.
\]
This proves that 
\[
\lim_{n \rightarrow \infty} \| Q_n - Q \|_{\mathcal{E}} = 0.
\]

To show  $Q \in \mathcal{E}^{0}$, first pick $n \in \mathbb{N}$ such that $\|Q - Q_n \|_{\mathcal{E}} < 1$. Then we have
\[
\| Q \|_{\mathcal{E}} \leq \| Q-Q_n \|_{\mathcal{E}} + \|Q_n\|_{\mathcal{E}} < 1 + \|Q_n\|_{\mathcal{E}},
\]
and thereby the $\mathcal{E}$-norm of $Q$ is bounded. Note that this also implies $Q(0, \cdot) = 0$ by definition of the $\mathcal{E}$-norm. To show continuity of $Q$, observe that, for each $(\xi,\eta) \in \XSet \times \overline{V}$, for $\Delta \in \XSet \times \RSet^n$ sufficiently small we have
\begin{align*}
| Q((\xi,\eta)+ \Delta) - Q(\xi,\eta) | &\leq | Q((\xi,\eta)+ \Delta) - Q_n((\xi,\eta)+\Delta) | \\
&\qquad \quad +  | Q_n((\xi,\eta)+ \Delta) - Q_n(\xi,\eta) | +  | Q_n(\xi,\eta)  - Q(\xi,\eta)|.
\end{align*}
Now continuity of $Q_n$ for all $n \in \mathbb{N}$ and convergence of $(Q_n)_{n \in \mathbb{N}}$ in the  $\mathcal{E}$-norm to $Q$ can be used to show continuity of $Q$ via a standard argument. Namely, fix $(\xi,\eta)  \in \XSet \times \overline{V}$, and for arbitrary $\eps > 0$ pick $n \in \mathbb{N}$ sufficiently large and  $\delta > 0$ sufficiently small such that $|\Delta| < \delta$ implies  both
\[
\|Q_n - Q\|_{\mathcal{E}}  < \eps / 3, \quad | Q_n((\xi,\eta)+ \Delta) - Q_n(\xi,\eta) | < \eps / 3.
\]
Then we see that for $\delta > 0$ sufficiently small, $|\Delta| < \delta$ must imply that
\[
| Q((\xi,\eta)+ \Delta) - Q(\xi,\eta) | \leq \eps( |\xi| + 1 ).
\]
This implies continuity $Q$, and thereby the proof that $(\mathcal{E}^0, \| \blank \|_{\mathcal{E}})$ is a Banach space is completed.
\end{proof}


\section \nopunct  \label{app:D}

This appendix provides an extension of the Uniform Contraction Mapping Theorem, which we use in \cref{sec:slow-mfds-smoothness}.

\begin{theorem}\label{thm:U-Contr-Map}
Let $\XSet$, $\YSet$ and $\Lambda$ be Banach spaces, $U \subset \XSet$ and $V \subset \Lambda$ are open subsets. Suppose that $\overline{U}$ can be continuously embedded in $\YSet$ via the operator $J: \overline{U} \rightarrow \YSet$. Let $F: \overline{U} \times V \rightarrow \overline{U}$ be a uniform contraction, i.e. there exists a single constant $\kappa \in [0,1)$ such that for every $\lambda \in \Lambda$ it holds that
\[
\| F(x_2,\lambda) - F(x_1, \lambda) \|_{\XSet} \leq \kappa \|x_2 - x_1\|_{\XSet} \quad \text{for all } x_1,x_2 \in \XSet.
\]
For each $\lambda \in V$ we then denote by $g(\lambda)$ the unique fixed point of the contraction $x \mapsto F(x,\lambda)$ in $\overline{U}$. 

Then if $J \circ F: \overline{U} \times \Lambda \rightarrow \YSet$ depends continuously on $\lambda \in \Lambda$ in the $\| \blank \|_{\YSet}$-norm, the map $J \circ g: \Lambda \rightarrow \YSet$ depends continuously on $\lambda \in \Lambda$ in the $\| \blank \|_{\YSet}$-norm as well. 
\end{theorem}

\begin{proof}
This is a special case of the material in Section 4 of \cite{VaVG87}.
\end{proof}

\end{document}